\date{} 
\title{Liouville quantum gravity spheres as \\ matings of finite-diameter trees}
\author{Jason Miller and Scott Sheffield}
\documentclass[12pt,naturalnames]{article}
\usepackage{amsmath}
\usepackage{amssymb}
\usepackage{amsthm}
\usepackage{amsfonts}
\usepackage{mathrsfs}
\usepackage{graphicx}
\usepackage{color}
\usepackage{tabularx}
\usepackage{subfig}
\usepackage{enumerate}
\usepackage{comment}
\usepackage{microtype}

\usepackage[usenames,dvipsnames]{xcolor}
\usepackage[pdftitle={Liouville quantum gravity spheres as matings of finite-diameter trees},
  pdfauthor={Jason Miller and Scott Sheffield},
colorlinks=true,linkcolor=NavyBlue,urlcolor=RoyalBlue,citecolor=PineGreen,bookmarks=true,bookmarksopen=true,bookmarksopenlevel=2,unicode=true,linktocpage]{hyperref}

\def\@rst #1 #2other{#1}
\newcommand\MR[1]{\relax\ifhmode\unskip\spacefactor3000 \space\fi
  \MRhref{\expandafter\@rst #1 other}{#1}}
\newcommand{\MRhref}[2]{\href{http://www.ams.org/mathscinet-getitem?mr=#1}{MR#2}}

\usepackage[margin=1.2in]{geometry}

\newcommand{\one}{{\mathbf 1}}

\newcommand{\CF}{{\mathcal F}}
\newcommand{\CS}{{\mathcal S}}

\allowdisplaybreaks

\newif\ifdraft
\drafttrue

\numberwithin{equation}{section}
\numberwithin{figure}{section}

\newtheorem{theorem}{Theorem}
\numberwithin{theorem}{section}

\newtheorem{lemma}[theorem]{Lemma}

\newtheorem{proposition}[theorem]{Proposition}

\theoremstyle{remark}
\theoremstyle{remark}\newtheorem{remark}[theorem]{Remark}

\newcommand{\R}{\mathbf{R}}

\renewcommand{\C}{\mathbf{C}}
\newcommand{\D}{\mathbf{D}}
\newcommand{\Z}{\mathbf{Z}}

\newcommand{\N}{\mathbf{N}}
\newcommand{\W}{\mathcal{W}}
\newcommand{\HH}{\mathbf{H}}
\newcommand{\h}{\HH}

\definecolor{purple}{rgb}{0.7,0,0.7}
\definecolor{gray}{rgb}{0.6,0.6,0.6}
\definecolor{dgreen}{rgb}{0.0,0.4,0.0}
\definecolor{dblue}{rgb}{0.0,0.0,0.5}

\newcommand{\re}{{\mathrm {Re}}}
\newcommand{\im}{{\mathrm {Im}}}
\newcommand{\Fh}{{\mathfrak h}}
\newcommand{\CC}{{\mathcal C}}
\newcommand{\ol}{\overline}

\newcommand{\wh}{\widehat}
\newcommand{\wt}{\widetilde}

\newcommand{\CD}{{\mathcal D}}

\newcommand{\CE}{{\mathcal E}}
\newcommand{\CI}{{\mathcal I}}

\newcommand{\giv}{\,|\,}
\newcommand{\bes}{\mathrm{BES}}

\def\diam{\mathop{\mathrm{diam}}}

\def\dist{\mathop{\mathrm{dist}}}

\def\Re{{\rm Re}\,}

\newcommand{\SLE}{{\rm SLE}}
\newcommand{\QLE}{{\rm QLE}}
\newcommand{\CLE}{{\rm CLE}}
\newcommand{\CA}{{\mathcal A}}

\newcommand{\CH}{{\mathcal H}}

\newcommand{\CU}{{\mathcal U}}
\newcommand{\CT}{{\mathcal T}}

\newcommand{\CZ}{{\mathcal Z}}

\newcommand{\strip}{{\mathscr{S}}}
\newcommand{\cyl}{\mathscr{C}}
\newcommand{\ppp}{p.p.p.}

\newcommand{\p}{{\mathbf P}}
\newcommand{\pr}[1]{\p\!\left[#1\right]}

\newcommand{\var}[1]{\mathrm{var}\!\left(#1\right)}
\newcommand{\cov}[2]{\mathrm{cov}\!\left(#1,#2\right)}
\newcommand{\prstart}[2]{\p^{#1}\!\left[#2\right]}

\newcommand{\SM}{\mathsf M}
\newcommand{\SN}{\mathsf N}
\newcommand{\bdisk}{\SM}
\newcommand{\bsphere}{\SM_{\rm BES}}

\newcommand{\lsphere}{\SM_{\mathrm {LEV}}}
\newcommand{\lexcursion}{\SN}

\newcommand{\conelaw}{{\mathsf m}}

\newcommand{\Msone}{\SM_{\mathrm{SPH}}^1}
\newcommand{\Mstwo}{\SM_{\mathrm{SPH}}^2}

\newcommand{\Msk}{\SM_{\mathrm{SPH}}^k}

\newcommand{\musone}{\mu_{\mathrm{SPH}}^1}
\newcommand{\mustwo}{\mu_{\mathrm{SPH}}^2}

\newcommand{\musk}{\mu_{\mathrm{SPH}}^k}

\def\Ito/{It\^o}

\def \P {{\bf P}}

\def \p {{\P}}
\def \E {{\bf E}}

\begin{document} \maketitle

\begin{abstract}
We show that the unit area Liouville quantum gravity sphere can be constructed in two equivalent ways.  The first, which was introduced by the authors and Duplantier in  \cite{dms2014mating}, uses a Bessel excursion measure to produce a Gaussian free field variant on the cylinder. The second uses a correlated Brownian loop and a ``mating of trees'' to produce a Liouville quantum gravity sphere decorated by a space-filling path.

In the special case that $\gamma=\sqrt{8/3}$, we present a third equivalent construction, which uses the excursion measure of a $3/2$-stable L\'evy process (with only upward jumps) to produce a pair of trees of quantum disks that can be mated to produce a sphere decorated by $\SLE_6$.   This construction is relevant to a program for showing that the $\gamma=\sqrt{8/3}$ Liouville quantum gravity sphere is equivalent to the Brownian map.
\end{abstract}
\newpage
\tableofcontents
\newpage

\parindent 0 pt
\setlength{\parskip}{0.25cm plus1mm minus1mm}

\medbreak {\noindent\bf Acknowledgements.} We have benefited from conversations about this work with many people, a partial list of whom includes Omer Angel, Itai Benjamini, Nicolas Curien, Hugo Duminil-Copin, Amir Dembo, Bertrand Duplantier, Ewain Gwynne, Nina Holden, Jean-Fran{\c{c}}ois Le Gall, Gregory Miermont, R\'emi Rhodes, Steffen Rohde, Oded Schramm, Stanislav Smirnov, Xin Sun, Vincent Vargas, Menglu Wang, Samuel Watson, Wendelin Werner, David Wilson, and Hao Wu.  We also thank an anonymous referee for a number of helpful comments which led to many improvements to this article.

We would also like to thank the Isaac Newton Institute (INI) for Mathematical Sciences, Cambridge, for its support and hospitality during the program on Random Geometry where part of this work was completed.  J.M.'s work was also partially supported by DMS-1204894 and J.M.\ thanks Institut Henri Poincar\'e for support as a holder of the Poincar\'e chair, during which part of this work was completed.  S.S.'s work was also partially supported by DMS-1209044, a fellowship from the Simons Foundation, and EPSRC grants {EP/L018896/1} and {EP/I03372X/1}.

\section{Introduction}
\label{sec::introduction}

\subsection{Overview}

Suppose that $h$ is an instance of the Gaussian free field (GFF) on a planar domain $D$ and $\gamma \in [0,2)$ is fixed.  Then the $\gamma$-Liouville quantum gravity (LQG) surface associated with $h$ is described by the measure $\mu_h$ which is formally given by $e^{\gamma h(z)} dz$ where $dz$ denotes Lebesgue measure on $D$.  Since the GFF $h$ does not take values at points, one has to regularize in some way to make this definition precise. Let $h_\epsilon(z)$ be the average of $h$ on $\partial B(z,\epsilon)$, a quantity that is a.s.\ well defined for each $\epsilon > 0$ and $z \in D$ such that $B(z,\epsilon) \subseteq D$ \cite[Section~3]{ds08}.  The process $(z,\epsilon) \mapsto h_\epsilon(z)$ is jointly continuous in $(z,\epsilon)$ and one can define $e^{\gamma h(z)} dz$ to be the weak limit as $\epsilon \to 0$ along negative powers of $2$ of $\epsilon^{\gamma^2/2} e^{\gamma h_\epsilon(z)} dz$ \cite{ds08}.  We will often write $\mu_h$ for the measure $e^{\gamma h(z)} dz$.  LQG surfaces have also been constructed and analyzed for $\gamma > 2$ \cite{ds2009qg_prl,bjrv2013super_critical,dms2014mating} and for $\gamma=2$ \cite{DRSV1,MR3215583} but this paper will only be concerned with the case that $\gamma \in [0,2)$.

The regularization procedure used to construct $\mu_h$ leads to the following change of coordinates formula \cite[Proposition~2.1]{ds08}.  Suppose that $D,\wt{D}$ are planar domains and $\varphi \colon D \to \wt{D}$ is a conformal map.  If $\wt{h}$ is a GFF on $\wt{D}$ and
\begin{equation}
\label{eqn::coord_change}
h = \wt{h} \circ \varphi + Q \log|\varphi'| \quad\text{where}\quad Q = \frac{2}{\gamma} + \frac{\gamma}{2},
\end{equation}
then
\begin{equation}
\label{eqn::measures_equivalent}
\mu_h(A) = \mu_{\wt{h}}(\varphi(A))
\end{equation}
for all Borel sets $A$.  This allows us to define an equivalence relation on pairs $(D,h)$ by declaring $(D,h)$ and $(\wt{D},\wt{h})$ to be equivalent if $h$ and $\wt{h}$ are related as in~\eqref{eqn::coord_change}.  An equivalence class of such $(D,h)$ is then referred to as a {\em quantum surface} \cite{ds08}.  A representative $(D,h)$ of such an equivalence class is referred to as an \emph{embedding} of the quantum surface.  In many situations, it is natural to consider quantum surfaces with one or more marked points or paths.  In this case, the equivalence relation is defined in the same way except we require in addition that the conformal map in~\eqref{eqn::coord_change} takes the marked points and paths associated with the first surface to the corresponding marked points and paths associated with the second surface.

As described above, LQG has a number of variants because the GFF has a number of variants (e.g., free boundary, fixed boundary, free boundary plus a harmonic function). And this brings us to a natural question: what is the {\em right} way to define Liouville quantum gravity on the sphere, which has no boundary? More precisely, how do we describe the object that one would expect to see as the scaling limit of the most natural discrete random planar map models on the sphere?  The most obvious answer (sample the ordinary GFF on the sphere --- which is defined modulo additive constant --- and adjust the constant {\em a posteriori} to make the total $\mu_h$ area $1$) appears to be wrong.

In fact, this problem is more subtle than one might initially guess. It turns out that there are a number of ways to describe the right answer mathematically, but they all require at least a page or two of text to properly motivate and explain. One somewhat less explicit approach is to describe the answer using limits,\footnote{Short version: one defines a fixed boundary GFF $h$ on a fixed domain $D$, conditions on $\mu_h(D) = C $, rescales to make $\mu_h(D) = 1$, and then considers the $C \to \infty$ of the resulting law on surfaces.} an idea suggested and briefly sketched in \cite{SHE_WELD}. Another more explicit construction, appearing in \cite{dms2014mating}, uses the cylinder as a parameter space (mapping one ``quantum typical'' point to each of its two endpoints), and makes use of a reparameterized Bessel excursion measure to describe the averages of $h$ on cylinder slices. It is shown in \cite{dms2014mating} that the limit suggested in \cite{SHE_WELD} is well-defined and equivalent to the object constructed in \cite{dms2014mating}.

A third approach, presented by David, Kupiainen, Rhodes, and Vargas in~\cite{lqg_sphere}, uses the complex plane $\C$ as a parameter space, fixes the location of three ``quantum typical points,'' and describes the law in terms of an integral over the space of possible averages of the field $h$ w.r.t.\ a fixed background measure.  Although it is not obvious from their construction, a work of Aru, Huang, and Sun \cite{twoperspectives} shows that the construction of \cite{lqg_sphere} is equivalent to the ones we mentioned above. We note that \cite{lqg_sphere} (see also \cite{hrvdisk, drvtorus}) closely follows similar constructions that appeared in the physics literature some decades ago; it also surveys and recovers a number of explicit calculations from that literature, which is quite extensive and which prefigures much of the recent mathematical work in this area.\footnote{Both \cite{dms2014mating} and \cite{lqg_sphere} also discuss ``non-unit-area'' LQG spheres. One way to describe a general quantum surface with finite area is via a pair $(\CS,A)$ where $A$ is the total area, and $\CS$ is the unit area surface obtained by ``rescaling'' the original (i.e., by adding a constant to $h$ to make the total $\mu_h$ mass $1$).  If $d \CS$ is a measure on unit area quantum spheres, then for any measurable function $f:\R_+ \to \R_+$, there is a measure $d\CS \otimes f(A) dA$ on $(\CS,A)$ pairs. In \cite{dms2014mating} this $f$ is a taken to be a power of $A$.  In \cite{lqg_sphere}, it is a power of $A$ times $e^{-\mu A}$, where $\mu$ is the so-called cosmological constant. In both papers, it turns out to be sometimes easier to first construct the non-constant-area measure, and then obtain the unit area measure as the conditional law of $\CS$ once $A$ is given.} We will not attempt to survey the physics literature here.

In \cite{dms2014mating}, the authors along with Duplantier explain how to construct and interpret infinite-volume LQG surfaces as conformal matings of pairs of trees. Along the way,  \cite{dms2014mating} develops a number of connections between different types of LQG surfaces and random curves related to the Schramm-Loewner evolution ($\SLE$) \cite{S0}. These results build on the imaginary geometry theory derived in \cite{MS_IMAG,MS_IMAG2,MS_IMAG3,MS_IMAG4} and the conformal welding theory derived in \cite{SHE_WELD}. The goal of the present article is to extend these results and connections to the unit area quantum sphere described in \cite{dms2014mating}.

Let us stress however that this is far from a straightforward
extension of \cite{dms2014mating}, and that the work in this paper is very different in character from what appears in \cite{dms2014mating}. As we outline in Section~\ref{subsec::overview}, the main work in the current paper involves making
sense of various types of ``bottleneck conditioning,'' which are
conceptually clean but technically quite subtle.

We also remark that it remains an important open problem to establish higher genus analogs of the statements in this paper (where the sphere is replaced by a genus $g$ torus).  Figure~\ref{fig:dependencies} illustrates how the results of the present paper fit into the existing literature, and the higher genus analogs of most of the boxes and arrows shown there have not yet been established. Exceptions include the Brownian map box (work in preparation by Bettinelli and Miermont \cite{bm_compact2}) and the triple-fixed-point construction box (work of Guillarmou, Rhodes and Vargas \cite{guillarmourhodesvargas}, which builds on physics literature constructions).

\subsection{Scaling limit motivation}

\begin{figure}[ht!]
\begin{center}
\includegraphics[scale=0.85]{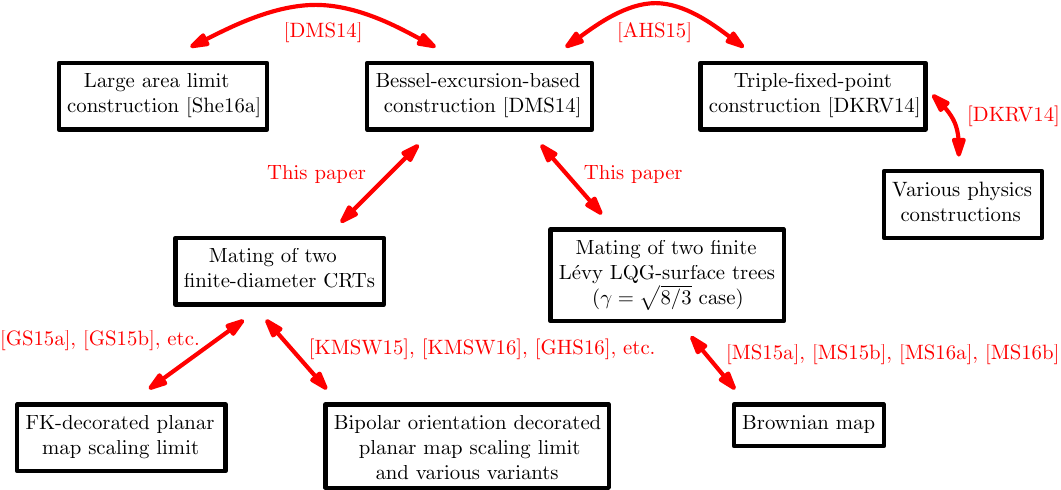}
\end{center}
\caption{\label{fig:dependencies} Several constructions of the Liouville quantum gravity sphere and their relationships. The boxes represent equivalent LQG-sphere definitions.}
\end{figure}

The unit area quantum sphere is significant in part because when $\gamma^2 \in [2,4)$ it has been conjectured to be the scaling limit of the FK-weighted random planar map on the sphere, as discussed for example in \cite[Section~4.2]{2011arXiv1108.2241S}. But why do we expect this conjecture to be true? In other words, how do we know that the LQG sphere definition (in any of its equivalent forms) is the right one for the purpose of understanding FK-scaling limits? There are various ways to answer this question, but our strongest answer is that there is one version of this conjecture that has actually been proved in work by Gwynne and Sun \cite{finitevolumeestimates,finitevolumelimit}, as indicated by one of the arrows in Figure~\ref{fig:dependencies}.

As explained in \cite{2011arXiv1108.2241S}, one may encode a loop-decorated quadrangulation of the sphere by a spanning tree and dual tree pair, which are in turn encoded by a walk on $\Z_+^2$.  The two coordinate functions in this walk are the contour functions of the trees, and the loop-decorated map is viewed as a gluing of this pair of trees along a space-filling path. The work in \cite{finitevolumeestimates,finitevolumelimit} establishes a precise form for the scaling limit of this pair of trees: it is a particular pair of correlated finite-diameter trees, each closely related to the continuum random tree (CRT).

In this paper, we show in Theorem~\ref{thm::sphere_equivalent_constructions} that the unit area quantum sphere constructed in \cite{dms2014mating} can be understood as a ``conformal mating'' of the same pair of trees. Thus, together with  \cite{finitevolumeestimates,finitevolumelimit}, this implies that the FK-decorated random planar maps converge to $\CLE$-decorated LQG spheres in a topology where two decorated spheres are considered close when their corresponding trees are close. Although the topology may not be the first that would come to mind when formulating a scaling limit conjecture, this is already an ``honest'' scaling limit result in the sense that in both the discrete and continuum settings, the pair of trees encodes the entire structure of the surface. The infinite volume version of this story is developed and explained in much more detail in \cite{2011arXiv1108.2241S,dms2014mating}.

A work in progress is \cite{strongertopology}, which builds on \cite{finitevolumeestimates,finitevolumelimit} in order to strengthen this topology of convergence.  The work in \cite{strongertopology} aims to show that the entire discrete loop structure (lengths of loops, areas of regions surrounded by loops, locations along loops---as measured by loop length---where self-intersections and intersections with other loops appear) converges to the analogous continuum loop structure. This ``continuum loop structure'' is a countable collection of measure-endowed loops, together with a set of intersection points and a planar embedding defined up to homeomorphism. It is essentially the object one gets by looking at a $\CLE_{\kappa'}$-decorated $\gamma$-LQG sphere ($\kappa' = 16/\gamma^2$) and remembering all the loop lengths and intersection points but ``forgetting'' how the whole structure is conformally embedded. It is shown in \cite{strongertopology} that the continuum loop structure a.s.\ uniquely determines the embedding, so this forgetting involves no actual loss of information.

A similar convergence result for bipolar oriented maps of the sphere towards the $\sqrt{4/3}$-LQG sphere decorated by $\SLE_{12}$ has been established in \cite{kmsw2015bipolar} and extended in \cite{ghs2016bipolar}.  (See also \cite{kmsw2016bipolar-lattice} for more on bipolar orientations on a planar lattice.)

The works mentioned above do not address the (still open) question of whether the natural conformal embeddings of the discrete models in the sphere (circle parkings, square tilings, Riemannian uniformizations of glued-together unit squares, etc.) approximate the analogous embeddings of the continuum models.  (See \cite{gms2017tutte} for a statement of this type for a random planar map model defined out of the mating of trees construction of LQG.  See also \cite{gms2018tutte} for a statement of this type for the adjacency graph formed by the cells in the Poisson-Voronoi tessellation of a Brownian surface.)

\subsection{Brownian map motivation}

In the special case $\gamma = \sqrt{8/3}$, the current work will also play an important role in a program announced by the authors in \cite{ms2013qle} to use the so-called $\QLE(8/3,0)$ to put a metric on $\sqrt{8/3}$-Liouville quantum gravity, and to show that the resulting metric space agrees in law with the Brownian map, the scaling limit of uniformly random maps on the sphere \cite{legalluniqueanduniversal,miermontlimit}.  Indeed, the current paper shows that a whole-plane $\SLE_6$ drawn on top of an independent $\sqrt{8/3}$-LQG sphere satisfies the correct symmetries so that we can make sense of a form of $\QLE(8/3,0)$ on the $\sqrt{8/3}$-LQG sphere.  In particular, we will show that:
\begin{itemize}
\item The holes cut out by an $\SLE_6$ are given by a Poissonian collection of quantum disks,
\item The law of the region which contains the target point of an $\SLE_6$ is equal to that of a quantum disk weighted by its quantum area, and
\item The law of the tip of an $\SLE_6$ is distributed according to the quantum length measure on the boundary of the unexplored region.	
\end{itemize}
The final point mentioned above implies that the reshuffling operation introduced in \cite{ms2013qle} when applied in the present setting has the interpretation of being a continuum analog of the Eden model on a $\sqrt{8/3}$-LQG sphere.  These three properties will in fact be critical in \cite{qlebm}, in which the metric for $\sqrt{8/3}$-LQG is constructed.

The rest of the program for connecting the Brownian map and the $\sqrt{8/3}$-LQG sphere is carried out in \cite{map_making,qlebm, qle_continuity,qle_determined}.  We refer the reader to the introduction of \cite{qlebm} for an overview of how the different articles fit together.

\subsection{Main results}
\label{subsec::main_results}
When stating our results (and throughout the paper) we assume that $\gamma \in (0,2)$ and
\begin{equation}
\label{eqn::kappa_kappa_prime}
 \kappa = \gamma^2 \in (0,4), \quad \kappa' = \frac{16}{\kappa} = \frac{16}{\gamma^2} \in (4,\infty), \quad\text{and}\quad Q = \frac{2}{\gamma} + \frac{\gamma}{2} \in (2,\infty).
\end{equation}
We also assume that the reader is familiar with the definitions of
\begin{itemize}
\item LQG surfaces for $\gamma \in (0,2)$ (briefly described above; see also \cite{ds08}) and the GFF \cite{SHE06}.
\item $\SLE$ and $\SLE_\kappa(\rho)$ processes.  See \cite{S0,LAW05,W03} for more on $\SLE$ and \cite[Section~8.3]{LSW_RESTRICTION} as well as \cite{SCHRAMM_WILSON} and the preliminaries sections of \cite{MS_IMAG,MS_IMAG2,MS_IMAG3,MS_IMAG4} for more on $\SLE_\kappa(\rho)$ processes.
\item Space-filling $\SLE$.  See the introduction of \cite{MS_IMAG4}.  (We will provide an additional review in Section~\ref{subsec::space_filling_sle}.)
\end{itemize}
We further assume that the reader is familiar with the tree-mating constructions as they are described in the introduction of \cite{dms2014mating}.  That is, since this paper is in some sense a follow up to \cite{dms2014mating}, we will not replicate the introduction here.  It is, however, not necessary for the reader to have digested all of \cite{dms2014mating} in order to understand the present paper.  For the convenience of the reader and to set notation we will recall below some constructions that are used repeatedly.

First let us briefly recall the construction of the unit area quantum sphere given in \cite{dms2014mating}.  In order to do so, we first need to introduce some Hilbert spaces.  For functions $f,g \colon \C \to \R$ with $L^2$ gradients, we define their Dirichlet inner product to be
\begin{equation}
\label{eqn::dirichlet}
(f,g)_\nabla = \frac{1}{2\pi} \int_D \nabla f(x) \cdot \nabla g(x) dx.
\end{equation}
For a domain $D \subseteq \C$, we let $H(D)$ be the Hilbert space closure of the subspace of $C^\infty(D)$ with $L^2$ gradients with respect to $(\cdot,\cdot)_\nabla$.  Let $\cyl = \R \times [0,2\pi]$ denote the infinite cylinder with the lines $\R$ and $\R + 2\pi i$ identified.  We then let $\CH_1(\cyl)$ (resp.\ $\CH_2(\cyl)$) denote the subspace of $H(\cyl)$ given by those functions which are constant on vertical lines (resp.\ have mean-zero on vertical lines).  Then $\CH_1(\cyl) \oplus \CH_2(\cyl)$ gives a $(\cdot,\cdot)_\nabla$-orthogonal decomposition of $H(\cyl)$; see \cite[Lemma~4.3]{dms2014mating}.

The starting point for the construction of the unit area quantum sphere is a certain infinite measure $\bsphere$ on doubly-marked quantum surfaces $(\cyl,h,-\infty,+\infty)$.  To describe the measure, we assert that one may sample from the measure via the following steps:
\begin{itemize}
\item Take the projection of $h$ onto $\CH_1(\cyl)$ to be given by the process $\tfrac{2}{\gamma} \log Z$ reparameterized to have quadratic variation $du$ where $Z$ is picked from the It\^o excursion measure $\nu_\delta^\bes$ of a Bessel process of dimension $\delta = 4-\tfrac{8}{\gamma^2}$.  (We review the construction of $\nu_\delta^\bes$ in Section~\ref{subsubsec::bessel}.  Even though $\delta \leq 0$ for $\gamma \in (0,\sqrt{2}]$, $\nu_\delta^\bes$ still makes sense.)
\item Sample the projection of $h$ onto $\CH_2(\cyl)$ independently from the law of the corresponding projection of a whole-plane GFF on $\cyl$.
\end{itemize}
Since $\nu_\delta^\bes$ is an infinite measure, so is $\bsphere$.  However, if one conditions on the quantum area of $\bsphere$ being a particular positive and finite value, then the conditional law is a well-defined probability measure.\footnote{This point is justified carefully in \cite[Section~4]{dms2014mating} for the quantum disk measures that describe the pieces of so-called thin quantum wedges. The quantum disk measures in \cite[Section~4]{dms2014mating} are constructed from Bessel excursions the same way as $\bsphere$ but with a different range of choices for the parameter $\delta$. The conditioning argument used in  \cite[Section~4]{dms2014mating} also applies to $\bsphere$.} The {\bf unit area quantum sphere} is the measure which is given by sampling $(\cyl,h,-\infty,+\infty)$ as above conditioned on having unit quantum area.

As mentioned earlier, it is also shown in \cite{dms2014mating} that the law of the unit area quantum sphere can be constructed using the limiting procedure suggested in \cite{2011arXiv1108.2241S}.  It is shown in \cite[Proposition~A.13]{dms2014mating}, which follows from the limiting construction, that the points which correspond to $\pm \infty$ conditionally on $(\cyl,h)$ as a quantum surface are uniformly and independently distributed according to $\mu_h$.  That is, the law of the field $h$ (modulo a horizontal translation and global rotation about $\pm \infty$) is invariant under the operation of picking $x,y \in \cyl$ independently from $\mu_h$, letting $\varphi \colon \cyl \to \cyl$ be a conformal transformation with $\varphi(+\infty) = x$, $\varphi(-\infty) = y$, and then replacing $h$ with the field $h \circ \varphi + Q \log|\varphi'|$.

The infinite volume companion of the unit area quantum sphere is the so-called $\gamma$-quantum cone described in \cite{SHE_WELD,dms2014mating}.  Just as in the case of the former, the latter can also be constructed using Bessel processes; we will recall its construction in Section~\ref{subsec::surfaces} as it will play an important role in this paper.  Two of the main results of \cite{dms2014mating}, namely \cite[Theorem~1.9 and Theorem~1.11]{dms2014mating}, give that a $\gamma$-quantum cone can be constructed and described entirely in terms of a certain (correlated) two-dimensional Brownian motion.  In particular, if $(\cyl,h,+\infty,-\infty)$ is a $\gamma$-quantum cone with $\gamma \in [\sqrt{2},2)$ where $-\infty$ (resp.\ $+\infty$) is the marked point about which neighborhoods have infinite (resp.\ finite) mass and $\eta'$ is a space-filling $\SLE_{\kappa'}$ process \cite{MS_IMAG4} from $-\infty$ to $-\infty$ sampled independently of $h$ and then reparameterized so that $\mu_h(\eta'([s,t])) = t-s$ for all $s < t$ then the change in the quantum lengths $(L,R)$ of the left and right boundaries of $\eta'$ relative to time $0$ evolve as a correlated two-dimensional Brownian motion with (up to a linear reparameterization of time) \footnote{Space-filling $\SLE_{\kappa'}$ from~$\infty$ to $\infty$ in $\C$ is constructed in \cite{MS_IMAG4}.  The process in $\cyl$ from $-\infty$ to $-\infty$ is defined in the same way with a whole-plane GFF on $\cyl$.  Alternatively, it can be constructed by taking the process on $\C$ and then applying a conformal transformation which takes $\infty$ to $-\infty$.}
\begin{equation}
\label{eqn::l_r_cov}
\var{L_t} = |t|, \quad \var{R_t} = |t|, \quad\text{and}\quad \cov{L_t}{R_t} = - \cos(\pi \gamma^2/4) |t| \geq 0.
\end{equation}
Moreover, $(L,R)$ a.s.\ determines both $\eta'$ and the quantum surface $(\cyl,h,+\infty,-\infty)$.  It is in fact shown in \cite{dms2014mating} that the quantum lengths $(L,R)$ of $\eta'$ evolve as a two-dimensional Brownian motion for all values of $\gamma \in (0,2)$ and that $(L,R)$ a.s.\ determines the path decorated quantum surface $(\cyl,h,+\infty,-\infty)$, $\eta'$.  It was later shown in \cite{ghms2015correlation} that the covariance matrix for $(L,R)$ is the same function of $\gamma$ as in~\eqref{eqn::l_r_cov} for all $\gamma \in (0,2)$.

Our first main result (stated just below) serves to extend this result to the setting of the unit area quantum sphere.  In this setting, the pair $(L,R)$ is no longer a correlated Brownian motion but rather a correlated Brownian loop.  Fix $\gamma \in (0,2)$ and suppose that $(X,Y)$ is a two-dimensional Brownian motion starting from the origin with the same covariance as in~\eqref{eqn::l_r_cov}.  Let $(L,R)$ be given by the law of $(X,Y)$ conditioned on $X_1 = Y_1 = 0$ and $X_t, Y_t \geq 0$ for all $t \in [0,1]$. (This involves conditioning on an event of measure zero; we explain how to make this precise in Section~\ref{sec::brownian_excursions}.)

\begin{theorem}
\label{thm::sphere_equivalent_constructions}
Suppose that $\gamma \in (0,2)$ and that $(\cyl,h,-\infty,+\infty)$ is a unit area quantum sphere.  Let $\eta'$ be a space-filling $\SLE_{\kappa'}$ process from $-\infty$ to $-\infty$ sampled independently of $h$ and then reparameterized by quantum area.  That is, we take $\eta'(0) = -\infty$ and we parameterize time so that for $0 \leq s < t \leq 1$ we have that $\mu_h(\eta'([s,t])) = t-s$.  Let $L_t$ (resp.\ $R_t$) denote the quantum length of the left (resp.\ right) side of $\eta'([0,t])$.  Then the law of $(L,R)$ is as described just above.  Moreover, the path-decorated quantum surface $(\cyl,h,-\infty,+\infty)$, $\eta'$ is a.s.\ determined by $(L,R)$.
\end{theorem}

In \cite[Theorem~1.17]{dms2014mating}, it is shown that it is also natural to explore a $\gamma$-quantum cone $(\C,h,0,\infty)$ with a whole-plane $\SLE_{\kappa'}(\kappa'-6)$ process $\wt{\eta}'$ from $0$ to $\infty$. The quantum surfaces which correspond to the complementary components of $\C \setminus \wt{\eta}'$ are described by so-called forested lines.  In the case that $\kappa'=6$ so that $\wt{\eta}'$ is an ordinary whole-plane $\SLE_{6}$, it is shown in \cite[Corollary~10.2]{dms2014mating} that the evolution of the quantum length of the boundary of the unbounded component of $\C \setminus \wt{\eta}'([0,t])$ is given by a totally asymmetric $3/2$-stable L\'evy process with only negative jumps conditioned to be non-negative when $\wt{\eta}'$ is parameterized by quantum natural time (quantum natural time is introduced just before the statement of \cite[Theorem~1.18]{dms2014mating}).  In our next theorem, we describe the analog of this latter statement in the case of a unit area quantum sphere with $\gamma=\sqrt{8/3}$.

Suppose that $X_t$ is a totally asymmetric $3/2$-stable process with only upward jumps and let $I_t = \inf\{ X_s : s \leq t\}$ be the running infimum of $X$.  Let $\lexcursion$ be the excursion measure associated with the excursions that $X_t - I_t$ makes from $0$.  We note that $\lexcursion$ is an infinite measure (we will recall its construction in Section~\ref{subsubsec::stable_levy}; see also \cite[Chapter~VIII.4]{bertoin96levy}), though for each $\epsilon > 0$ we have that $\lexcursion$ assigns finite mass to those excursions which have length at least $\epsilon$. Let $\lsphere$ be the infinite measure on collections of oriented, marked quantum disks such that sampling from $\lsphere$ amounts to:
\begin{itemize}
\item Sampling a stable L\'evy excursion $e$ from $\lexcursion$.
\item Given $e$, sampling a collection of conditionally independent quantum disks (we will review the definition of a quantum disk in Section~\ref{subsubsec::quantum_disks}) indexed by the jumps of $e$ whose boundary length is equal to the size of the jump made by $e$ and then orienting each by the toss of an i.i.d.\ fair coin.
\item Marking the boundary of each quantum disk with a conditionally independent point chosen from its quantum boundary measure.
\end{itemize}

\begin{theorem}
\label{thm::pure_sphere_equivalent_constructions}
Let $(\cyl,h,-\infty,+\infty)$ be a unit area quantum sphere with $\gamma=\sqrt{8/3}$.  Let $\wt{\eta}'$ be a whole-plane $\SLE_6$ process in $\cyl$ from $-\infty$ to $+\infty$ sampled independently of $h$ and then parameterized by quantum natural time.  For each $t \geq 0$, let $X_t$ denote the length of the boundary of the connected component of $\cyl \setminus \wt{\eta}'([0,t])$ which contains $+\infty$.  Then the joint law of $X$ and the oriented (by the order in which $\wt{\eta}'$ draws the disk boundary), marked (by the last point on the disk boundary visited by $\wt{\eta}'$) quantum disks cut out by $\wt{\eta}'$ is equal to the time-reversal of a sample produced from $\lsphere$ conditioned on the total quantum area of the quantum disks being equal to $1$.  Moreover, the path-decorated quantum surface $(\cyl,h,-\infty,+\infty)$, $\wt{\eta}'$ is a.s.\ determined by the ordered sequence of oriented, marked components cut out by $\wt{\eta}'$ viewed as quantum surfaces.
\end{theorem}

We will show in Lemma~\ref{lem::finite_mass} that for each $a > 0$, $\lsphere$ assigns finite mass to those configurations for which the sum of the area of the quantum disks is at least $a$.

\begin{remark}
\label{rem::pure_sphere_points_uniform}
By \cite[Proposition~A.13]{dms2014mating}, the points which correspond to $-\infty$ and $+\infty$ in the unit area quantum sphere as constructed above are independent and uniformly distributed according to the quantum measure conditional on the surface. In particular, Theorem~\ref{thm::pure_sphere_equivalent_constructions} applies if one starts with a unit area quantum sphere, picks points $x,y \in \CS$ independently and uniformly at random using the quantum area measure, and then lets $\wt{\eta}'$ be a whole-plane $\SLE_6$ on $\CS$ from $x$ to $y$.
\end{remark}

\begin{remark}
\label{rem::other_versions}  We expect statements analogous to Theorem~\ref{thm::pure_sphere_equivalent_constructions} to also hold for other values of $\gamma \in (\sqrt{2},2)$.  Namely, we expect it to be possible to describe a certain kind of doubly-marked quantum sphere decorated with a whole-plane $\SLE_{\kappa'}$ process in terms of a $\kappa'/4$-stable L\'evy excursion where each of the jumps correspond to conditionally independent quantum disks whose boundary length is equal to the size of the jump.  The case $\gamma=\sqrt{8/3}$ is special because the surface is given by the unit area quantum sphere.  For other values of $\gamma$, the law on spheres should be constructed in the same manner as the unit area quantum sphere except with a different Bessel process dimension.  In particular, for $\gamma \neq \sqrt{8/3}$, the starting and ending points of the $\SLE_{\kappa'}$ process are not uniformly distributed according to the quantum area measure.  It should also be possible to describe a (standard) unit area quantum sphere with $\gamma \in (\sqrt{2},2)$ decorated by an independent whole-plane $\SLE_{\kappa'}(\kappa'-6)$ process connecting two points chosen uniformly from the quantum measure in terms of finite volume analogs of the forested lines considered in \cite{dms2014mating}.  We will describe some extensions in this direction in Section~\ref{sec::extensions}.
\end{remark}

Theorem~\ref{thm::pure_sphere_equivalent_constructions} implies that $\lsphere$ can be thought of as an infinite measure on path-decorated doubly-marked quantum spheres.  More generally, Theorem~\ref{thm::pure_sphere_equivalent_constructions} implies that a sample produced from $\lsphere$ conditioned on having a given quantum area $A > 0$ has the same law as a sample produced from $\bsphere$ conditioned on the quantum area being equal to the same value $A$.  It therefore follows that the Radon-Nikodym derivative of $\lsphere$ with respect to $\bsphere$ is a function of quantum area alone.  Our final main result is the explicit identification of this function.

\begin{theorem}
\label{thm::pure_sphere_radon_nikodym_derivative}
There exists a constant $c_{\mathrm {LB}} > 0$ such that
\begin{align}
   \frac{d \lsphere}{d \bsphere} &= c_{\mathrm {LB}}. \label{eqn::cm_m_rn}
 \end{align}
\end{theorem}

Let $A$ denote the quantum area of a quantum surface $\CS$ sampled from $\lsphere$ or $\bsphere$.  For $\gamma=\sqrt{8/3}$, it turns out that the density of $A$ under $\bsphere$ with respect to Lebesgue measure is given by a constant times $A^{-3/2}$; see Proposition~\ref{prop::bessel_area_decay}.  Theorem~\ref{thm::pure_sphere_radon_nikodym_derivative} then implies that the density of $A$ under $\lsphere$ with respect to Lebesgue measure is given by a constant times $A^{-3/2}$. It is not a coincidence that this is the same exponent that one encounters in the ``grand canonical'' doubly marked Brownian map measure $\mustwo$, as discussed for example in \cite{map_making}.

The equivalence of $\lsphere$ with $\bsphere$ implies that the conditional law of the two marked points of $\lsphere$ given the underlying surface are uniformly random from the quantum measure.  For this reason, in the subsequent papers \cite{qlebm,qle_continuity,qle_determined} we will refer to this measure as $\Mstwo$.  We can generate spheres with fewer or more marked points by unweighting or weighting $\Mstwo$ by quantum area.  In particular, the infinite measure on quantum spheres $\Msone$ with only one marked point can be sampled by first picking $A$ from the infinite measure $A^{-5/2} dA$ where $dA$ denotes Lebesgue measure on $\R_+$ and then, given $A$, picking a quantum sphere with area equal to $A$, and then finally picking the marked point from the quantum area measure.  This measure is in correspondence with $\musone$ from \cite{map_making}.  More generally, we define the infinite measure $\Msk$ on quantum spheres with $k$ marked points by first picking $A$ from the measure $A^{-7/2+k} dA$, and then, given $A$, picking a quantum sphere with area equal to $A$, and then finally picking the $k$ marked points conditionally independently from the quantum area measure.  This measure is in correspondence with $\musk$ from \cite{map_making}.

\subsection{Outline}
\label{subsec::overview}

The remainder of this article is structured as follows.  We will review some preliminary facts about Bessel and stable L\'evy processes, quantum surfaces, and conformal maps in Section~\ref{sec::preliminaries}. In particular, we will recall in Theorem~\ref{thm::disk_explore} how to describe a quantum disk in terms of a correlated two-dimensional Brownian excursion; this result already appeared in \cite{dms2014mating}.
Next, in Section~\ref{sec::brownian_excursions} we will give a rigorous construction of the measure on correlated Brownian loops described just before the statement of Theorem~\ref{thm::sphere_equivalent_constructions}. This measure essentially corresponds to a correlated two-dimensional Brownian bridge, starting and ending at the origin and conditioned to stay in the positive quadrant. We construct and make some basic observations about this process. In Section~\ref{sec::sphere_from_cone} we will explain how one can construct a unit area quantum sphere from a $\gamma$-quantum cone.  We will then make use of this result in Section~\ref{sec::bessel_brownian_constructions} and Section~\ref{sec::levy_construction}, where we will respectively establish Theorem~\ref{thm::sphere_equivalent_constructions} and Theorems~\ref{thm::pure_sphere_equivalent_constructions} and~\ref{thm::pure_sphere_radon_nikodym_derivative}.  We will discuss extensions of our results for $\gamma=\sqrt{8/3}$ to general values of $\gamma \in (\sqrt{2},2)$ in Section~\ref{sec::extensions}.

All of the results in this paper build on the infinite volume constructions that appear in \cite{dms2014mating}. Intuitively, one way to get from an infinite volume quantum surface to a unit area quantum sphere is to condition the former to have a small ``bottleneck,'' so that the area to one side of the bottleneck is about $1$.  One can take a limit as the bottleneck is required to be, in some sense, smaller and smaller. For each of the different ways to describe the infinite volume surface that are shown to be equivalent in \cite{dms2014mating} (via the Bessel process, the correlated 2D Brownian motion, or the stable L\'evy process) there is natural way to define a bottleneck and to make sense of the sphere obtained in the ``small bottleneck'' limit.  The technical challenge, which the bulk of this paper is devoted to addressing, is to show that all of these different approaches actually agree in the limit.

Before addressing these challenges, we will recall that a quantum cone is a random surface with two marked points, an ``infinite mass'' point (about which every neighborhood has infinite mass) and a ``finite mass'' point (about which small neighborhoods have finite mass). We will then consider various ways to explore this random surface from the infinite mass point toward the finite mass point---either deterministically (parameterizing the surface by a cylinder and exploring the cylinder from left to right) or randomly (drawing a whole plane $\SLE$ from one endpoint to another). With each approach, we may stop the exploration when (in some sense) the boundary of the unexplored region is small and then consider the conditional law of the unexplored region---in particular, we would like to understand the law of the unexplored region conditioned on the event that its quantum area is much larger than one would expect.

This analysis will require some work, and a number of careful estimates, but there are a few tricks that make the job more pleasant than it might otherwise be. One involves using some basic conformal map estimates, similar to those that appear in \cite{dms2014mating}, to argue that certain fairly drastic local operations on a quantum surface (such as cutting out a disk with a very small quantum area and then gluing in a disk with the same boundary length but a much larger quantum area) actually have little effect on the global conformal embedding. Another involves using the ``target invariance'' of certain types of $\SLE$ along with properties of the quantum cone to show that one can sometimes partially ``forget'' the location of the finite-mass endpoint of a quantum cone --- and resample it from the LQG measure on some specified region --- without changing the overall law of the path decorated surface.

\section{Preliminaries}
\label{sec::preliminaries}

In this section, we will review some preliminary facts about random processes and quantum surfaces.  We will begin in Section~\ref{subsec::bessel_levy} with a short review of Bessel and stable L\'evy processes.  Next, in Section~\ref{subsec::space_filling_sle} we will give a review of space-filling $\SLE_{\kappa'}$ as constructed in \cite{MS_IMAG4}.  Then in Section~\ref{subsec::surfaces} we will remind the reader of the various types of quantum surfaces which were constructed in \cite{dms2014mating} and are relevant for the present article.  Finally, in Section~\ref{subsec::maps}, we will record an elementary estimate for conformal maps which will be used when we perform cutting/gluing operations on quantum surfaces.

\subsection{Bessel and stable L\'evy processes}
\label{subsec::bessel_levy}

We will now collect a few facts about Bessel and stable L\'evy processes.  For the former, we refer the reader to \cite[Chapter~XI]{RY04} for a more detailed introduction; see also \cite[Section~3.2]{dms2014mating}.  For the latter, we refer the reader to \cite{bertoin96levy}.

\subsubsection{Bessel processes}
\label{subsubsec::bessel}

A {\bf Bessel process} $X_t$ of dimension $\delta \in \R$, denoted by $\bes^\delta$, is described by the SDE
\begin{equation}
\label{eqn::bessel_evolution}
dX_t = \frac{\delta-1}{2} \cdot \frac{1}{X_t} dt + dB_t,\quad X_0 \geq 0
\end{equation}
where $B$ is a standard Brownian.  Standard results for SDEs imply that a unique strong solution to~\eqref{eqn::bessel_evolution} exists up until the first time $t$ that $X_t \leq 0$ for all $\delta \in \R$.  For $\delta \geq 2$, a $\bes^\delta$ a.s.\ does not hit $0$ (except possibly at its starting point) while for $\delta < 2$, a $\bes^\delta$ a.s.\ hits $0$ in finite time.  This can be seen by observing that $X_t^{2-\delta}$ is a continuous local martingale.  For $\delta \in (1,2)$, a $\bes^\delta$ process can be defined for all times, is instantaneously reflecting at $0$, is a semimartingale, and satisfies the integrated form of~\eqref{eqn::bessel_evolution}.  For $\delta \in (0,1]$, there is also a unique way of defining a $\bes^\delta$ for all times which is instantaneously reflecting at $0$.  In this case, the process is not a semimartingale and satisfies~\eqref{eqn::bessel_evolution} only in those intervals in which it is not hitting $0$.  In order to make sense of the integrated version of~\eqref{eqn::bessel_evolution} in this case, it is necessary to introduce a principal value correction.

For $\delta \in (0,2)$, one can use It\^o excursion theory to decompose a $\bes^\delta$ into its excursions from $0$.  In order to describe this, we let $\CE_h$ be the set of continuous functions $\phi \colon [0,h] \to \R$ and $\CE = \cup_{h > 0} \CE_h$.  We then let $\nu_\delta^\bes$ be the (infinite) measure on $\CE$ which can be sampled from by:
\begin{itemize}
\item Picking a sample $t$ from the measure $c_\delta t^{\delta/2-2} dt$ where $dt$ denotes Lebesgue measure on $\R_+$ and $c_\delta > 0$ is a constant.
\item Given $t$, picking an excursion of a $\bes^\delta$ from $0$ to $0$ of length $t$.
\end{itemize}
Note that a Bessel process can be constructed as a chain of excursions: precisely, a sample from the law of a $\bes^\delta$ can be produced by first picking a Poisson point process (\ppp)\ $\Lambda$ with intensity measure $du \otimes \nu_\delta^\bes$ on $\R_+ \otimes \CE$, where $du$ is Lebesgue measure on $\R_+$, and then concatenating together the second component of the elements $(u,e)$ of $\Lambda$ with $(u,e)$ coming before $(u',e')$ if and only if $u < u'$.

One can also generate a \ppp\ $\Lambda$ with intensity measure $du \otimes \nu_\delta^\bes$ when $\delta \leq 0$.  In this case, however, it is not possible to string together the excursions chronologically to form a continuous process.  As explained in the introduction, the excursion measure associated with a $\bes^\delta$ is the starting point for the construction of the unit area quantum sphere given in \cite{dms2014mating}.

Another way to generate a Bessel process is by exponentiating a Brownian motion with linear drift and then reparameterizing it to have quadratic variation $dt$.  Namely, if $X_t = B_t + a t$ where $B$ is a standard Brownian motion, then the process which arises by setting $Z_t = e^{X_t}$ and then changing time so that $d\langle Z \rangle_t = dt$ is a $\bes^\delta$ with $\delta = 2+2a$ (stopped at the first time that it hits $0$).  Conversely, if $Z$ is a $\bes^\delta$, then $X_t = \log Z_t$ reparameterized to have quadratic variation $d \langle X \rangle_t = dt$ is a standard Brownian motion with linear drift $at$ with $a =(\delta-2)/2$.  (See \cite[Chapter~XI]{RY04} or \cite[Proposition~3.4]{dms2014mating}.)

\subsubsection{Stable L\'evy processes}
\label{subsubsec::stable_levy}

Fix $\alpha \in (0,2)$ and recall that a L\'evy process $X$ is said to be {\bf $\alpha$-stable} if for each $u > 0$ fixed it has the property that $(t \mapsto u^{-1/\alpha} X_{u t}) \stackrel{d}{=} (X_t)$.

Suppose that $X$ is an $\alpha$-stable process with $\alpha \in (1,2)$ with only positive jumps.  Let $I_t = \inf\{X_s : s \in [0,t]\}$ be the running infimum of $X$.  Then the process $X-I$ can be decomposed into a Poissonian collection of excursions from $0$ \cite[Chapter~VIII.4]{bertoin96levy}.  Let $\lexcursion$ be the measure which is sampled from using the following steps:
\begin{itemize}
\item Pick a lifetime $t$ from the measure $c_\alpha t^{\rho-2} dt$ where $\rho = 1-1/\alpha$ is the positivity parameter of the process \cite[Chapter~VIII.1]{bertoin96levy}, $dt$ denotes Lebesgue measure on $\R_+$, and $c_\alpha > 0$ is a constant.
\item Given $t$, pick a sample from the normalized excursion measure of an $\alpha$-stable L\'evy process and then rescale it spatially and in time so that it has length $t$.
\end{itemize}

One can then produce a sample from the law of the process $X - I$ by sampling a \ppp\ $\Lambda$ with intensity measure $du \otimes d\lexcursion$, with $du$ given by Lebesgue measure on $\R_+$, and then concatenating the second component of the elements $(u,e) \in \Lambda$ where $(u,e)$ comes before $(u',e')$ if and only if $u < u'$.

The collection of jumps made by $X$ up to a given time $T$ also has a Poissonian structure.  Namely, if $\Lambda$ is a \ppp\ on $[0,T] \times \R_+$ sampled with intensity measure $dt \otimes \wt{c}_\alpha u^{-1-\alpha} du$ where $dt$ denotes Lebesgue measure on $[0,T]$, $du$ denotes Lebesgue measure on $\R_+$, and $\wt{c}_\alpha > 0$ is a constant, then the elements $(t,u)$ are in correspondence with the jumps made by $X$ up to time $T$ where $t$ gives the time at which the jump occurred and $u$ gives the size of the jump.

We finish by recording one final useful fact about $\alpha$-stable L\'evy processes with only positive jumps.  Suppose that $X_t$ is an $\alpha$-stable process with only positive jumps, $X_0 > 0$, and let $\tau = \inf\{t \geq 0 : X_t = 0\}$.  Then the time-reversal $X_{t-\tau}$ has the law of an $\alpha$-stable process with only negative jumps conditioned to be non-negative and stopped at the last time that it hits $X_0$ \cite[Chapter~VII, Theorem~18]{bertoin96levy}.

\subsection{Space-filling $\SLE$}
\label{subsec::space_filling_sle}

Fix $\kappa' > 4$.  In this section, we will recall the construction of {\bf space-filling $\SLE_{\kappa'}$} from \cite{MS_IMAG4} and also explain how it is related to chordal and radial $\SLE_{\kappa'}$.  Roughly speaking, space-filling $\SLE_{\kappa'}$ for $\kappa' \in (4,8)$ is an $\SLE$ process which iteratively fills up the bubbles as it cuts them off from $\infty$ (or its target point) and for $\kappa' > 8$ it is the same as ordinary $\SLE_{\kappa'}$.  The particular variant of space-filling $\SLE_{\kappa'}$ which is most important for this article is the version which is defined on $\C$ and is an infinite path from $\infty$ back to itself.  The starting point for its construction in \cite{MS_IMAG4} is a whole-plane GFF $h$ with values defined up to a global multiple of $2\pi \chi$ where
\[ \chi = \frac{2}{\sqrt{\kappa}} - \frac{\sqrt{\kappa}}{2} \quad\text{and}\quad \kappa = \frac{16}{\kappa'} \in (0,4).\]
It is shown in \cite{MS_IMAG4} that it is possible to construct the flow lines of the formal vector field $e^{i h / \chi}$ (this only requires us to have defined the field values modulo a global multiple of $2\pi \chi$).  For $z \in \C$, the flow line $\eta_z$ of $h$ from $z$ to $\infty$ is a whole-plane $\SLE_\kappa(2-\kappa)$ process from $z$ to $\infty$.  More generally, for $\theta \in \R$ we can define the flow line $\eta_z^\theta$ of $h$ from $z$ to $\infty$ as the flow line of $h+ \theta \chi$ and $\eta_z^\theta$ is a whole-plane $\SLE_\kappa(2-\kappa)$  process from $z$ to $\infty$.  If $\wt{\theta} = \theta + 2\pi \chi$ then $\eta_z^\theta = \eta_z^{\wt{\theta}}$ but otherwise $\eta_z^\theta$, $\eta_z^{\wt{\theta}}$ are distinct paths.  In other words, we have a $2\pi$ range of angles for flow lines starting from each $z \in \C$.  It is shown in \cite[Theorem~1.9]{MS_IMAG4} that for $z,w \in \C$ distinct we a.s.\ have that the flow lines $\eta_z$, $\eta_w$ starting from $z,w$, respectively, a.s.\ merge with each other and do not subsequently separate.  The same is also true for the $\eta_z^\theta$, $\eta_w^\theta$ for each value of $\theta$.  This means that we can use the flow lines to define a space-filling tree and a space-filling dual tree.

	More precisely, for each $z \in \C$ we let $\eta_z^L$ (resp.\ $\eta_z^R$) be the flow line of $h$ starting from $z$ with angle $\tfrac{\pi}{2}$ (resp.\ $-\tfrac{\pi}{2}$).  The paths $\eta_z^L$ are the branches of the tree and the paths $\eta_z^R$ are branches of the dual tree.  Given a countable dense set $(z_n)$ of $\C$, we can define an ordering on the $z_n$ by saying that $z_n$ comes before $z_m$ if $\eta_{z_n}^L$ merges with $\eta_{z_m}^L$ on its right side.  This turns out to be equivalent to  $\eta_{z_n}^R$ merging with $\eta_{z_m}^R$ on its left side.  Whole-plane space-filling $\SLE_{\kappa'}$ from $\infty$ to $\infty$ is a continuous, non-self-tracing and non-self-crossing path $\eta'$ which fills all of $\C$ and visits the points of $(z_n)$ according to the above order.  It is not difficult to see from the construction that the resulting path is a.s.\ the same for any two fixed choices of countable dense sets.  Space-filling $\SLE_{\kappa'}$ can thus be thought of as the peano curve which traces between the tree and dual tree defined above.

	If $z \in \C$ is fixed, then we can consider $\eta'$ targeted at $z$, which means that we parameterize the path by capacity as seen from $z$.  Call this path $\wt{\eta}_z'$.  In other words, $\wt{\eta}_z'$ does not fill in the bubbles that it disconnects from $z$.  It turns out that the law of $\wt{\eta}_z'$ is that of a whole-plane $\SLE_{\kappa'}(\kappa'-6)$ process from $\infty$ to $z$ and is the counterflow line of $h$ from $\infty$ to $z$.  For different points $z,w$, the paths $\wt{\eta}_z'$, $\wt{\eta}_w'$ agree with each other until $z$ and $w$ are separated and afterwards their evolution continues independently.

\subsection{Quantum cones and disks}
\label{subsec::surfaces}

We are now going to give a brief overview of the types of quantum surfaces which will be important for this article.  We refer the reader to \cite[Sections~1 and~4]{dms2014mating} for a much more detailed introduction and motivation for the definitions we give here.

Throughout, we let $\strip = \R \times [0,\pi]$ and $\strip_\pm = \R_\pm \times [0,\pi]$.  We also let $\cyl = \R \times [0,2\pi]$ and $\cyl_\pm = \R_\pm \times [0,2\pi]$ with the top and bottom identified in both cases.  When $X \in \{ \strip, \strip_\pm, \cyl, \cyl_\pm\}$, we let $\CH_1(X)$ be the subspace of $H(X)$ consisting of those functions which are constant on vertical lines and let $\CH_2(X)$ be the subspace of $H(X)$ consisting of those functions which have mean zero on vertical lines.  As explained in the introduction (see also \cite[Lemma~4.3]{dms2014mating}), we have that $\CH_1(X) \oplus \CH_2(X)$ gives an orthogonal decomposition of $H(X)$.

\subsubsection{Quantum cones}
\label{subsubsec::quantum_cones}

Fix $\alpha < Q$.  An $\alpha$-quantum cone $\CC = (\cyl,h,+\infty,-\infty)$ is the quantum surface whose law can be sampled from using the following steps:
\begin{itemize}
\item Take the projection of $h$ onto $\CH_1(\cyl)$ to be given by $2\gamma^{-1} \log Z$ parameterized to have quadratic variation $du$ where $Z$ is the time-reversal of a $\bes^\delta$ with $\delta = 2 + \tfrac{4}{\gamma}(Q-\alpha) > 2$ starting from $0$.  This determines the projection of $h$ onto $\CH_1(\cyl)$ up to horizontal translation; we can fix the horizontal translation by taking it so that the projection first hits $0$ at $u=0$.
\item Sample the projection of~$h$ onto $\CH_2(\cyl)$ independently from the law of the corresponding projection of a GFF on~$\cyl$.
\end{itemize}
In many instances, it is also natural to parameterize a quantum cone by~$\C$ rather than~$\cyl$.  For the purposes of this article, however, we will always parameterize our cones by~$\cyl$.

The reason that we take the time-reversal of $Z$ in place of $Z$ itself is so that every neighborhood of $-\infty$ a.s.\ contains an infinite amount of quantum area and every sufficiently small neighborhood of $+\infty$ (i.e.,\ bounded away from $-\infty$) a.s.\ contains a finite amount of quantum area; at times it will be useful to explore from the ``infinite area'' end to the ``finite area'' end of the cylinder, so it is mildly more convenient to orient time that way.

As explained in \cite{dms2014mating}, it is natural to explore a $\gamma$-quantum cone with an independent space-filling $\SLE_{\kappa'}$ process.  When parameterized by $\cyl$ as above, we take $\eta'$ to be a space-filling $\SLE_{\kappa'}$ from $-\infty$ to $-\infty$ sampled independently of $h$ and then reparameterized by $\gamma$-LQG area with time normalized so that $\eta'(0) = +\infty$.  \cite[Theorem~1.9]{dms2014mating} implies that the processes $(L,R)$ which describe the change in the quantum length of the left and right boundaries of $\eta'$ relative to time $0$ (i.e., $L_0 = R_0 = 0$) evolve as a pair of correlated Brownian motions with covariance as in~\eqref{eqn::l_r_cov} (the covariance for $\gamma \in (0,\sqrt{2})$ is identified in \cite{ghms2015correlation}) and \cite[Theorem~1.11]{dms2014mating} implies that $(L,R)$ a.s.\ determines both the quantum cone and $\eta'$, up to a rotation and translation (i.e.,\ conformal transformation of $\cyl$ which fix $\pm \infty$).

By \cite[Corollary~10.2]{dms2014mating}, it is also natural to explore an $\alpha$-quantum cone,
\begin{equation}
\label{eqn::alpha_gamma_value}
\alpha = \frac{\gamma^2+8}{4\gamma},
\end{equation}
with an independent whole-plane $\SLE_{\kappa'}$ process $\wt{\eta}'$ from $+\infty$ to $-\infty$ when $\gamma \in (\sqrt{2},2)$ so that $\kappa' \in (4,8)$.  When $\wt{\eta}'$ is parameterized by quantum natural time, the quantum length $X$ of its outer boundary evolves as a $\kappa'/4$-stable process with only downward jumps starting from $0$ and conditioned to be non-negative.  Given the realization of $X$, the regions cut out are conditionally independent quantum disks (a type of finite-volume surface described in the next subsection) whose boundary lengths are given by the jump sizes of $X$.  Moreover, by \cite[Theorem~1.17]{dms2014mating} the quantum cone is a.s.\ determined by $X$, the quantum disks, their orientation (whether or not they are surrounded on the left or right side of $\wt{\eta}'$), and the marked boundary point on each which corresponds to the first (resp.\ equivalently last) point visited by $\wt{\eta}'$.  The value $\gamma=\sqrt{8/3}$ (corresponding to $\kappa'=6$) is special because it is the unique positive solution to $(\gamma^2+8)/(4\gamma) = \gamma$.  More generally, by \cite[Theorem~1.17]{dms2014mating} it is natural to explore a $\gamma$-quantum cone with an independent whole-plane $\SLE_{\kappa'}(\kappa'-6)$ process $\wt{\eta}'$ from $-\infty$ to $+\infty$ though for $\kappa' \neq 6$ the process which gives the quantum length of the outer boundary of $\wt{\eta}'$ is more complicated to describe.

\subsubsection{Quantum disks}
\label{subsubsec::quantum_disks}

As in the case of the unit area quantum sphere described in the introduction, the starting point for the construction of the unit boundary length quantum disk is an infinite measure on quantum surfaces which is derived from the (infinite) excursion measure for a certain Bessel process.  As in \cite{dms2014mating}, we will take our quantum disks to be parameterized by $\strip$ (with ``marked'' points at the two endpoints).  A natural infinite measure $\bdisk$ on quantum disks with two marked boundary points can be sampled from by:
\begin{itemize}
\item Taking the projection $h$ onto $\CH_1(\strip)$ to be given by $2 \gamma^{-1} \log Z$ where $Z$ is sampled from the excursion measure of a Bessel process of dimension $3-\tfrac{4}{\gamma^2}$ parameterized to have quadratic variation $2du$.
\item Sampling the projection of $h$ onto $\CH_2(\strip)$ from the law of the corresponding projection of a free boundary GFF on $\strip$.
\end{itemize}
The {\bf unit boundary length quantum disk} is the law on quantum surfaces that one gets by sampling from the measure $\bdisk$ conditioned to have quantum boundary length equal to $1$.

As in the case of $\gamma$-quantum cones, it is shown in \cite{dms2014mating} that it is also natural to explore a unit boundary length quantum disk using a space-filling $\SLE_{\kappa'}$ process.  We are going to give a precise statement of this result below for the convenience of the reader.  Before we do so, we first need to remind the reader of the definition of so-called $\pi/2$-cone times and excursions.  Suppose that $Z = (X,Y)$ is a continuous process.  Then a time $t$ is said to be a {\bf $\pi/2$-cone time} for $Z$ if there exists $h > 0$ such that $Z_s \geq Z_t$ (i.e., this inequality holds coordinate-wise) for all $s \in [t,t+h]$.  We call the restriction of $Z$ to an interval of time $[s,t]$ a {\bf $\pi/2$-cone excursion} if it has the property that $Z_r \geq Z_s$ for all $r \in [s,t]$ and $X_s = X_t$ or $Y_s = Y_t$.  We refer to the quantity $\max(X_t - X_s, Y_t - Y_s)$ as the {\bf terminal displacement} of the $\pi/2$-cone excursion.  In the case that $Z$ is a two-dimensional Brownian motion, it is not difficult to see that it is possible to represent $Z$ as a Poissonian collection of $\pi/2$-cone excursions sampled using a certain infinite measure on $\pi/2$-cone excursions (this is essentially carried out in the proof of \cite[Proposition~10.3]{dms2014mating}).  We will give a direct construction of the law of a $\pi/2$-cone excursion for a correlated Brownian motion of either unit length or terminal displacement in Section~\ref{sec::brownian_excursions}.

\begin{theorem}
\label{thm::disk_explore}
Fix $\gamma \in [\sqrt{2},2)$ and suppose that $(\strip,h,-\infty,+\infty)$ is a unit boundary length quantum disk.  Let $\eta'$ be a space-filling $\SLE_{\kappa'}$ process from $-\infty$ to $-\infty$ sampled independently of $h$ and then reparameterized by quantum area.  Let $L_t$ (resp.\ $R_t$) denote the length of the left (resp.\ right) side of $\eta'([0,t])$.  Then $(L,R)$ evolves as a $\pi/2$-cone excursion with terminal displacement $1$ of a two-dimensional Brownian motion with covariance as in~\eqref{eqn::l_r_cov}.  Moreover, $(L,R)$ a.s.\ determines both~$h$ and $\eta'$,  up to a conformal transformation of $\strip$ which fixes $-\infty$.
\end{theorem}

\begin{remark}
\label{rem::disk_explore_condition_on_both_area_and_bl}
Theorem~\ref{thm::disk_explore} describes the behavior of $(L,R)$ when one conditions on the quantum boundary length of the quantum disk.  What happens when one conditions on both the quantum area and boundary length?  The process $(L,R)$ is a function of the quantum disk and the independent space-filling $\SLE_{\kappa'}$ process $\eta'$.  We know that if we fix the boundary length of the disk, then $(L,R)$ evolves as a $\pi/2$-cone excursion with terminal displacement given by the boundary length.  If we condition further on the quantum area of the disk, then this has the effect of fixing the length of the $\pi/2$-cone excursion.  Since this extra conditioning only depends on the quantum disk and not on $\eta'$, it follows that when we explore a quantum disk of a given quantum boundary length and area with an independent space-filling $\SLE_{\kappa'}$ process then the left/right boundary lengths evolve as a correlated Brownian excursion of the given length and terminal displacement.
\end{remark}

\subsection{Distortion estimate for conformal maps on the cylinder}
\label{subsec::maps}

In our proofs of Theorem~\ref{thm::sphere_equivalent_constructions} and Theorem~\ref{thm::pure_sphere_equivalent_constructions}, we will perform a number of ``cutting'' and ``gluing'' operations for quantum surfaces.  The following elementary estimates for conformal maps tell us how much these operations distort the embedding of the rest of the surface.  Recall that a set $K \subseteq \C$ is said to be a {\bf hull} if it is compact and $\wh{\C} \setminus K$ is simply connected, where $\wh{\C}$ denotes the Riemann sphere.  We begin with a restatement of \cite[Lemma~9.6]{dms2014mating}.

\begin{lemma}
\label{lem::conformal_whole_plane_hulls}
There exist constants $C_1,C_2 > 0$ such that the following is true.  Let $K_1 \subseteq \C$ be a hull of diameter at most $r$ and $K_2 \subseteq \C$ another hull such that there exists a conformal map $F \colon \C \setminus K_1 \to \C \setminus K_2$ with $|F(z)-z| \to 0$ as $z \to \infty$.  Then whenever $\dist(z,K_1) \geq C_1 r$ we have that
\[ | F(z) - z| \leq C_2 r^2 |z-b_1|^{-1}\]
where $b_1$ is the harmonic center of $K_1$.  That is, if $F_1 \colon \C \setminus \ol{\D} \to \C \setminus K_1$ is the unique conformal map fixing $\infty$ and with positive derivative at $\infty$ then $b_1$ is equal to the average of $F$ on $\partial B(0,r)$ for any $r > 1$.  (It is elementary to check that this definition does not depend on the choice of $r$.)
\end{lemma}

We say that a set $K \subseteq \cyl$ is a hull if the image of $K$ under the map $\cyl \to \C$ given by $z \mapsto e^{z}$ is a hull as defined above.  We are now going to use Lemma~\ref{lem::conformal_whole_plane_hulls} to deduce a similar estimate in the setting of hulls in $\cyl$.  As in Section~\ref{subsec::surfaces}, we let $\cyl_- = \{ z \in \cyl : \re(z) \leq 0\}$ and $\cyl_+ = \{z \in \cyl : \re(z) \geq 0\}$.

\begin{lemma}
\label{lem::conformal_cylinder_hulls}
There exist constants $C_1,C_2 > 0$ such that the following is true.  Suppose that $K_1 \subseteq \cyl_-$ is a hull and $K_2 \subseteq \cyl$ is another hull such that there exists a conformal map $F \colon \cyl \setminus K_1 \to \cyl \setminus K_2$ with $|F(z)-z| \to 0$ as $z \to +\infty$.  Then we have that
\begin{equation}
\label{eqn::f_w_w_bound}
|F(w) - w| \leq C_2\exp(-\re(w)) \quad\text{for all}\quad w \in \cyl_+ + C_1.
\end{equation}
\end{lemma}
\begin{proof}
Let $G(z) = \exp(F(\log(z)))$.  Then $G$ is a conformal transformation of $\C \setminus \wt{K}_1$ to $\C \setminus \wt{K}_2$ where $\wt{K}_i = \exp(K_i)$ for $i=1,2$ with $|G(z) -z| \to 0$ as $z \to \infty$.  Note that $\diam(\wt{K}_1) \leq 1$ since $K_1 \subseteq \cyl_-$.  Consequently, Lemma~\ref{lem::conformal_whole_plane_hulls} implies that there exist constants $\wt{C}_1,\wt{C}_2 > 0$ such that $|G(z) - z| \leq \wt{C}_2$ whenever $|z| \geq \wt{C}_1$.  Suppose that $z \in \C$ with $|z| \geq \wt{C}_1$ and let $w = \log(z) \in \cyl$.  Then this implies that
\[ |\exp(F(w)) - \exp(w)| \leq \wt{C}_2.\]
Equivalently,
\begin{equation}
\label{eqn::equiv_bound}
|\exp(F(w)-w) -1| \leq \wt{C}_2 |\exp(-w)|.
\end{equation}
By the triangle inequality, this implies that
\[ \exp(\re(F(w)-w)) \leq \wt{C}_2 \exp(-\re(w)) + 1.\]
Taking logs of both sides, we get for a constant $C_2 > 0$ that
\begin{equation}
\label{eqn::re_ubd}\re(F(w)-w) \leq \log(\wt{C}_2 \exp(-\re(w))+1) \leq C_2 \exp(-\re(w)).
\end{equation}
Similarly, we also have that
\[ 1 \leq \wt{C}_2 \exp(-\re(w)) + \exp(\re(F(w)-w))\]
which implies that
\[ 1- \wt{C}_2 \exp(-\re(w)) \leq \exp(\re(F(w)-w)).\]
This gives us that, by possibly increasing the value of $C_2$, we have
\begin{equation}
\label{eqn::re_lbd}
\re(F(w)-w) \geq -C_2 \exp(-\re(w))
\end{equation}
and therefore combining~\eqref{eqn::re_ubd} and~\eqref{eqn::re_lbd} we have
\begin{equation}
\label{eqn::re_bound}
|\re(F(w) -w)| \leq C_2 \exp(-\re(w)).
\end{equation}
Inserting~\eqref{eqn::re_bound} into~\eqref{eqn::equiv_bound}, we see that by possibly increasing the value of $C_2$ we have
\begin{equation}
\label{eqn::im_bound}
|\exp(i \im(F(w)-w))-1| \leq C_2 \exp(-\re(w)).
\end{equation}
Combining~\eqref{eqn::re_bound} with~\eqref{eqn::im_bound} implies~\eqref{eqn::f_w_w_bound} with $C_1 = \log \wt{C}_1$.
\end{proof}

\section{Correlated Brownian loops and excursions}
\label{sec::brownian_excursions}

We are now going to give a rigorous construction of the correlated Brownian loop which was described just before the statement of Theorem~\ref{thm::sphere_equivalent_constructions}.  We will at the same time give the construction of the law of a $\pi/2$-cone excursion of length $1$ and given terminal displacement.  We will then show that the correlated Brownian loop can be constructed as the limit of a $\pi/2$-cone excursion of a correlated Brownian motion of length $1$ and terminal displacement tending to $0$.  This is natural in the context of relating quantum disks and spheres.

\begin{theorem}
\label{thm::brownian_excursion_measure_exists}
Fix $\alpha \in (-1,1)$ and $(x_1,y_1) \in \partial \R_+^2$.  There exists a unique law on pairs of continuous processes $Z = (X,Y)$ defined on $[0,1]$ with $X_0 = Y_0 = 0$ and $X_1 = x_1$, $Y_1 = y_1$ such that the following hold.
\begin{enumerate}[(i)]
\item $\pr{ X_t, Y_t > 0 } = 1$ for all $t \in (0,1)$.
\item For each $0 < s < t < 1$, the conditional law of $Z|_{[s,t]}$ given $Z|_{[0,s]}$ and $Z|_{[t,1]}$ is given by that of a two-dimensional Brownian motion $(A,B)$ on $[s,t]$ with $\var{A_u} = \var{B_u} = u-s$ and $\cov{A_u}{B_u} = \alpha (u-s)$ conditioned on $A_u,B_u \geq 0$ for all $u \in [s,t]$ and on $(A_s,B_s) = Z_s$ and $(A_t,B_t) = Z_t$.
\end{enumerate}
\end{theorem}

Throughout, we let $Z = (X,Y)$ be a two-dimensional Brownian motion with $\var{X_t} = \var{Y_t} = t$ and $\cov{X_t}{Y_t} = \alpha t$ for $\alpha \in (-1,1)$ fixed.  For $z \in \C$, we let $\p^z$ denote the law under which $Z_0 = z$ and let $\E^z$ be the corresponding expectation.  The first step in the proof of the existence component of Theorem~\ref{thm::brownian_excursion_measure_exists} is to prove the existence of a process taking values in $\R_+^2$ which corresponds to Brownian motion conditioned to stay in $\R_+^2$.  This process was constructed by Shimura in \cite{shimura1985conebm} and the next two lemmas can be found in \cite{shimura1985conebm}.

\newcommand{\stayinlaw}{\nu}

For each $\delta \geq 0$, we let $P_\delta$ be the event that $X_t,Y_t \geq -\delta$ for all $t \in [0,1]$.

\begin{lemma}
\label{lem::conditioned_process_exists}
There exists a law $\stayinlaw$ on continuous processes $W \colon [0,1] \to \R_+^2$ with $W_0 = 0$ such that:
\begin{enumerate}[(i)]
\item $\p[ W_t > 0] = 1$ for all $t \in (0,1)$,
\item For each $t \in (0,1)$, the conditional law of $W|_{[t,1]}$ given $W|_{[0,t]}$ is that of a Brownian motion $(A,B)$ on $[t,1]$ with $\var{A_u} = u-t$, $\var{B_u} = u-t$, and $\cov{A_u}{B_u} = \alpha (u-t)$ conditioned (on the positive probability event) to stay in $\R_+^2$ in $[t,1]$ and starting from $W_t$, and
\item With respect to the topology of uniform convergence, we have both
\begin{enumerate}[(a)]
\item\label{it::delta_to_zero_conv} The law of $Z|_{[0,1]}$ under $\prstart{0}{ \cdot \giv P_\delta}$ converges weakly to $\stayinlaw$ as $\delta \to 0$.
\item\label{it::z_to_zero_conv} The law of $Z|_{[0,1]}$ under $\prstart{z}{ \cdot \giv P_0}$ converges weakly to $\stayinlaw$ as $z \to 0$ in $\R_+^2$.
\end{enumerate}
\end{enumerate}
\end{lemma}

\begin{lemma}
\label{lem::transition_kernel_continuous}
For $0 \leq s \leq t \leq 1$, we let $p^{s,t}(z,w)$ be the transition density for the process constructed in Lemma~\ref{lem::conditioned_process_exists}.  Then $(s,t,z,w) \mapsto p^{s,t}(z,w)$ is bounded and continuous in $z,w \in \R_+^2$.  Moreover, $p^{s,t}(z,w) > 0$ for all $0 \leq s < t \leq 1$ and $z,w \in \R_+^2$.
\end{lemma}

We direct the reader to \cite{shimura1985conebm} for an explicit formula for $p^{s,t}(x,y)$.  By combining Lemma~\ref{lem::conditioned_process_exists} and Lemma~\ref{lem::transition_kernel_continuous}, we obtain the following.

\begin{lemma}
\label{lem::conditioned_process_given_endpoint_exists}
Fix $w \in \R_+^2$, $s, t \in (0,1)$ with $s < t$, and let $\stayinlaw$ be as in Lemma~\ref{lem::conditioned_process_exists}.  Then both of the following hold with respect to the topology of uniform convergence.
\begin{enumerate}[(i)]
\item The law of $Z|_{[0,s]}$ under $\prstart{0}{ \cdot \giv P_\delta, Z_t = w}$ converges weakly as $\delta \to 0$ to $\stayinlaw[ \cdot \giv Z_t = w]$.
\item The law of $Z|_{[0,s]}$ under $\prstart{z}{ \cdot \giv P_0, Z_t = w}$ converges weakly as $z \to 0$ in $\R_+^2$ to $\stayinlaw[ \cdot \giv Z_t = w]$. 
\end{enumerate}
\end{lemma}
\begin{proof}
We are going to prove the first assertion of the lemma.  The second assertion is proved similarly.  Fix $s,t \in (0,1)$ with $s < t$ and $w \in \R_+^2$.  By a Bayes' rule calculation and the Markov property, the Radon-Nikodym derivative of the law of $Z|_{[0,s]}$ under $\prstart{0}{ \cdot \giv P_\delta, Z_t = w}$ with respect to $\prstart{0}{ \cdot \giv P_\delta}$ is given by
\begin{equation}
\label{eqn::terminal_conditioned_rn}
\CZ_{w,\delta}^{s,t} = \frac{p_\delta^{s,t}(Z_s,w)}{p_\delta^{0,t}(0,w)}
\end{equation}
where $p_\delta^{s,t}$ is the transition density for $Z$ given $P_\delta$.  Moreover, the Radon-Nikodym derivative $\CZ_w^{s,t}$ of $\stayinlaw[ \cdot \giv Z_t = w]$ with respect to $\stayinlaw$ takes the same form.  Lemma~\ref{lem::transition_kernel_continuous} implies that $\CZ_{w,\delta}^{s,t}$ is bounded and continuous as a function of $Z_s$ and it follows from \cite{shimura1985conebm} that $\CZ_{w,\delta}^{s,t} \to \CZ_w^{s,t}$ uniformly as $\delta \to 0$.  Therefore the result follows from Lemma~\ref{lem::conditioned_process_exists}.
\end{proof}

\begin{proof}[Proof of Theorem~\ref{thm::brownian_excursion_measure_exists}]
We are first going to show that there is at most one such law.

Suppose that $Z$ and $\wt{Z}$ are two processes which satisfy the hypotheses of the proposition.  Fix $t \in (0,1)$ and note that the Markovian hypothesis implies that both $Z_t$ and $\wt{Z}_t$ have continuous densities with respect to Lebesgue measure which are everywhere positive in $\R_+^2$.  Moreover, Lemma~\ref{lem::conditioned_process_given_endpoint_exists} implies that for each $t \in (0,1)$ and $w \in \R_+^2$ we have that the laws of $Z|_{[0,t]}$ given $Z_t = w$ and $\wt{Z}|_{[0,t]}$ given $\wt{Z}_t = w$ are the same.  Similarly, the laws of $Z|_{[t,1]}$ given $Z_t = w$ and $\wt{Z}|_{[t,1]}$ given $\wt{Z}_t = w$ are the same.

Fix $0 = t_0 < t_1 < t_2 < t_3 = 1$ and fix $R > 0$.  Consider the Markov chain which in each step sequentially resamples $Z_{t_j}$ given $Z_{t_{j-1}}$ and $Z_{t_{j+1}}$ for $j=1,2$ conditioned on $Z_{t_j} \leq R$ (i.e., each coordinate is at most $R$).  Then we have that the laws of both $(Z_{t_1},Z_{t_2})$ conditioned on $Z_{t_j} \leq R$ for $j=1,2$ and $(\wt{Z}_{t_1},\wt{Z}_{t_2})$ conditioned on $\wt{Z}_{t_j} \leq R$ for $j=1,2$ are invariant for this chain (i.e., the same resampling kernel).  Note that the conditional law of $Z_{t_1}$ conditioned on $Z_{t_1} \leq R$ given any value of $Z_{t_2}$ has a density with respect to Lebesgue measure on $[0,R]^2$ which is positive on $(0,R)^2$.  The same is likewise true with the roles of $t_1$ and $t_2$ swapped and with $\wt{Z}$ in place of $Z$.  It thus follows that by running this chain for one step, we can couple $Z$ and $\wt{Z}$ conditioned on both coordinates being at most $R$ at times $t_1,t_2$ so that $Z_{t_j} = \wt{Z}_{t_j}$ for $j=1,2$ with positive probability (since the conditional law of each given its neighbors is the same).  Therefore it follows from \cite[Theorem~14.10]{georgii2011gibbs} that the law of $(Z_{t_1},Z_{t_2})$ conditioned so that $Z_{t_j} \leq R$ for $j=1,2$ is the same as the law of $(\wt{Z}_{t_1},\wt{Z}_{t_2})$ conditioned so that $\wt{Z}_{t_j} \leq R$ for $j=1,2$.  Indeed, the above implies that this chain cannot have distinct ergodic measures because distinct ergodic measures are necessarily singular.  Since $R > 0$ was arbitrary, we therefore have that the law of $(Z_{t_1},Z_{t_2})$ is the same as the law of $(\wt{Z}_{t_1},\wt{Z}_{t_2})$.  Uniqueness follows since the Markov hypothesis for $Z$ and $\wt{Z}$ implies that we can sample the rest of $Z$ and $\wt{Z}$ to be a.s.\ the same given that the two processes agree at times $t_1$ and $t_2$.

We will now prove existence.  Let $\wh{Z}$ have the law of the process constructed in Lemma~\ref{lem::conditioned_process_exists}.  We are now going to construct the law of $\wh{Z}$ conditioned on $\wh{X}_1 = x_1$ and $\wh{Y}_1 = y_1$ and then argue that the resulting conditioned process satisfies the properties listed in the proposition.  This will complete the proof of existence.

Let $w_1 = (x_1,y_1)$.  In order to construct this process, we first fix $\epsilon > 0$ and consider the law of $\wh{Z}$ conditioned on the positive probability event that $\wh{Z}_1 \in B(w_1,\epsilon)$.  By an application of Bayes' rule, the Radon-Nikodym derivative of the conditioned law of $\wh{Z}|_{[0,t]}$ with respect to the unconditioned law is given by
\begin{equation}
\label{eqn::rn_bayes_form}
\wh{\CZ}_t^\epsilon(y) = \frac{\pr{\wh{Z}_1 \in B(w_1,\epsilon) \giv \wh{Z}_t=y}}{\pr{\wh{Z}_1 \in B(w_1,\epsilon)}} \quad\text{for}\quad y \in \R_+^2.
\end{equation}
We define 
\[ \pr{\wh{Z}_1 \in B(w_1,\epsilon) \giv \wh{Z}_t = 0} = \lim_{w \to 0} \pr{\wh{Z}_1 \in B(w_1,\epsilon) \giv \wh{Z}_t = w}.\]
It is easy to see from the explicit form of $p^{s,t}$ given in \cite{shimura1985conebm} that there exists a constant $c_0 > 0$ such that
\begin{equation}
\label{eqn::rn_bayes_form2}
\wh{\CZ}_t^\epsilon(y) \leq c_0 \left( \frac{\pr{\wh{Z}_1 \in B(w_1,\epsilon) \giv \wh{Z}_t=0}}{\pr{\wh{Z}_1 \in B(w_1,\epsilon)}} \right) \quad\text{for all}\quad y \in \R_2^+.
\end{equation}
Fix $h > 0$ so that $t < t+h < 1$ and let
\[ \wh{Q}_{t,h} = \left\{ z \in \R_+^2 : \pr{ \wh{Z}_1 \in B(w_1,\epsilon) \giv \wh{Z}_{t+h} = z} \geq \frac{1}{2} \pr{ \wh{Z}_1 \in B(w_1,\epsilon) \giv \wh{Z}_t = 0 } \right\}.\]
The Markov property together with Lemma~\ref{lem::transition_kernel_continuous} implies that $\wh{Q}_{t,h}$ has positive Lebesgue measure.  Applying Lemma~\ref{lem::transition_kernel_continuous} again implies that there exists $p_{t,h} > 0$ depending only on $t$ and $h$ such that
\begin{equation}
\label{eqn::p_t_h_lbd}
\pr{ \wh{Z}_{t+h} \in \wh{Q}_{t,h}} = p_{t,h} > 0.
\end{equation}
Combining~\eqref{eqn::p_t_h_lbd} with the Markov property for $\wh{Z}$ implies that
\begin{equation}
\label{eqn::markov_lbd}
\pr{ \wh{Z}_1 \in B(w_1,\epsilon) } \geq \frac{p_{t,h}}{2} \pr{ \wh{Z}_1 \in B(w_1,\epsilon) \giv \wh{Z}_t = 0}.
\end{equation}
Inserting~\eqref{eqn::markov_lbd} into~\eqref{eqn::rn_bayes_form2} implies that
\begin{equation}
\label{eqn::rn_pth_ubd}
\sup_{y \in \R_+^2} \wh{\CZ}_t^\epsilon(y) \leq \frac{2 c_0 }{p_{t,h}} < \infty.
\end{equation}

Therefore the random variables $\wh{\CZ}_t^{\epsilon} = \wh{\CZ}_t^{\epsilon}(\wh{Z}_t)$ are uniformly integrable.  Therefore there exists a positive sequence $(\epsilon_k)$ decreasing to $0$ such that the law of $\wh{\CZ}_t^{\epsilon_k}(\wh{Z}_t)$ converges weakly to a limit $\wh{\CZ}_t$ which has expectation $1$.  By passing to a further (diagonal) subsequence if necessary, we can arrange so that $\wh{\CZ}_t^{\epsilon_k}$ converges weakly as $k \to \infty$ to a limit which has expectation $1$ for all rational $t \in (0,1)$.  It is easy to see that the family of measures obtained by weighting the law of $\wh{Z}|_{[0,t]}$ by $\wh{\CZ}_t$ for $t \in (0,1)$ is consistent and the measure obtained from the $t \to 1$ limit satisfies the Markov property described in the statement of the lemma.  The continuity of the process at the terminal point can be seen by a time-reversal argument.
\end{proof}

As recalled in Theorem~\ref{thm::disk_explore}, it is shown in \cite{dms2014mating} that when $\alpha \in [0,1)$, a $\pi/2$-cone excursion of $Z$ of terminal displacement $\epsilon > 0$ naturally encodes a quantum disk with $\gamma \in [\sqrt{2},2)$ with quantum boundary length $\epsilon$ where $\alpha$ and $\gamma$ are related as in~\eqref{eqn::l_r_cov}.  We are now going to show that if one generates such a $\pi/2$-cone excursion conditioned further on the event that its length is equal to $1$, then the conditional law converges as $\epsilon \to 0$ to the law constructed and characterized in Theorem~\ref{thm::brownian_excursion_measure_exists}.  This is natural in view of Theorem~\ref{thm::sphere_equivalent_constructions} because it is natural to expect that the law of a quantum disk with boundary length $\epsilon$ and area $1$ converges to that of a unit area quantum sphere as $\epsilon \to 0$.

\begin{proposition}
\label{prop::cone_to_sphere_brownian_limit}
Fix $\epsilon > 0$ and suppose that $Z$ has the law of a $\pi/2$-cone excursion of length $1$ and terminal displacement equal to $\epsilon$.  Then the law of $Z$ converges weakly as $\epsilon \to 0$ to the law constructed in Theorem~\ref{thm::brownian_excursion_measure_exists} with respect to the topology of uniform convergence.
\end{proposition}
\begin{proof}
Let $Z = (X,Y)$ be a $\pi/2$-cone excursion with length $1$ and terminal displacement equal to $\epsilon$ and let $\wt{Z}$ be the time-reversal of $Z$.  Then arguing as in the proof of Theorem~\ref{thm::brownian_excursion_measure_exists} we can view $\wt{Z}$ as a correlated Brownian motion conditioned to stay in~$\R_+^2$ and conditioned to terminate at the origin.  The argument of the proof of Theorem~\ref{thm::brownian_excursion_measure_exists} gives that the Radon-Nikodym derivative between the law of $\wt{Z}|_{[0,t]}$, $t \in (0,1)$ fixed, and the law of a correlated Brownian motion in $[0,t]$ starting from $\wt{Z}_0$ and conditioned to be in $\R_+^2$ is bounded from above by a constant which depends only on $t$.  For the latter law, the results of \cite{shimura1985conebm} imply that we have a limiting process as $\wt{Z}_0 \to 0$.  This limiting process clearly satisfies the hypotheses of Theorem~\ref{thm::brownian_excursion_measure_exists}, which completes the proof.
\end{proof}

We end this section with the following estimate for the probability of a bi-infinite correlated Brownian motion having a $\pi/2$-cone excursion with terminal displacement at most $1$ and starting time in the interval $[-k-1,-k]$.

\begin{proposition}
\label{prop::finitely_many_large_cone_excursions_short_displacement}
Fix $\alpha \in (-1,1)$.  Suppose that $Z = (X,Y)$ is a Brownian motion $\R \to \R^2$ normalized so that $Z_0 = 0$ with
\[ \var{X_t} = \var{Y_t} = |t| \quad\text{and}\quad \cov{X_t}{Y_t} = \alpha |t|.\]
There exist constants $c_0 > 0$, $\beta > 1$ depending only on $\alpha$ such that the following is true.  For each $k \in \N$, let $E_k$ be the event that $Z$ has a $\pi/2$-cone excursion starting in $[-k-1,-k]$ of length at least $k+1$ and terminal displacement at most $1$.  Then
\begin{equation}
\label{eqn::ek_bm_bound}
\p[E_k] \leq c_0 k^{-\beta}.
\end{equation}
In particular, the number of $k \in \N$ such that $E_k$ occurs is finite a.s.
\end{proposition}
\begin{proof}
Fix $\epsilon > 0$; we will adjust its value at the end of the proof.  We first note that the reflection principle for Brownian motion implies that there exist constants $c_1,c_2 > 0$ such that the probability of the event~$F_k$ that $\sup_{s,t \in [-k-1,-k]} |Z_s-Z_t| \geq k^\epsilon$ is at most~$c_1 e^{-c_2 k^{2\epsilon}}$.  It therefore suffices to establish~\eqref{eqn::ek_bm_bound} with $E_k \cap F_k^c$ in place of~$E_k$.

The probability of the event $E_k \cap F_k^c$ is bounded from above by the probability of the event~$G_k$ that a Brownian motion~$\wt{Z}$ in~$\R^2$ with the same covariance as $Z$ starting from $k^\epsilon(1+i)$ stays in $\R_+^2$ for all $t \in [0,k]$ and hits $\partial B(0,2k^\epsilon)$ after time~$k$ before exiting~$\R_+^2$.

We note that there exist constants $c_3,c_4 > 0$ such that the probability of the event that $\wt{Z}|_{[0,k]}$ stays in $B(0,k^{1/2-\epsilon}) \cap \R_+^2$ is at most $c_3 e^{-c_4 k^{2\epsilon}}$.  Indeed, the reason for this is that in each round of time of length $k^{1-2\epsilon}$, we have that $\wt{Z}$ has a positive of chance of leaving $B(0,k^{1/2-\epsilon})$ which is uniform in its starting point in $B(0,k^{1/2-\epsilon})$ at the start of the round.  It therefore suffices to bound the probability of the event $G_k'$ that $\wt{Z}$ hits $\partial B(0,k^{1/2-\epsilon}) \cap \R_+^2$ and then hits $B(0,2k^\epsilon)$ before exiting $\R_+^2$.

Let
\[ \Lambda = \frac{1}{(1-\alpha^2)^{1/2}} \begin{pmatrix} (1-\alpha^2)^{1/2} & 0 \\ -\alpha  & 1 \end{pmatrix}.\]
Then $\wh{Z} = \Lambda \wt{Z}$ is a standard Brownian motion in $\R^2$.  Moreover, $\Lambda$ takes $\R_+^2$ to a Euclidean wedge $\W_\theta$ of opening angle $\theta = \arccos(-\alpha) \in (0,\pi)$.  Let $\zeta = \pi/\theta > 1$.  Fix $a \in \C$ with $|a|=1$ so that the map $z \mapsto a z^\zeta$ takes $\W_\theta$ to $\h$ and let $\acute{Z}$ be the image of $\wh{Z}$ under this map.  

The event that $\wt{Z}$ hits $\partial B(0,k^{1/2-\epsilon})$ before exiting $\R_+^2$ corresponds to the event that $\acute{Z}$ escapes to distance of order $k^{\zeta(1/2-\epsilon)}$ before exiting $\h$.  Since $\acute{Z}$ starts with imaginary part of order $k^{\zeta \epsilon}$, it follows that there exists a constant $c_5 > 0$ such that the probability of this event is at most
\begin{equation}
\label{eqn::escape_probability}
c_5 k^{\zeta(2\epsilon-1/2)}.
\end{equation}
Conditional on this event, it is clear from the explicit form of the Poisson kernel on~$\h$ that there exists a constant $c_6 > 0$ such that the probability that $\acute{Z}$ hits a ball centered at the origin with size proportional to $k^{\zeta \epsilon}$ after reaching distance of order $k^{\zeta(1/2-\epsilon)}$ and before exiting $\h$ is at most
\begin{equation}
\label{eqn::come_back_probability}
c_6 k^{\zeta(2\epsilon-1/2)}.
\end{equation}
By taking $\epsilon > 0$ sufficiently small,~\eqref{eqn::ek_bm_bound}
 follows by combining~\eqref{eqn::escape_probability} and~\eqref{eqn::come_back_probability}.

The second assertion of the proposition is an immediate consequence of the first and the Borel-Cantelli lemma.
\end{proof}

\section{Constructing a unit area quantum sphere from a $\gamma$-quantum cone}
\label{sec::sphere_from_cone}

The purpose of this section is to show how to construct a unit area quantum sphere by pinching off a unit of quantum area from a $\gamma$-quantum cone.  This construction will be important for our proofs of Theorem~\ref{thm::sphere_equivalent_constructions} and Theorem~\ref{thm::pure_sphere_equivalent_constructions}.  We remark that this result is similar to \cite[Proposition~A.11]{dms2014mating}, though the present setting turns out to be simpler.

\begin{proposition}
\label{prop::cone_to_sphere}
Fix $\gamma \in (0,2)$ and suppose that $\CC = (\cyl,h,+\infty,-\infty)$ is a $\gamma$-quantum cone.  Let $X$ be the projection of $h$ onto $\CH_1(\cyl)$.  For each $r \in \R$ and $\epsilon > 0$, let
\begin{equation}
\label{eqn::tau_r_e_r_eps_def}
\tau_r = \inf\{ u \in \R : X_u \leq r\} \quad\text{and}\quad E_{r,\epsilon} = \left\{ 1 \leq \mu_h( \cyl_+ + \tau_r) \leq 1+ \epsilon \right\}.
\end{equation}
The laws of the quantum surfaces $(\cyl_+ + \tau_r,h)$ given $E_{r,\epsilon}$ converge weakly in the space of distributions to that of the unit area quantum sphere when we take a limit first as $r \to -\infty$ and then as $\epsilon \to 0$.  More precisely, if we start with the quantum surface $(\cyl_+ + \tau_r,h)$ given $E_{r,\epsilon}$ and then embed it by taking the horizontal translation so that the amount of quantum mass in $\cyl_+$ is equal to $1/2$, then its law converges weakly in the space of distributions as $r \to -\infty$ and then $\epsilon \to 0$ to that of the unit area quantum sphere with the embedding taken so that the amount of mass assigned $\cyl_+$ is $1/2$.

Moreover, for any fixed $S > 0$, the conditional law of the surfaces $(\cyl_+ + \tau_r,h)$ given $E_{r,\epsilon}$ and the restriction of $h$ to $U_r = (-\infty,\tau_r+S] \times [0,2\pi]$ converges in probability to that of the unit area quantum sphere with respect to the topology of weak convergence in the space of distributions (over the realization of the restriction of $h$ to $U_r$) when we take a limit first as $r \to -\infty$ and then $\epsilon \to 0$.
\end{proposition}

The idea of the proof of Proposition~\ref{prop::cone_to_sphere} is to introduce the auxiliary event (defined in the statement of Lemma~\ref{lem::cond_x_large} just below) that $X$ takes on the value $\gamma^{-1} \log (\beta^{-1})$ after time $\tau_r$ for a fixed value of $\beta > 0$.  Standard facts about Bessel processes imply that the law of $X$ conditioned on this event converges as $r \to -\infty$ to the $\log$ of the same type of Bessel excursion used to construct the unit area quantum sphere conditioned to take on a large value and then reparameterized by quadratic variation.  We will then argue that this event occurs with probability tending to $1$ given $E_{r,\epsilon}$ as we decrease $\beta$ and that, conversely, $E_{r,\epsilon}$ occurs with positive probability for each fixed choice of $\beta$ uniformly in $r$ (Lemma~\ref{lem::two_e_events}).  Combining these two results will lead to the first assertion of Proposition~\ref{prop::cone_to_sphere}.  The second assertion follows from a similar argument.

\begin{lemma}
\label{lem::cond_x_large}
Suppose that we have the setup described in Proposition~\ref{prop::cone_to_sphere}.  Let
\begin{equation}
\label{eqn::e_c_beta_def}
E_{r,\beta}' = \left\{ \sup_{u \geq \tau_r} X_u \geq \gamma^{-1} \log\left( \beta^{-1} \right) \right\}.
\end{equation}
Let $\wt{X}$ be given by $2 \gamma^{-1} \log Z$ reparameterized to have quadratic variation $du$ where $Z$ is sampled from the excursion measure of a Bessel process of dimension $\delta = 4-\tfrac{8}{\gamma^2}$ conditioned on having maximum at least~$\beta^{-1/2}$.  Then the law of $u \mapsto X_{u+\tau_r}$ conditioned on~$E_{r,\beta}'$ converges as $r \to -\infty$ weakly with respect the topology of local uniform convergence to the law of~$\wt{X}$, where we have taken the horizontal translation for both so that they hit~$\beta^{-1/2}$ for the first time at $u=0$.
\end{lemma}
\begin{proof}
This follows from some standard properties of Bessel processes, which can be found in \cite[Section~3]{dms2014mating}.  (See, for example, \cite[Lemma~3.6]{dms2014mating}.)  In particular,
\begin{itemize}
\item If $B$ is a standard Brownian motion and $a \in \R$, then the process given by $e^{B_t + at}$ reparameterized to have quadratic variation $dt$ is a Bessel process of dimension $2+2a$.  Applying this to the process $X$, we see that $e^{(\gamma/2) X_t}$ reparameterized to have quadratic variation $dt$ is a Bessel process of dimension $4-8/\gamma^2$.  Sending $r \to -\infty$ corresponds to taking the starting point of the Bessel process to be equal to $0$.
\item Conditioning $X$ to exceed the value $\gamma^{-1} \log (\beta^{-1})$ is equivalent to conditioning $e^{(\gamma/2) X_t}$ to exceed the value $\beta^{-1/2}$.
\end{itemize}
\end{proof}

\begin{lemma}
\label{lem::two_e_events}
Suppose that we have the same setup as described in Proposition~\ref{prop::cone_to_sphere}.  Let $E_{r,\epsilon}$ be as in~\eqref{eqn::tau_r_e_r_eps_def} and $E_{r,\beta}'$ be as in~\eqref{eqn::e_c_beta_def}.  Then we have both
\begin{align}
\pr{ E_{r,\beta}' \giv E_{r,\epsilon}} \to 1 \quad\text{as}\quad \beta \to \infty \quad\text{uniformly in}\quad r \leq 0 \quad\text{and}\\
\pr{E_{r,\epsilon} \giv E_{r,\beta}'} > 0 \quad\text{uniformly in}\quad r \leq 0 \quad\text{for}\quad \epsilon,\beta > 0 \quad\text{fixed}.
\end{align}
\end{lemma}
\begin{proof}
This follows from the same argument used to prove \cite[Lemma~A.4]{dms2014mating}.
\end{proof}

\begin{proof}[Proof of Proposition~\ref{prop::cone_to_sphere}]
By Lemma~\ref{lem::cond_x_large} and Lemma~\ref{lem::two_e_events}, it follows that the law of $h$ conditioned on both $E_{r,\epsilon}$ and $E_{r,\beta}'$ converges as $r \to -\infty$ and then as $\epsilon \to 0$ to that of a unit area quantum sphere conditioned on the positive probability event that the supremum of its projection onto $\CH_1(\cyl)$ exceeds $\gamma^{-1} \log(\beta^{-1})$.  Taking a further limit as $\beta \to \infty$ yields the law of a unit area quantum sphere.  The first assertion of the proposition then follows because Lemma~\ref{lem::two_e_events} implies that the conditional law of $h$ given both $E_{r,\epsilon}$ and $E_{r,\beta}'$ is close to the law conditioned on only $E_{r,\epsilon}$ when $\beta > 0$ is large.

We are now going to justify the second assertion.  Fix $S > 0$. Let $\Fh_r$ be the distribution on $\cyl_+ + \tau_r$ which is given by harmonically extending $h$ from $U_r$ to $\cyl_+ + \tau_r + S$.  We first claim that, given both $E_{r,\epsilon}$ and $E_{r,\beta}'$, $\Fh_r$ restricted to $\cyl_+ + \tau_r + S + T$ is close to a constant with respect to the uniform topology for large $T$.  Lemma~\ref{lem::cond_x_large} implies that this is the case when we only condition on $E_{r,\beta}'$.  Indeed, we can write this harmonic extension as the sum $f_1 + f_2$ where $f_i$ is the part which comes from harmonically extending the projection of $h$ onto $\CH_i(\cyl)$ for $i=1,2$.  The restriction of $f_1$ to $\tau_r + S + T$ is constant as it is given by the harmonic extension of a function which is constant on vertical lines.  The law of $f_2$ restricted to $\cyl_+ + \tau_r + S+T$ converges weakly to a constant with respect to the uniform topology because conditioning on $E_{r,\beta}'$ does not affect $h_2$, $h_2$ is independent of $\tau_r$, and the law of $h_2$ is translation invariant.  Lemma~\ref{lem::two_e_events} implies that the same convergence holds when we condition on both $E_{r,\epsilon}$ and $E_{r,\beta}'$.  The remainder of the proof of the second assertion thus follows from the same argument used to establish the first assertion.
\end{proof}

\section{Equivalence of Bessel and Brownian constructions}
\label{sec::bessel_brownian_constructions}

The purpose of this section is to give the proof of Theorem~\ref{thm::sphere_equivalent_constructions}.  A variant of this argument, which we will explain in Section~\ref{sec::levy_construction}, gives Theorem~\ref{thm::pure_sphere_equivalent_constructions}.  

\subsection{Setup and strategy}
\label{subsec::setup_and_strategy}

Before we proceed to the proof, we will first give an overview of the steps (and introduce many of the objects and events). We will first consider the case that $\gamma \in (\sqrt{2},2)$.  We suppose that we are working on a $\gamma$-quantum cone $\CC = (\cyl,h,+\infty,-\infty)$ and that~$\eta'$ is a space-filling $\SLE_{\kappa'}$ process in $\cyl$ from $-\infty$ to $-\infty$ sampled independently of $h$ and then reparameterized by quantum area.  In other words, for each $s < t$ we have that $\mu_h(\eta'([s,t])) = t-s$ with the normalization that $\eta'(0) = +\infty$.  Let $\wt{\eta}' \colon \R \to \cyl$ be the $\SLE_{\kappa'}(\kappa'-6)$ process from $-\infty$ to $+\infty$ which arises by taking $\eta'$ and reparameterizing it according to capacity as seen from $+\infty$.  Equivalently, $\wt{\eta}'$ is the counterflow line from $-\infty$ to $+\infty$ of the GFF used to generate $\eta'$.  For each time $t$, we let $\theta_t$ be equal to $2\pi$ times the harmonic measure of the left side of $\wt{\eta}'([0,t])$ as seen from $+\infty$.   Since $\theta_t$ is continuous in $t$ we have that the set $\wt{\CT}$ of times $t$ such that $\theta_t = 0$ or $\theta_t = 2\pi$ is closed in $\R$. Hence we can write $\R \setminus \wt{\CT}$ as a countable disjoint union of open intervals $\cup_j (s_j,t_j)$.  We then let $\CT$ be the countable and discrete subset (i.e., without limit points) of $\wt{\CT}$ which consists of those $t_j$ such that $\theta_{t_j} \neq \theta_{s_j}$.  These times correspond to when $\wt{\eta}'$ makes loops which disconnect $+\infty$ from $-\infty$ with alternating clockwise and counterclockwise orientation.

Let $L$ (resp.\ $R$) denote the change in the quantum boundary length of the left (resp.\ right) side of $\eta'((-\infty,t])$ relative to time $0$.  Each of the times in $\CT, \wt{\CT}$ corresponds to a $\pi/2$-cone excursion of the time-reversal $(\wt{L},\wt{R})$ of $(L,R)$ on intervals of time which contain $0$.  Recall that the definition of a $\pi/2$-cone time is given in Section~\ref{subsubsec::quantum_disks} and the connection between $\pi/2$-cone times for $(\wt{L},\wt{R})$ and the behavior of space-filling $\SLE_{\kappa'}$ is explained in the introduction of \cite{dms2014mating}; see in particular \cite[Figure~1.13]{dms2014mating}.

For $r \in \R_-$ and $C > 1$, we let $\zeta_{r,C}$ be the first time $t \in \CT$ that the quantum boundary length of the component containing $+\infty$ is at most $C^{-1} e^{\gamma r/2}$.  We note that a first such time $t$ a.s.\ exists by Proposition~\ref{prop::finitely_many_large_cone_excursions_short_displacement} as the complementary component containing $+\infty$ of $\wt{\eta}'$ drawn up to such a time corresponds to a $\pi/2$-cone excursion of $(\wt{L},\wt{R})$ with terminal displacement at most $C^{-1} e^{\gamma r/2}$.  Let $U_{r,C}$ be this component and let $\ell_{r,C}$ be its quantum boundary length.

Let $F_{r,\epsilon,C}$ be the event that $\ell_{r,C}$ is contained in $I_{r,C} := [\tfrac{1}{2},1] \cdot C^{-1} e^{\gamma r/2}$ and the quantum area of $U_{r,C}$ is in $[1,1+\epsilon]$.  Throughout, we let $\tau_r$ and $E_{r,\epsilon}$ be as in~\eqref{eqn::tau_r_e_r_eps_def}.

The first step (carried out in Section~\ref{subsec::comparison_of_pinched_quantum_cones}) is to show that the conditional probability of $E_{r,\epsilon}$ given $F_{r,\epsilon,C}$ converges to $1$ as $C \to \infty$ uniformly in $r$ and that the conditional probability of $F_{r,\epsilon,C}$ given $E_{r,\epsilon}$ is uniformly positive in $r$ when $C$ is fixed.  This implies that we can view the joint law of $h$ and $\eta'$ conditioned on $F_{r,\epsilon,C}$ as arising by first conditioning on $E_{r,\epsilon}$ and then subsequently conditioning the resulting law on the uniformly positive conditional probability event $F_{r,\epsilon,C}$.

The second step (carried out in Section~\ref{subsec::asymptotic_mixing}) is to show that the conditional law of $h$ and $\eta'$ given $E_{r,\epsilon}$ is close to the law which results when we condition further on $F_{r,\epsilon,C}$.  The idea to establish this is first to take the horizontal translation of the embedding of $\CC$ into $\cyl$ so that the quantum area of $\cyl_+$ is equal to $1/2$.  Whether or not the event $F_{r,\epsilon,C}$ occurs is determined by the behavior of $\eta'$ and $h$ in $\cyl_- + u$ for $u < 0$ very negative.  Since the conditional law of $\eta'$ and $h$ in $\cyl_+ + v$ for $v$ much larger than $u$ given their behavior in $\cyl_- + u$ is not far from their unconditioned law (both $h$ and $\eta'$ ``forget their past'' quickly), it follows that their joint conditional law converges to that of a unit area quantum sphere decorated by an independent space-filling $\SLE_{\kappa'}$ process upon taking limits.

The result then follows for $\gamma \in (\sqrt{2},2)$ because by Proposition~\ref{prop::cone_to_sphere_brownian_limit} the law of $(\wt{L},\wt{R})$ conditional on $F_{r,\epsilon,C}$ restricted to the interval $J_{r,C}$ of time in which $\eta'$ is filling $U_{r,C}$ converges when we take appropriate limits to a correlated Brownian loop as constructed in Section~\ref{sec::brownian_excursions}.

At the end of this section, we will explain how a variant of this argument gives the case that $\gamma \in (0,\sqrt{2}]$.

\subsection{Exploring a $\gamma$-quantum cone}
\label{subsec::exploring_gamma_cone_infinity_to_zero}

We are now going to identify the conditional law of the unexplored region in a $\gamma$-quantum cone, $\gamma \in (\sqrt{2},2)$, when one draws a whole-plane $\SLE_{\kappa'}(\kappa'-6)$ in $\cyl$ from $-\infty$ to $+\infty$ up to the first time $t \in \CT$ that the quantum boundary length of the complementary component containing the origin falls below $1$.  This result will in particular imply that the surface parameterized by this component is conditionally independent of the outside surface given its quantum boundary length.

\begin{proposition}
\label{prop::unexplored_region_qc_inf_to_0}
Suppose that $\gamma \in (\sqrt{2},2)$ and let $\tau$ be the first time $t \in \CT$ that the quantum boundary length of the component~$U$ of $\cyl \setminus \wt{\eta}'([0,t])$ containing $+\infty$ falls below~$1$.  Then the conditional law of the quantum surface $(U,h)$ given its quantum boundary length is that of a quantum disk weighted by its quantum area.  That is, if $b$ denotes the quantum boundary length of $\partial U$ and $\bdisk^b$ denotes the law of a quantum disk with boundary length $b$, then the conditional law of $(U,h)$ given $b$ has Radon-Nikodym derivative with respect to $\bdisk^b$ equal to the quantum area of $U$ times a normalization constant to make it a probability measure.  Moreover, given its quantum boundary length, the quantum surface $(U,h)$ is conditionally independent of the quantum surface $(\cyl \setminus U,h)$ and the ordered sequence of marked, oriented quantum surfaces separated by $\wt{\eta}'$ from $+\infty$ before time $\tau$.  Finally, $\wt{\eta}'(\tau)$ is uniformly distributed from the quantum boundary measure on $\partial U$.
\end{proposition}

Recall that weighting the law of a quantum disk by its quantum area corresponds to adding a marked point to the interior of the disk which is sampled from its quantum area measure.  In the setting of Proposition~\ref{prop::unexplored_region_qc_inf_to_0}, the role of the marked point will be played by the origin of the quantum cone $(\cyl,h,+\infty,-\infty)$, i.e., the point at $+\infty$.  Before we proceed to the details of the proof of Proposition~\ref{prop::unexplored_region_qc_inf_to_0}, let us briefly review the strategy.  We first recall from \cite{dms2014mating} that if we draw an independent $\SLE_{\kappa'}(\kappa'-6)$ process on top of a quantum cone, then the structure of the components (viewed as quantum surfaces) that it cuts out is described by a Poissonian collection of quantum disks.  We will combine this fact with the target invariance of $\SLE_{\kappa'}(\kappa'-6)$ processes where the new target point will come from moving the marked point at $+\infty$ of the quantum cone using an independent space-filling $\SLE_{\kappa'}$.

It is natural in view of considerations from the discrete models that in the statement of Proposition~\ref{prop::unexplored_region_qc_inf_to_0} we require that $\tau \in \CT$.  Indeed, $\tau \in \CT$ is the continuous analog of having the boundary conditions for the discrete model on $\partial U$ all having ``the same color.''  Proposition~\ref{prop::unexplored_region_qc_inf_to_0} will be a consequence of the following two observations.  The first is a version of \cite[Lemma~8.3]{dms2014mating}.

For each $z \in \cyl$, we let $\psi_z \colon \cyl \to \cyl$ be the unique conformal transformation with $\psi_z(z) = +\infty$ and $\psi_z'(-\infty) > 0$.  Explicitly, $\psi_z(u) = -\log( e^{-u} - e^{-z})$.  We note that $\psi_z$ is the analog of the map $\C \to \C$ which corresponds to translating a point $z \in \C$ to the origin in the setting in which we represent $\C$ by $\cyl$ with $+\infty$ (resp.\ $-\infty$) corresponding to $0$ (resp.\ $\infty$).

\begin{lemma}
\label{lem::retarget_inf_to_0}
Suppose that we have the same setup as in Proposition~\ref{prop::unexplored_region_qc_inf_to_0}.  Conditionally on $h$ and $\wt{\eta}'$, we suppose that $w \in U$ is picked uniformly from the quantum area measure on $U$ and let $\wt{\eta}_w'$ be a whole-plane $\SLE_{\kappa'}(\kappa'-6)$ process from $-\infty$ to $w$ coupled so as to agree with $\wt{\eta}'$ until $w$ and $+\infty$ are first separated and then taken to continue conditionally independently afterwards.  Then we have as path decorated surfaces that
\[ (\cyl, h \circ \psi_w^{-1} + Q\log|(\psi_w^{-1})'|,+\infty,-\infty,\psi_w(\wt{\eta}_w')) \stackrel{d}{=} (\cyl,h,+\infty,-\infty,\wt{\eta}').\]
\end{lemma}
\begin{proof}
We take $\wt{\eta}'$ and $\wt{\eta}_w'$ to be coupled together so that they both correspond to taking the whole-plane space-filling $\SLE_{\kappa'}$ process $\eta'$ from $-\infty$ to $-\infty$, respectively targeted at $+\infty$ and $w$.  The idea of the proof is to generate a marked point $z$ chosen uniformly from the quantum measure on $\eta'([-R,R])$ for $R > 0$ large.  Since $\eta'$ is parameterized according to quantum area, this is equivalent to taking $z = \eta'(V)$ where $V$ is chosen uniformly in $[-R,R]$ according to Lebesgue measure independently of everything else.  It is convenient to work with the point $z$ in place of $w$ because \cite[Lemma~9.3]{dms2014mating} implies that the quantum surface $(\cyl,h,+\infty,-\infty)$ decorated by $\eta'$ has the same law as the quantum surface $(\cyl, h \circ \psi_z^{-1} + Q \log|(\psi_z^{-1})'|,+\infty,-\infty)$ decorated by the path $\psi_z(\eta')$.  Since $\psi_z(\wt{\eta}_z')$ is generated from $\psi_z(\eta')$ in the same deterministic manner as $\wt{\eta}'$ is generated from $\eta'$, it follows that the quantum surface $(\cyl,h,+\infty,-\infty)$ decorated by $\wt{\eta}'$ has the same law as the quantum surface $(\cyl, h \circ \psi_z^{-1} + Q \log|(\psi_z^{-1})'|,+\infty,-\infty)$ decorated by the path $\psi_z(\wt{\eta}_z')$.  As we will see momentarily, the lemma will follow by considering the latter conditioned on the event $E_R$ that $z \in U$.

Let $\CA$ be any positive probability event for $(\cyl,h \circ \psi_z^{-1}+Q\log|(\psi_z^{-1})'|,+\infty,-\infty,\psi_z(\wt{\eta}_z'))$.  By Bayes' rule, we have that
\begin{align}
    \p[ \CA \giv E_R]
&= \frac{\p[ E_R \giv \CA]}{\p[E_R]} \p[\CA]. \label{eqn::condition_ratio}
\end{align}
As explained above, by \cite[Lemma~8.3]{dms2014mating}, we have that $\p[\CA]$ is equal to the probability of the same event with $(\cyl,h,+\infty,-\infty,\wt{\eta}')$ in place of $(\cyl,h \circ \psi_z^{-1}+Q\log|(\psi_z^{-1})'|,+\infty,-\infty,\psi_z(\wt{\eta}_z'))$.  Assume that $\CA$ is an event of the form $\wt{\CA} \cap \{ a \leq \mu_h(U) \leq a+\delta\}$ for some $a \geq 0$ and $\delta > 0$.  Then sending $R \to \infty$, we see that the right hand side of~\eqref{eqn::condition_ratio} is bounded from below by a constant times $a \p[\CA]$ and from above by the same constant times $(a+\delta) \p[\CA]$.  Therefore the conditional law of $(\cyl,h \circ \psi_z^{-1} + Q \log|(\psi_z^{-1})'|,+\infty,-\infty,\psi_z(\wt{\eta}_z'))$ given $E_R$ converges as $R \to \infty$ to the law of $(\cyl,h,+\infty,-\infty,\wt{\eta}')$ weighted by the quantum area of $U$.  We note that if we fix the quantum area of $U$, then the conditional law of $(\cyl,h \circ \psi_z^{-1} + Q\log|(\psi_z^{-1})'|,+\infty,-\infty,\psi_z(\wt{\eta}_z'))$ given $E_R$ converges as $R \to \infty$ to that of $(\cyl,h \circ \psi_w^{-1}+Q\log|(\psi_w^{-1})'|,+\infty,-\infty,\psi_w(\wt{\eta}_w'))$ (with the quantum area of $U$ for the latter fixed to be the same as the former).  Therefore the law of $(\cyl,h \circ \psi_w^{-1}+Q\log|(\psi_w^{-1})'|,+\infty,-\infty, \psi_w(\wt{\eta}_w'))$ given the quantum area of $U$ is equal to the law of $(\cyl,h,+\infty,-\infty,\wt{\eta}')$ given the quantum area of $U$.  This implies the result because both laws induce the same law on the quantum area of~$U$.
\end{proof}

\begin{lemma}
\label{lem::unexplored_region_resample}
Suppose that we have the same setup as in Proposition~\ref{prop::unexplored_region_qc_inf_to_0} and let $\wh{\eta}'$ be given by $\wt{\eta}'|_{[\tau,\infty)}$.  Suppose that $w$ is picked uniformly in $U$ from the quantum area measure conditionally independently of everything else and let $U_w$ be the component of $U \setminus \wh{\eta}'$ which contains $w$.  Suppose that we condition on the following:
\begin{enumerate}
\item The event that $\partial U_w$ is entirely contained in either the left or the right side of $\wh{\eta}'$,
\item The quantum area and boundary length of $(U_w,h)$, and
\item The quantum areas and boundary lengths of all of the components of $U \setminus \wh{\eta}'$.
\end{enumerate}
Then we have that the quantum surface $(U_w,h)$ has the law of a quantum disk with the given quantum boundary length and area.  In particular, $(U_w,h)$ is conditionally independent of the other components of $U \setminus \wh{\eta}'$ (viewed as quantum surfaces) given its quantum boundary length and area.
\end{lemma}
\begin{proof}
This follows because we know from \cite[Theorem~1.17]{dms2014mating} (see also \cite[Figure~1.18]{dms2014mating}) that the quantum surfaces parameterized by the components of $U \setminus \wh{\eta}'$ whose boundary is entirely contained in either the left or right side of $\wh{\eta}'$ are conditionally independent quantum disks given their boundary length.
\end{proof}

\begin{figure}[ht!]
\begin{center}
\includegraphics[scale=0.65,page=1]{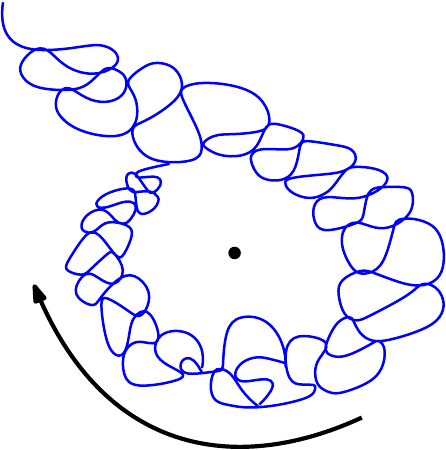} \hspace{0.00\textwidth}\includegraphics[scale=0.65,page=2]{figures/orientation_reversing}\hspace{0.00\textwidth}\includegraphics[scale=0.65,page=3]{figures/orientation_reversing}
\end{center}
\caption{\label{fig::orientation_changing} {\bf Left:} A whole-plane $\SLE_{\kappa'}(\kappa'-6)$ process from $\infty$ to $0$ stopped at a time that it has just closed a clockwise bubble around $0$.  The time-reversal of the associated space-filling $\SLE_{\kappa'}$ will enter this bubble and then fill it up before exiting, which is why it corresponds to a $\pi/2$-cone excursion of the corresponding boundary length process.  {\bf Middle:} More of the $\SLE_{\kappa'}(\kappa'-6)$ process is drawn.  The component containing $0$ is surrounded by one side of the $\SLE_{\kappa'}(\kappa'-6)$ so again corresponds to a $\pi/2$-cone excursion (which is nested inside the previous one). {\bf Right:} The $\SLE_{\kappa'}(\kappa'-6)$ drawn up until finishing the first bubble with a counterclockwise orientation after the time shown in the left panel.  This bubble corresponds to an orientation changing $\pi/2$-cone excursion.}	
\end{figure}

Suppose that $I = [s,t]$ is an interval of time in which $(\wt{L},\wt{R})$ has a $\pi/2$-cone excursion.  Then we say that the $\pi/2$-cone excursion has a left (resp.\ right) orientation if $\wt{L}_t = \wt{L}_s$ (resp.\ $\wt{R}_t = \wt{R}_s$).  This terminology is motivated because a left (resp.\ right) cone excursion exits $\R_+^2 + (\wt{L}_s,\wt{R}_s)$ in its left (resp.\ right) boundary; see \cite[Figure~1.13]{dms2014mating}.  We note that the left (resp.\ right) $\pi/2$-cone excursions of $(\wt{L},\wt{R})$ whose time-interval~$I$ contains~$0$ correspond to the clockwise (resp.\ counterclockwise) loops made by the $\SLE_{\kappa'}(\kappa'-6)$ process~$\wt{\eta}'$ around $+\infty$.   We say that such a $\pi/2$-cone excursion (on the interval~$I$) is orientation changing if it has the property that if $I' \supseteq I$ is any interval of time during which $(\wt{L},\wt{R})$ has a $\pi/2$-cone excursion, there exists an interval of time $I''$ during which it is having another $\pi/2$-cone excursion of the opposite orientation of that on $I$ with $I \subseteq I'' \subseteq I'$.  We note that orientation changing $\pi/2$-cone excursions whose time-interval contains $0$ cannot cluster. Indeed, this follows because $\wt{\eta}'$ is continuous and orientation changing $\pi/2$-cone excursions correspond to the times when $\wt{\eta}'$ makes a loop around $+\infty$ with the opposite orientation of the previous loop.  See Figure~\ref{fig::orientation_changing} for an illustration of these definitions in the context of an $\SLE_{\kappa'}(\kappa'-6)$ (but in this with the quantum cone parameterized by $\C$ with marked points at $0$ and $\infty$).  (However, an orientation changing $\pi/2$-cone excursion is typically preceded by a cluster of $\pi/2$-cone excursions with the opposite orientation.)

\begin{lemma}
\label{lem::conex_containing_0}
Let $\wt{\CI}$ be the collection of $\pi/2$-cone excursions of $(\wt{L},\wt{R})$ whose time interval contains $0$.  For $A_1,A_2 \in \wt{\CI}$, we say that $A_1$ comes before $A_2$ if the interval of time for $A_1$ contains that for $A_2$.  Let $\CI$ consist of those $\pi/2$-cone excursions in $\wt{\CI}$ which are orientation changing.  (As remarked above, the elements in $\wt{\CI}$ are in correspondence with the times in $\wt{\CT}$ and the elements in $\CI$ are in correspondence with the times in $\CT$.)  Let $A$ be the largest element in $\CI$ whose terminal displacement $d$ is at most $1$.  Given $d$, the conditional law of~$A$ is that of a $\pi/2$-cone excursion with terminal displacement~$d$ weighted by its length.  That is, the Radon-Nikodym derivative between the law of~$A$ given~$d$ and the law of a $\pi/2$-cone excursion with terminal displacement~$d$ is given by the length of $A$ times a normalization constant.
\end{lemma}
\begin{proof}
We begin by noting that the following is true.  Fix $L_0 > 0$ large.  We will eventually take a limit as $L_0 \to \infty$.  Pick $V$ in $[0,L_0]$ uniformly at random.  Then
\[ (\wt{L}_{V+t} - \wt{L}_V,\wt{R}_{V+t} - \wt{R}_V) \stackrel{d}{=} (\wt{L}_t,\wt{R}_t).\]
Therefore the $\pi/2$-cone excursions of $(\wt{L},\wt{R})$ which contain $V$ have the same law as those which contain $0$.

Fix $M_0 > 0$ large.  We will eventually take a limit as $M_0 \to \infty$ after taking a limit as $L_0 \to \infty$.  We begin by making two observations:
\begin{itemize}
\item The probability of the event that the outermost orientation changing $\pi/2$-cone excursion which contains $V$ of terminal displacement at most $M_0$ is distinct from the outermost orientation changing $\pi/2$-cone excursion of terminal displacement at most $1$ tends to $1$ as $M_0 \to \infty$.
\item For $M_0$ fixed, the probability that the outermost orientation changing $\pi/2$-cone excursion $E_0$ of terminal displacement at most $M_0$ containing $V$ is contained in the time-interval $[0,L_0]$ tends to $1$ as $L_0 \to \infty$.
\end{itemize}
Consequently, if we let $\wt{V}$ be chosen uniformly in the time interval for $E_0$ and let $\wt{E}$ be the outermost orientation changing $\pi/2$-cone excursion of terminal displacement at most $1$ contained in $E_0$, then we have that the total variation distance between the laws of $\wt{E}$ and $A$ (as in the statement of the lemma) tends to $0$ as $L_0 \to \infty$ and then $M_0 \to \infty$.

Let $(E_j)$ be the collection of outermost orientation changing $\pi/2$-cone excursions of terminal displacement at most $1$ in $E_0$.  We let $\wt{T}$ (resp.\ $T_j$) be the terminal displacement of $\wt{E}$ (resp.\ $E_j$).  We assume that the $E_j$ are ordered so that $T_1 \geq T_2 \geq \cdots$.  Fix an event $\CA$ such that $\p[E_j \in \CA]$ is positive for all $j$.  Fix $d,\delta > 0$ and let $I = [d,d+\delta]$.  We have that,
\begin{align}
   &   \p[ \wt{E} \in \CA \giv \wt{T} \in I]
= \frac{\p[ \wt{E} \in \CA, \wt{T} \in I]}{\p[\wt{T} \in I]}
  = \sum_j \frac{\p[ \wt{E} = E_j, E_j \in \CA, T_j \in I]}{\p[\wt{T} \in I]} \notag\\
=&  \sum_j \p[ \wt{E} = E_j \giv E_j \in \CA,  T_j \in I] \p[ E_j \in \CA \giv T_j \in I] \frac{\p[T_j \in I]}{\p[\wt{T} \in I]} . \label{eqn::et_form}
\end{align}
Fix $\zeta > 0$ and assume that $\CA$ is an event of the form $\wt{\CA} \cap \{\ell \leq l \leq \ell + \zeta\}$ where $l$ is the length of the time-interval for the $\pi/2$-cone excursion.  Then we note for a constant $c_0 > 0$ that
\begin{equation}
\label{eqn::e_give_length}
\p[ \wt{E} = E_j \giv E_j \in \CA, T_j \in I] = c_0 \ell(1+O(\zeta)).
\end{equation}
Combining~\eqref{eqn::et_form} and~\eqref{eqn::e_give_length}, we have that
\begin{equation}
\label{eqn::prob_ratio} \p[ \wt{E} \in \CA \giv \wt{T} \in I] = \ell(1+O(\zeta))  \left(c_0 \sum_j  \p[ E_j \in \CA \giv  T_j \in I] \frac{\p[ T_j \in I]}{\p[\wt{T} \in I]} \right).
\end{equation}
It is not difficult to see that if we take a limit as $L_0 \to \infty$, $M_0 \to \infty$, and $\delta \to 0$, we have that the term in the parenthesis on the right hand side of~\eqref{eqn::prob_ratio} converges to a constant times the probability of $\CA$ under the law on $\pi/2$-cone excursions with terminal displacement equal to $d$.  Thus taking a further limit as $\zeta \to 0$ implies the result.
\end{proof}

\begin{proof}[Proof of Proposition~\ref{prop::unexplored_region_qc_inf_to_0}]
We begin by picking $w \in U$ uniformly from the quantum area measure conditionally independently of everything else.  Let $\wh{\eta}'$ be as in the statement of Lemma~\ref{lem::unexplored_region_resample} and let $\wh{\eta}_w'$ be an $\SLE_{\kappa'}(\kappa'-6)$ process in $U$ starting from the same point on $\partial U$ as $\wh{\eta}'$ and coupled to be the same as $\wh{\eta}'$ until the first time that $+\infty$ and $w$ are separated and to evolve conditionally independently afterwards.  Let $U_0$ be the component of $U \setminus \wh{\eta}_w'$ which contains $+\infty$.   Let $\CF$ be the $\sigma$-algebra which is generated by the quantum area and boundary length of $(U_0,h)$ and the quantum areas and boundary lengths of all of the components of $U \setminus \wh{\eta}_w'$.  Lemma~\ref{lem::unexplored_region_resample} implies that conditional on the event that $\partial U_0$ is entirely contained in either the left or right side of $\wh{\eta}_w'$ (i.e., $\wh{\eta}_w'$ separates $+\infty$ from $\partial U$) and given $\CF$, we have that the quantum surface $(U_0,h)$ has the law of a quantum disk with its given quantum boundary length and area.  Let $\wh{\tau}$ be the first time $t \in \CT$ strictly after time $\tau$.  As the event that the component $V_0$ of $U \setminus \wh{\eta}'([0,\wh{\tau}])$ containing $+\infty$ is equal to $U_0$ is measurable with respect to the $\sigma$-algebra that we have conditioned on, it follows that the same holds when we further condition on the event that $V_0 = U_0$.  Throughout the rest of the proof, we shall assume that we have conditioned on this event.

We are now going to show that:
\begin{enumerate}
\item The conditional law of the quantum area of the quantum surface $(U_0,h)$ given its quantum boundary length is conditionally independent of the quantum areas and boundary lengths of all of the other components of $U \setminus \wh{\eta}_w'$ and that
\item The conditional law of the quantum area of $(U_0,h)$ given its quantum boundary length is equal to that of a quantum disk weighted by its quantum area with the given quantum boundary length.
\end{enumerate}
Recall that the quantum surface $(U,h)$ corresponds to the first $\pi/2$-cone excursion of the time-reversal $(\wt{L},\wt{R})$ of $(L,R)$ which contains $0$ with terminal displacement at most~$1$ such that the previous cone excursion has the opposite orientation.  The quantum surface $(U_0,h)$ corresponds to the next $\pi/2$-cone excursion of $(\wt{L},\wt{R})$ with the opposite orientation as for $(U,h)$ and the other components of $U \setminus \wh{\eta}_w'$ correspond to other $\pi/2$-cone excursions of $(\wt{L},\wt{R})$ which occur in intervals of time which are disjoint from the one which corresponds to $(U_0,h)$.  The first claim is obvious from this representation of the joint law of the quantum areas and boundary lengths.  The second claim follows from this representation together with Lemma~\ref{lem::conex_containing_0}.

We have shown so far that the conditional law of $(U_0,h)$ given its quantum boundary length is given by that of a quantum disk weighted by its quantum area.  The proof works verbatim if we take $U$ instead to correspond to the component containing $+\infty$ at the first time $t \in \CT$ that its quantum boundary length falls below $R > 0$ where $R > 0$ is fixed.  By taking $R > 0$ very large, the result follows by successively exploring the components which contain $+\infty$ until first reaching the one with quantum boundary length at most $1$.
\end{proof}

\subsection{Comparison of pinched quantum cones}
\label{subsec::comparison_of_pinched_quantum_cones}

In Section~\ref{sec::sphere_from_cone} we exhibited one way of constructing a unit area quantum sphere by pinching a unit of quantum area off a $\gamma$-quantum cone.  In both this subsection as well as the next, we will be interested in another way of pinching a unit area quantum sphere off a $\gamma$-quantum cone by conditioning on the event that an $\SLE_{\kappa'}(\kappa'-6)$ process cuts out a quantum disk with small quantum boundary length and area close to $1$.  The purpose of the following proposition, which is the main result of this subsection, is to compare this type of conditioning with that used in Section~\ref{sec::sphere_from_cone}.  We will make use of the same notation as in Section~\ref{sec::sphere_from_cone}.

\begin{proposition}
\label{prop::space_filling_almost_sphere}
Fix $\gamma \in (\sqrt{2},2)$ and suppose that $\CC = (\cyl,h,+\infty,-\infty)$ is a $\gamma$-quantum cone.    Let $\wt{\eta}'$ be an $\SLE_{\kappa'}(\kappa'-6)$ process from $-\infty$ to $+\infty$ sampled independently of $h$.  For each $\epsilon, \delta > 0$ there exists $C_0 \geq 1$ such that for all $r \in \R_-$ and $C \geq C_0$ the following is true.  Let $\CT$, $\zeta_{r,C}$, $U_{r,C}$, and $\ell_{r,C}$ be as described in the beginning of the section and let $A_{r,C}$ be the quantum area of $U_{r,C}$.  Let
\begin{equation}
\label{eqn::i_r_c}
I_{r,C} = [\tfrac{1}{2},1] \cdot C^{-1} e^{\gamma r/2}
\end{equation}
and let 
\begin{equation}
\label{eqn::f_r_c_eps}
F_{r,\epsilon,C} = \{ \ell_{r,C} \in I_{r,C},\ 1 \leq A_{r,C} \leq 1+\epsilon \}.
\end{equation}
Then
\begin{equation}
\label{eqn::e_r_eps_lbd}
\pr{ E_{r,\epsilon} \giv F_{r,\epsilon,C} } \geq 1-\delta.
\end{equation}
Moreover, for each fixed choice of $C > 1$ and $\epsilon > 0$ there exists $p_0 > 0$ such that
\begin{equation}
\label{eqn::f_c_r_lbd}
\pr{F_{r,\epsilon,C} \giv E_{r,\epsilon}} \geq p_0 \quad\text{for all}\quad r \in \R_-.
\end{equation}
\end{proposition}

Fix $C > 1$ and $r \in \R_-$.  We are going to prove Proposition~\ref{prop::space_filling_almost_sphere} by considering three different laws:
\begin{itemize}
\item The joint law $\conelaw$ of a $\gamma$-quantum cone $\CC = (\cyl,h,+\infty,-\infty)$ and an $\SLE_{\kappa'}(\kappa'-6)$ process $\wt{\eta}'$ in $\cyl$ from $-\infty$ to $+\infty$ sampled independently of $h$.
\item The law $\conelaw_F$ given by $\conelaw$ conditioned on $F_{r,\epsilon,C}$.
\item The law $\conelaw_G$ on pairs consisting of a quantum surface and a path whose joint law can be sampled from by:
	\begin{enumerate}
	\item Sampling a $\gamma$-quantum cone and an independent $\SLE_{\kappa'}(\kappa'-6)$ process conditioned on $F_{r,\epsilon,C}$ as in the definition of $\conelaw_F$.
	\item Resampling $\CD = (U_{r,C},h)$ (which on the event $F_{r,\epsilon,C}$ has boundary length $\ell_{r,C} \in I_{r,C}$ and quantum area $A_{r,C} \in [1,1+\epsilon]$) according to its $\conelaw$-conditional law given $\ell_{r,C}$ to yield $\wt{\CD}$.
	\end{enumerate} 
\end{itemize}
The laws $\conelaw$ and $\conelaw_F$ correspond to the laws considered in the statement of Proposition~\ref{prop::space_filling_almost_sphere}.  The law $\conelaw_G$ is an auxiliary law that will be used in the proof of Proposition~\ref{prop::space_filling_almost_sphere} (but does not appear in the statement).

The first step is to compare $\conelaw$ and $\conelaw_G$.

\begin{lemma}
\label{lem::rn_bound}
Let 
\begin{equation}
\label{eqn::z_r_e_c}
\CZ_{r,\epsilon,C} = \frac{d\conelaw_G}{d\conelaw}
\end{equation}
be the Radon-Nikodym derivative of $\conelaw_G$ with respect to $\conelaw$.  For each $\epsilon, \delta > 0$ there exists $K > 1$ such that
\[ \conelaw[ \CZ_{r,\epsilon,C} \geq K ] \geq 1-\delta \quad\text{for all}\quad r \in \R_- \quad\text{and}\quad C > 1.\] 
\end{lemma}

We will first need to collect Lemmas~\ref{lem::conditionally_independent}--\ref{lem::disk_boundary_length_conditioned} before completing the proof of Lemma~\ref{lem::rn_bound}.  Before we proceed with the details, we will give an overview of the different steps used to prove Lemma~\ref{lem::rn_bound}.  First, we will prove in Lemma~\ref{lem::conditionally_independent} that the Radon-Nikodym derivative $\CZ_{r,\epsilon,C}$ from Lemma~\ref{lem::rn_bound} is a function of the boundary length $\ell_{r,C}$, so that it suffices to compare the law of $\ell_{r,C}$ under $\conelaw_G$ and $\conelaw$.  This will involve controlling how likely it is that $\ell_{r,C} \in I_{r,C}$ under $\conelaw$.  This will be accomplished by observing in Lemma~\ref{lem::loops_density} that the $\log$ of the successive loop boundary lengths are i.i.d.\ and have a density with respect to Lebesgue measure, which reduces the problem to an overshoot analysis for random walks carried out in Lemma~\ref{lem::random_walk_density} and Lemmas~\ref{lem::disk_boundary_length_rn}, \ref{lem::disk_boundary_length_conditioned}.  The density of $\ell_{r,C}$ under $\conelaw_G$ is then given by weighting its density under $\conelaw$ by the probability that the quantum area of $U_{r,C}$ is in $[1,1+\epsilon]$ and noting that this probability is comparable for different values of $\ell_{r,C}$ in $I_{r,C}$.

\begin{lemma}
\label{lem::conditionally_independent}
The Radon-Nikodym derivative $\CZ_{r,\epsilon,C}$ in~\eqref{eqn::z_r_e_c} is equal to the Radon-Nikodym derivative between the laws of the boundary length $\ell_{r,C}$ under $\conelaw_G$ and $\conelaw$.  In particular, it depends only on $\ell_{r,C}$.
\end{lemma}
\begin{proof}
This follows because Proposition~\ref{prop::unexplored_region_qc_inf_to_0} implies that, under $\conelaw$, the quantum surface parameterized by $U_{r,C}$ is conditionally independent of the quantum surface parameterized by $\cyl \setminus U_{r,C}$ given $\ell_{r,C}$.  The same is therefore true under $\conelaw_F$ and $\conelaw_G$, which implies the result.
\end{proof}

We next collect the following elementary fact about random walks.

\begin{lemma}
\label{lem::random_walk_density}
Let $(X_j)$ be an i.i.d.\ sequence in $\R$ and assume that the law of $X_1$ has a density $f$ with respect to Lebesgue measure on $\R$ which is positive Lebesgue almost everywhere.  For each $j$, let $Y_j = X_1 + \cdots + X_j$.  Fix $A \subseteq \R$ with positive Lebesgue measure, let $\tau_A = \inf\{j \geq 1 : Y_j \in A\}$, and assume that $\p[ \tau_A < \infty] = 1$.  Then the law of $Y_{\tau_A}$ has a density with respect to Lebesgue measure on $A$ which is positive Lebesgue almost everywhere on $A$.
\end{lemma}
\begin{proof}
Let $f_k$ be the conditional density of $Y_k$ given $\tau_A > k$.  An elementary calculation implies that the conditional density of $(Y_{\tau_A-1},Y_{\tau_A})$ given $\tau_A=k$ is equal to 
\begin{equation}
\label{eqn::y_tau_cond_density}
g_k(x,y) = Z_x^{-1} \one_{A}(y) f(y-x) f_{k-1}(x) \quad\text{where}\quad Z_x = \int_{\R} \one_A(y) f(y-x) dy.
\end{equation}
Therefore the conditional density of $Y_{\tau_A}$ given $\tau_A=k$ is equal to
\begin{equation}
\label{eqn::y_tau_cond_density2}
g_k(y) =  \int_{\R} Z_x^{-1} \one_{A}(y) f(y-x)  f_{k-1}(x)dx.
\end{equation}
Letting $p_k = \p[ \tau_A =k+1]$ we thus have that the density of $Y_{\tau_A}$ is given by
\begin{equation}
\label{eqn::y_tau_density}
g(y) = \int_{\R} Z_x^{-1} \one_{A}(y) f(y-x)  \left(\sum_k f_{k-1}(x)p_{k-1} \right) dx.
\end{equation}
The result therefore follows from our assumptions on $f$.
\end{proof}

In the next series of lemmas, we will make use of the following notation.  Fix $R \in \R$.  We let $(U_j^R)_{j \in \Z}$ be the sequence of components of $\cyl \setminus \wt{\eta}'((-\infty,t])$ which contain $+\infty$ for $t \in \CT$ where we take $U_0^R$ to be the first such component with quantum boundary length at most $R$.  For each $j$, we let $Y_j^R$ be the $\log$ of the quantum boundary length of $U_j^R$ and we let $X_j^R = Y_j^R - Y_{j-1}^R$.  In the next two lemmas, we are going to check that the criteria of Lemma~\ref{lem::random_walk_density} hold for the sequence $(Y_j^R)_{j \in \N}$ where we will take $A$ to be an interval of the form $(-\infty,x]$ for $x \in \R$.

\begin{lemma}
\label{lem::loops_density}
The $(X_j^R)_{j \in \N}$ are i.i.d.\ and $X_1^R$ has a density with respect to Lebesgue measure on $\R$ which is positive Lebesgue almost everywhere.
\end{lemma}
\begin{proof}
That the $(X_j^R)_{j \in \N}$ are i.i.d.\ follows from Proposition~\ref{prop::unexplored_region_qc_inf_to_0}.  Alternatively, this follows since the corresponding quantum surfaces $(U_j^R,h)$ correspond to a certain subset of the $\pi/2$-cone excursions of $(\wt{L},\wt{R})$ whose time interval contains $0$.  Consequently, we just need to show that $X_1^R$ has an almost everywhere positive density with respect to Lebesgue measure on $\R$.  To see this, we pick a $(\cdot,\cdot)_\nabla$-orthonormal basis $(\phi_j)$ of $\CH_2(\cyl)$ consisting of $C_0^\infty(\cyl)$ functions such that $\phi_1|_{\cyl \setminus U_1^R} \equiv 0$ and $\phi_1|_{\partial U_j^R} \equiv 0$ for $j \neq 1$ and $\phi_1|_{\partial U_1^R} > 0$.  Since $\phi_1 \in \CH_2(\cyl)$, we can write $h = (h-\alpha_1 \phi_1) + \alpha_1 \phi_1$ so that $h - \alpha_1 \phi_1$ and $\alpha_1 \phi_1$ are conditionally independent given $\wt{\eta}'$ and the restriction of $h$ to $\cyl \setminus U_0^R$.  Let $\nu$ (resp.\ $\nu_1$) be the $\gamma$-LQG boundary length measure associated with $h$ (resp.\ $h-\alpha_1 \phi_1$).  Then we have that
\[ \frac{d\nu}{d\nu_1}(x) = e^{\gamma \alpha_1 \phi_1(x)/2}.\]
Consequently, we have that
\begin{equation}
\label{eqn::x_1_integral}
X_1^R = \log \left( \int_{\partial U_1^R} e^{\gamma \alpha_1 \phi_1(x) /2} d\nu_1(x) \right) - Y_0^R.
\end{equation}
Note that $X_1^R$ is smooth and strictly increasing viewed as a function of $\alpha_1$.  Moreover, $Y_0^R$ is determined by $h-\alpha_1 \phi_1$.  Applying the change of variables formula using the representation of $X_1^R$ as in~\eqref{eqn::x_1_integral} implies the result.
\end{proof}

Let $\ol{Y} = C e^{-\gamma r/2} \ell_{r,C}$.  Combining Lemma~\ref{lem::random_walk_density} and Lemma~\ref{lem::loops_density} we obtain the following.

\begin{lemma}
\label{lem::disk_boundary_length_rn}
The law of $\ol{Y}$ under $\conelaw$ has a density with respect to Lebesgue measure on $[0,1]$ which is almost everywhere positive on $[0,1]$ and which does not depend on $C$ or $r$.
\end{lemma}
\begin{proof}
Since the law of a $\gamma$-quantum cone is invariant under the operation of adding a constant to the field, it follows that the law of $\ol{Y}$ is independent of $r$ and $C$.  The result then follows in the case that we condition on the event $P_j^R$ that the first $j$ such that $Y_j^R \leq C^{-1} e^{\gamma r/2}$ is at least $1$ by combining Lemma~\ref{lem::random_walk_density} and Lemma~\ref{lem::loops_density}.  The result follows more generally using that the probability of $P_j^R$ tends to $1$ as $R \to \infty$.
\end{proof}

We are now going to establish an analog of Lemma~\ref{lem::disk_boundary_length_rn} for $\conelaw_F$.

\begin{lemma}
\label{lem::disk_boundary_length_conditioned}
The law of $\ol{Y}$ under $\conelaw_G$ has a density with respect to Lebesgue measure on $[1/2,1]$ which is almost everywhere positive in $[1/2,1]$ and bounded from above by a constant times the density of $\ol{Y}$ under $\conelaw$ on $[1/2,1]$.
\end{lemma}
\begin{proof}
Let $J \subseteq [1/2,1]$ be an open interval.  We have by Bayes' rule that 
\begin{equation}
\label{eqn::density_bayes}
\conelaw_G[\ol{Y} \in J] = \conelaw[ \ol{Y} \in J \giv F_{r,\epsilon,C} ] = \frac{\conelaw[ F_{r,\epsilon,C} \giv \ol{Y} \in J]}{\conelaw[ F_{r,\epsilon,C}]} \conelaw[ \ol{Y} \in J] := P_1 \conelaw[ \ol{Y} \in J].
\end{equation}
Lemma~\ref{lem::disk_boundary_length_rn} implies that $\conelaw[\ol{Y} \in [1/2,1] ] > 0$, which in turn implies that $P_1 \asymp 1$ as the conditional probability of $F_{r,\epsilon,C}$ given $\ol{Y}$ is comparable for each value of $\ol{Y} \in [1/2,1]$.  Indeed, this latter claim follows as the conditional probability of $F_{r,\epsilon,C}$ given $\ol{Y}$ is the probability of having a $\pi/2$-cone excursion of length in $[1,1+\epsilon]$ with terminal displacement $C^{-1} e^{\gamma r/2} \ol{Y}$.
\end{proof}

We can now complete the proof of Lemma~\ref{lem::rn_bound}.

\begin{proof}[Proof of Lemma~\ref{lem::rn_bound}]
By Lemma~\ref{lem::conditionally_independent}, we have that $\CZ_{r,\epsilon,C}$ is equal to the Radon-Nikodym derivative of the law of the quantum boundary length $\ell_{r,C}$ of $U_{r,C}$ under $\conelaw_G$ with respect to its law under $\conelaw$.  Thus the result follows by combining Lemma~\ref{lem::disk_boundary_length_rn} and Lemma~\ref{lem::disk_boundary_length_conditioned}.
\end{proof}

Let $u_{r,C}$ (resp.\ $v_{r,C}$) be the infimum (resp.\ supremum) of $\re(\partial U_{r,C})$.  In other words, $[u_{r,C},v_{r,C}] \times [0,2\pi]$ is the smallest annulus in $\cyl$ which contains $\partial U_{r,C}$.

\begin{figure}[ht!]
\begin{center}
\includegraphics[scale=0.85]{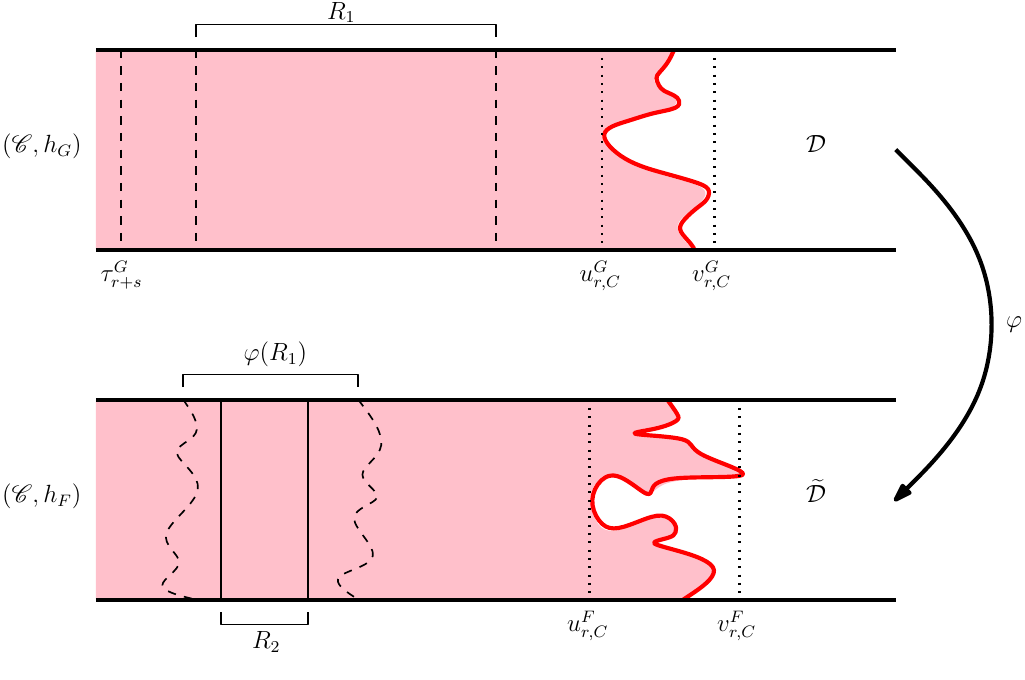}
\end{center}
\caption{\label{fig::bessel_brownian_reglue} Illustration of the setup of the proof of Lemma~\ref{lem::behavior_of_average_at_loop_mu_F}.  {\bf Top:} a surface sampled from the law $\conelaw_G$.  The red region is the part of the surface separated from $+\infty$ up until the first time that the $\SLE_{\kappa'}(\kappa'-6)$ process $\wt{\eta}'$ makes a loop with boundary length at most $C^{-1} e^{\gamma r/2}$.  {\bf Bottom:} the surface sampled from $\conelaw_F$ after resampling $\CD$ to yield $\wt{\CD}$. The superscripts $F$ and $G$ are used to indicate quantities associated with $h_F$ and $h_G$, respectively.}
\end{figure}

\newcommand{\av}{H}

\begin{lemma}
\label{lem::behavior_of_average_at_loop_mu}
Let $\av$ be the average of $h$ on $u_{r,C} + [0,2\pi i]$ in $\cyl$.  For any fixed $s \in \R_-$ we have that
\[ \conelaw[\av \geq r + s] \to 0 \quad\text{as}\quad C \to \infty \quad\text{uniformly}\quad\text{in}\quad r \in \R_-.\]
\end{lemma}
\begin{proof}
Suppose that $\wt{\av}$ has the law as described in the statement of the lemma for $r=0$ and $C=1$.  Then $\av = \wt{\av} + r - \tfrac{2}{\gamma} \log C$ has the law described in the statement of the lemma with the given value of $r$ and $C$.  The result then follows since $\conelaw[\wt{\av} < \infty] = 1$.
\end{proof}

We are now going to establish an analog of Lemma~\ref{lem::behavior_of_average_at_loop_mu} for $\conelaw_F$.

\begin{lemma}
\label{lem::behavior_of_average_at_loop_mu_F}
Let $\av$ be the average of $h$ on $u_{r,C} + [0,2\pi i]$.  For any fixed $s \in \R_-$ we have that
\[ \conelaw_F[ \av \geq r + s] \to 0 \quad\text{as}\quad C \to \infty \quad\text{uniformly}\quad\text{in}\quad r \in \R_-.\]
\end{lemma}
We will deduce Lemma~\ref{lem::behavior_of_average_at_loop_mu_F} from the corresponding result with $\conelaw_G$ in place of $\conelaw_F$, which in turn follows from Lemma~\ref{lem::rn_bound} and Lemma~\ref{lem::behavior_of_average_at_loop_mu}.  The idea of the proof will be to assume that we have samples from $\conelaw_F$ and $\conelaw_G$ coupled together on a common probability space so that one can transform from the former to the latter by cutting out the surface parameterized by $U_{r,C}$ and then gluing in one sampled from the $\conelaw$-conditional law given $\ell_{r,C}$.  This cutting and gluing operation involves applying a conformal transformation to the complement of $U_{r,C}$ in $\cyl$.  As we have just mentioned above, we have already proved that under $\conelaw_G$, the average considered in Lemma~\ref{lem::behavior_of_average_at_loop_mu_F} is very likely to be smaller than $r+s$ when $C$ is large.  The difficulty is that if we apply a conformal transformation to the circle $u_{r,C} + [0,2\pi i]$ in $\cyl$ we will not get another circle.  In order to deal with this challenge, we will first prove an intermediate result in which we have replaced the average over the circle $u_{r,C} + [0,2\pi i]$ in $\cyl$ with the worst-case average of the field when integrated against all smooth functions with compact support with support at most a bounded distance of $u_{r,C} + [0,2\pi i]$ (and with bounded derivative).  As we will see, averages of this type work well when applying conformal transformations because precomposing such a function with a conformal transformation leads to a function of the same type.  See Figure~\ref{fig::bessel_brownian_reglue} for an illustration of how we will implement this idea.  Before we give the proof of Lemma~\ref{lem::behavior_of_average_at_loop_mu_F}, we will need the following.

\begin{lemma}
\label{lem::behavior_of_sup_at_loop_mu}
Fix $a,b > 0$ and let $\Phi_{a,b}^1$ be the set of $C_0^\infty(\cyl)$ functions which are supported in the annulus $[0,a] \times [0,2\pi]$ with $\phi \geq 0$,  $\int \phi(x) dx = 1$, and $\| \phi' \|_\infty \leq b$.  Fix $r,s \in \R_-$ and let
\[  M(h) = \sup_{\phi \in \Phi_{a,b}^1} (h,\phi(\cdot+\tau_{r+s})).\]
We have that
\[ \conelaw[M(h) \geq r] \to 0 \quad\text{as}\quad s \to -\infty \quad\text{uniformly}\quad\text{in}\quad r \in \R_{-}.\]
\end{lemma}
\begin{proof}
We will first assume that $r =0$ and we take our quantum cone to be embedded so that the horizontal translation is so that the projection of $h$ onto $\CH_1(\cyl)$ first hits $r+s=s$ at $u=0$.  Then we have that the restriction of $h$ to $\cyl_+$ has the same law as the sum of the function $(\gamma-Q)\re(z)$ and a whole-plane GFF on $\cyl$ restricted to $\cyl_+$ with the additive constant fixed so that its average on $[0,2\pi i]$ is equal to $s$.  Therefore it suffices to prove the result in the case that $h$ is the sum of $(\gamma-Q)\re(z)$ and a whole-plane GFF on $\cyl$ with this normalization.  The argument explained in the paragraph just after \cite[Proposition~9.19]{dms2014mating} implies that, in this setting, $M(h)-s$ is a.s.\ finite with law which does not depend on $s$.  This proves the result for $r = 0$.  The result then follows for other values of $r$ because one can switch from the $r=0$ setting to the setting of general $r \in \R_-$ by adding $r$ to the field.
\end{proof}

\begin{proof}[Proof of Lemma~\ref{lem::behavior_of_average_at_loop_mu_F}]
In the proof, we will be working with samples from both~$\conelaw_G$ and from~$\conelaw_F$, coupled onto a common probability space.  We will add an extra subscript $F$ or $G$ in order to clarify to which law each will be associated.

Suppose that $(\cyl,h_G,+\infty,-\infty)$ and $\eta_G'$ are sampled from the law~$\conelaw_G$.  We think of the embedding of the surface into $\cyl$ as taking place in two steps.  Namely, we first embed the surface as usual for quantum cones as described in Section~\ref{subsubsec::quantum_cones} so that the horizontal translation is such that the projection of $h_G$ onto $\CH_1(\cyl)$ first hits $0$ at $u=0$ and then we adjust the horizontal translation so that $u_{r,C,G} = 0$.

Suppose that $(\cyl,h_F,+\infty,-\infty)$ is a sampled from $\conelaw_F$.  We take the embedding into $\cyl$ to be the same as for $(\cyl,h_G)$ described above (but with $u_{r,C,F}$ in place of $u_{r,C,G}$).  We assume that $(\cyl,h_F,+\infty,-\infty)$ and $(\cyl,h_G,+\infty,-\infty)$ are coupled together as in the definition of $\conelaw_G$.  That is, we can transform from the former to the latter by cutting out the quantum surface  $(U_{r,C,F},h_F)$ and then gluing in a conditionally independent copy sampled from the $\conelaw$-conditional law with given boundary length $\ell_{r,C,F}$ (without any conditioning on its quantum area).  Let $\varphi$ be the conformal transformation which corresponds to this cutting and gluing operation.  That is, $\varphi \colon \cyl \setminus U_{r,C,G} \to \cyl \setminus U_{r,C,F}$ is a conformal transformation fixing $-\infty$ with $h_G  = h_F \circ \varphi + Q\log|\varphi'|$ in $\cyl \setminus U_{r,C,G}$.

It follows from Lemma~\ref{lem::rn_bound} and Lemma~\ref{lem::behavior_of_average_at_loop_mu} that the $\conelaw_G$-probability that the average of $h_G$ on $[0,2\pi i]$ is at most $r$ tends to $1$ as $C \to \infty$ uniformly in $r \in \R_-$.  Our goal is to use this to deduce that the same is true with $h_F$ in place of $h_G$.  The difficulty is that $\varphi([0,2\pi i])$ is not a circle in $\cyl$.  To deal with this, we will instead work with the field integrated against smooth test functions rather than on circles.  In order to do this, we will need to make sure we are working with test functions which are on a domain in which the distortion which arises by applying $\varphi$ is bounded and this is where we will make use of the distortion bound Lemma~\ref{lem::conformal_cylinder_hulls}.

Let $C_1, C_2$ be the constants from Lemma~\ref{lem::conformal_cylinder_hulls}.  By Lemma~\ref{lem::rn_bound}, the modulus of continuity for the projection of $h_G$ onto $\CH_1(\cyl)$ is the same as that for a standard Brownian motion as this is also true under $\conelaw$.  Fix $s \in \R_-$.  It thus follows that $\tau_{r+s,G} \leq - C_1 - 3C_2$ with $\conelaw_G$-probability tending to $1$ as $C \to \infty$ uniformly in $r \in \R_-$.  Let $\Phi_{a,b}^1$ and $M(h)$ be as in the statement of Lemma~\ref{lem::behavior_of_sup_at_loop_mu} with $a = 3 C_2$ and $b > 1$ very large but fixed.  Then Lemma~\ref{lem::rn_bound} and Lemma~\ref{lem::behavior_of_sup_at_loop_mu} together imply that for each $\delta > 0$ there exists $s_0 \in \R_-$ such that for each $s \leq s_0$ there exists $C_0 > 1$ such that $C \geq C_0$ implies
\begin{equation}
\label{eqn::cone_good}
\conelaw_G[ M(h) \geq r] \leq \delta \quad\text{uniformly in}\quad r \in \R_-.
\end{equation}

Let $R_1$ be the annulus in $\cyl$ which is given by $\tau_{r+s} + [0,3C_2] \times [0,2\pi]$.

Lemma~\ref{lem::conformal_cylinder_hulls} implies that $\varphi(R_1)$ contains a non-empty rectangle $R_2 = [p,q] \times [0,2\pi]$ in $\cyl$.  Let $\phi_{R_2} \in C_0^\infty(\cyl)$ with $\phi_{R_2} \geq 0$ be a function which is rotationally symmetric about $\pm \infty$ with $\int \phi_{R_2}(x) dx = 1$ and support contained in $R_2$ and let $\phi = |(\varphi^{-1})'|^2 \phi_{R_2} \circ \varphi^{-1}$.  Then $\phi \in C_0^\infty(\cyl)$ is supported in $R_1$, $\phi \geq 0$, and $\int \phi(x) dx = 1$.  Moreover, by taking $b > 1$ large enough relative to $C_1,C_2$, we have that $\| \phi' \|_\infty \leq b$.  Consequently, it follows from~\eqref{eqn::cone_good} that we can make the probability that $(h_G,\phi) \leq r$ as close to $1$ as we like by choosing $s \in \R_-$ small enough and $C > 1$ large enough, uniformly in $r \in \R_-$.  Since $Q\log|(\varphi^{-1})'|$ restricted to $R_2$ is bounded by a deterministic constant it thus follows that we can make the probability of the event that $(h_F,\phi_{R_2}) \leq r$ as close to $1$ as we like by choosing $s \in \R_-$ small enough and $C > 1$ large enough, uniformly in $r \in  \R_-$.  As $\phi_{R_2}$ is rotationally symmetric about $\pm \infty$ with support contained in $R_2$, it follows that on the event that $(h_F,\phi_{R_2}) \leq r$ we have that the projection of $h_F$ onto $\CH_1(\cyl)$ first hits $r$ to the left of the right side of $R_2$.
\end{proof}

\begin{proof}[Proof of Proposition~\ref{prop::space_filling_almost_sphere}]
By the definition of $u_{r,C}$, under $\conelaw_F$ the quantum area of $U_{r,C}$ is in $[1,1+\epsilon]$.  Recall from the statement of Lemma~\ref{lem::behavior_of_average_at_loop_mu_F} that $\av$ is the average of $h$ on the vertical line $u_{r,C} + [0,2\pi i]$.  Lemma~\ref{lem::behavior_of_average_at_loop_mu_F} thus implies that under $\conelaw_F$, we can make the probability that $\tau_r \leq u_{r,C}$ arbitrarily close to $1$ by making $C$ large.  Thus to finish proving that~\eqref{eqn::e_r_eps_lbd} holds, we need to explain why the quantum area of the part of $\cyl \setminus U_{r,C}$ which is to the right of $\tau_r$ tends to $0$ in probability under $\conelaw_F$ provided $r \in \R_-$ is small enough (uniformly in the choice of $C$).  This holds since the amount of quantum area in $\cyl$ which is to the right of $\tau_r$ tends to $0$ in probability as $r \to -\infty$ under $\conelaw$ hence also under $\conelaw_G$ by Lemma~\ref{lem::rn_bound}.  Therefore the claim for $\conelaw_F$ holds because of the natural coupling between $\conelaw_F$ and $\conelaw_G$.

We now turn to explain why~\eqref{eqn::f_c_r_lbd} holds.  Let $E_{r,\epsilon}$ and $E_{r,\beta}'$ be as in Lemma~\ref{lem::two_e_events}.  We have that
\begin{align*}
  \p[ F_{r,\epsilon,C} \giv E_{r,\epsilon}]
&= \frac{\p[ F_{r,\epsilon,C}, E_{r,\epsilon}]}{\p[ E_{r,\epsilon}]}
 \geq \frac{\p[ F_{r,\epsilon,C}, E_{r,\epsilon}, E_{r,\beta}']}{\p[ E_{r,\epsilon}]}
 = \frac{\p[ F_{r,\epsilon,C}, E_{r,\epsilon} \giv E_{r,\beta}']}{\p[ E_{r,\epsilon}] / \p[ E_{r,\beta}']}	
\end{align*}
Lemma~\ref{lem::two_e_events} implies that the denominator is bounded above by a constant $c_0 > 0$ which depends only on $\epsilon$ and $\beta$.  Therefore it suffices to give a lower bound for $\p[ F_{r,\epsilon,C}, E_{r,\epsilon} \giv E_{r,\beta}']$, but this is easy to see from the definitions.  This proves~\eqref{eqn::f_c_r_lbd}.
\end{proof}

\subsection{Extra conditioning does not affect the limiting law}
\label{subsec::asymptotic_mixing}

Proposition~\ref{prop::space_filling_almost_sphere} implies that we can think of conditioning on $F_{r,\epsilon,C}$ as taking place in two steps: we first condition the surface on $E_{r,\epsilon}$ and then condition the result on the uniformly positive probability event $F_{r,\epsilon,C}$.  The purpose of this section is to argue that this second conditioning step does not have a significant effect on the resulting law.  This, in turn, will complete the proof of Theorem~\ref{thm::sphere_equivalent_constructions}.

\begin{lemma}
\label{lem::f_delta_far_left}
Fix $r \in \R$.  Suppose that $(\cyl,h,+\infty,-\infty)$ is a $\gamma$-quantum cone.  Let $\tau_r$ be as in Proposition~\ref{prop::cone_to_sphere} and $v_{r,C}$ as defined just before the statement of Lemma~\ref{lem::behavior_of_average_at_loop_mu}.  Then we have that
\[ \pr{ v_{r,C} \geq \tau_r + S \giv E_{r,\epsilon}, F_{r,\epsilon,C} } \to 0 \quad\text{as}\quad r \to -\infty \quad\text{and then}\quad S \to \infty \]
at a rate which only depends on $\epsilon,C$.
\end{lemma}
\begin{proof}
Let $\av$ be the average of $h$ on $u_{r,C} + [0,2\pi i]$.  Then we have that
\begin{align*}
   \p[ \av \geq r + M \giv E_{r,\epsilon}, F_{r,\epsilon,C} ]
&= \frac{\p[\av \geq r + M, E_{r,\epsilon} \giv F_{r,\epsilon,C} ]}{\p[ E_{r,\epsilon} \giv F_{r,\epsilon,C}]}.
\end{align*}
Using~\eqref{eqn::e_r_eps_lbd} of Proposition~\ref{prop::space_filling_almost_sphere}, we see that the probability is at most a constant times $\p[ \av \geq r + M \giv F_{r,\epsilon,C}]$.  By Lemma~\ref{lem::behavior_of_average_at_loop_mu_F}, we know that the probability of this event decays to $0$ as $M \to \infty$ uniformly in $r \in \R_-$.  Therefore it suffices to show that for each fixed $M$ we have that 
\[ \pr{ v_{r,C} \geq \tau_r + S, \av \leq r + M \giv E_{r,\epsilon}, F_{r,\epsilon,C} } \to 0 \quad\text{as}\quad r \to -\infty \quad\text{and then}\quad S \to \infty.\]

Using~\eqref{eqn::f_c_r_lbd} of Proposition~\ref{prop::space_filling_almost_sphere} in the inequality, we see that there exists a constant $c_0 > 0$ depending only on $C$ and $\epsilon$ such that
\begin{align*}
   \p[ v_{r,C} \geq \tau_r + S, \av \leq r + M \giv E_{r,\epsilon}, F_{r,\epsilon,C} ]
&= \frac{\p[ v_{r,C} \geq \tau_r + S, \av \leq r + M, F_{r,\epsilon,C} \giv E_{r,\epsilon}]}{\p[ F_{r,\epsilon,C} \giv E_{r,\epsilon}]}\\
&\leq c_0 \p[ v_{r,C} \geq \tau_r + S, \av \leq r + M, F_{r,\epsilon,C} \giv E_{r,\epsilon}].
\end{align*}
Let $E_{r,\beta}'$ be as in Lemma~\ref{lem::cond_x_large}.   We also have that
\begin{align*}
   &\p[ v_{r,C} \geq \tau_r +S,\av \leq r + M, F_{r,\epsilon,C} \giv E_{r,\epsilon}]\\
 \leq& \p[ v_{r,C} \geq \tau_r + S,\av \leq r + M, F_{r,\epsilon,C}, E_{r,\beta}' \giv E_{r,\epsilon}] + \p[(E_{r,\beta}')^c \giv E_{r,\epsilon}].
\end{align*}
By Lemma~\ref{lem::two_e_events}, the second term above can be made arbitrarily small by choosing $\beta > 0$ to be sufficiently large.  By Bayes' rule, the first term on the right hand side above is bounded from above by
\begin{align*}
  \frac{\p[ v_{r,C} \geq \tau_r + S,\av \leq r+ M, F_{r,\epsilon,C}, E_{r,\epsilon} \giv E_{r,\beta}']}{\p[ E_{r,\epsilon} \giv E_{r,\beta}']}.
\end{align*}
By Lemma~\ref{lem::two_e_events}, the ratio is bounded from above by a constant times the numerator.  Therefore it suffices to show that $\p[ v_{r,C} \geq \tau_r + S, \av \leq r + M, F_{r,\epsilon,C} \giv E_{r,\beta}'] \to 0$ as $r \to -\infty$ then $S \to \infty$.  We will deduce this from Lemma~\ref{lem::cond_x_large}.

Let $\sigma_\beta$ be the first $u \geq \tau_r$ that the projection $X$ of $h$ onto $\CH_1(\cyl)$ hits $\gamma^{-1} \log \beta^{-1}$.  We consider two possibilities: either $u_{r,C} \leq \sigma_\beta$ or $u_{r,C} > \sigma_\beta$.  We will first argue that the second possibility is very unlikely, so that we can exclude it.  We note that the event that $\av \leq r + M$, $u_{r,C} \geq \sigma_\beta$, and $F_{r,\epsilon,C}$ all hold is contained in the event that the quantum area in the part of $\cyl$ which is to the right of where $X$ first hits $r+M$ after $\sigma_\beta$ is at least $1$.  It is easy to see that the conditional probability of this given $E_{r,\beta}'$ tends to $0$ as $r \to -\infty$ as claimed.

We are thus left to bound
\[ \p[ v_{r,C} \geq \tau_r + S, \av \leq r + M, u_{r,C} \leq \sigma_\beta \giv E_{r,\beta}'].\]
Note that on $\av \leq r+M$ and $u_{r,C} \leq \sigma_\beta$ (and given $E_{r,\beta}'$), we have that $v_{r,C}$ is (non-strictly) to the left of the first loop of $\wt{\eta}'$ around $+\infty$ which is completed after $X$ attains its infimum on $[\tau_r,\sigma_\beta]$.  Since the conditional law of $X$ given $E_{r,\beta}'$ in $[\tau_r,\sigma_\beta]$ is that of a Brownian motion with positive drift starting from $r$ and run until the first time it hits $\gamma^{-1} \log \beta^{-1}$, it follows that the length of time it takes after time $\tau_r$ for the infimum to be attained has an exponential tail.  Thus the result easily follows.

\end{proof}

\begin{lemma}
\label{lem::f_delta_looks_like_sphere}
Suppose that $r \in \R$, $\epsilon > 0$, and $C > 1$.  Suppose that $(\cyl,h,+\infty,-\infty)$ is sampled from $\conelaw_F$.  Then we have that the quantum surfaces $(\cyl_+ + v_{r,C},h)$ converge weakly to the unit area quantum sphere (in the same sense as in Proposition~\ref{prop::cone_to_sphere}) when we take a limit as $r \to -\infty$, $C \to \infty$, and then $\epsilon \to 0$.
\end{lemma}
\begin{proof}
Proposition~\ref{prop::space_filling_almost_sphere} implies that the conditional law of $(\cyl_+ + v_{r,C},h)$ given $F_{r,\epsilon,C}$ is close in total variation to its conditional law given both $F_{r,\epsilon,C}$ and $E_{r,\epsilon}$ as the conditional probability of the latter given the former can be made arbitrarily close to $1$.  Lemma~\ref{lem::f_delta_far_left} implies that, conditional on these two events, we can fix a value of $S$ large so that with high probability we have that $v_{r,C} \leq \tau_r + S$, given both $F_{r,\epsilon,C}$ and $E_{r,\epsilon}$.  Whether this occurs is determined by the values of the field and the path in $(-\infty,\tau_r + S] \times [0,2\pi]$.  Therefore the result follows by applying the final assertion of Proposition~\ref{prop::cone_to_sphere}.

\end{proof}

\begin{proof}[Proof of Theorem~\ref{thm::sphere_equivalent_constructions}, $\gamma \in (\sqrt{2},2)$]

Suppose that $(\cyl,h,+\infty,-\infty)$ and $\eta'$ are sampled from $\conelaw_F$.  Then Lemma~\ref{lem::f_delta_looks_like_sphere} implies that the laws of the surfaces $(\cyl_+ + v_{r,C},h)$ converge weakly to the law of the unit area quantum sphere when we take a limit as $r \to -\infty$, then $C \to \infty$, and then $\epsilon \to 0$.  Moreover, by Proposition~\ref{prop::cone_to_sphere_brownian_limit} we have that the joint law of $(L,R)$ restricted to the interval $J_{r,C}$ of time in which $\eta'$ is filling the component $U_{r,C}$ converges weakly with respect to the topology of uniform convergence as $r \to -\infty$, then $C \to \infty$, and then $\epsilon \to 0$ to that of a correlated Brownian loop of unit length.

We next claim that $\eta'$ converges to a space-filling $\SLE_{\kappa'}$ from $-\infty$ to $-\infty$ which is independent of the limiting surface.  We know that for each fixed $r$, $C$, and $\epsilon$ that the conditional law of $\eta'$ in $J_{r,C}$ given $h$, $U_{r,C}$, and $\eta'(\zeta_{r,C})$ is that of a space-filling $\SLE_{\kappa'}$ in $U_{r,C}$ from $\eta'(\zeta_{r,C})$ to $\eta'(\zeta_{r,C})$.  Moreover, Lemma~\ref{lem::f_delta_far_left} implies that $v_{r,C} \to -\infty$ in probability as $r \to -\infty$ and then $C \to \infty$.  In particular, this implies that the law of the GFF used to generate $\eta'$ restricted to $U_{r,C}$ converges to a whole-plane GFF on $\cyl$ in the sense that for each fixed $x \in \R$ its restriction to the half-infinite cylinder $[x,\infty) \times [0,2\pi]$ converges in the total variation to the law of the corresponding restriction of a whole-plane GFF (see \cite[Proposition~2.10]{MS_IMAG4}).  This implies the claim.

Suppose that $(\cyl,h,-\infty,+\infty)$ has the law of a unit area quantum sphere and that~$\eta'$ is an independent space-filling $\SLE_{\kappa'}$ process from $-\infty$ to $-\infty$ reparameterized by quantum area.  Let $(L,R)$ be the quantum boundary length of the left/right side of~$\eta'$.  We have shown so far that $(L,R)$ evolves as a correlated Brownian loop.  To finish proving the result, we need to show that $(L,R)$ a.s.\ determines $(\cyl,h,-\infty,+\infty)$ and $\eta'$ up to a conformal transformation.  We let $\wt{\eta}'$ be the $\SLE_{\kappa'}(\kappa'-6)$ process which arises by reparameterizing $\eta'$ by capacity as seen from $+\infty$, let $\wt{\CT}$ be the set corresponding to~$\wt{\eta}'$ as defined in the beginning of the section, and let $\tau_\epsilon$ be the first $t \in \wt{\CT}$ that the component $U_\epsilon$ of $\cyl \setminus \wt{\eta}'([0,t])$ which contains $+\infty$ has quantum boundary length at least $\epsilon$.  Then it follows from the first part of the proof that the conditional law of the quantum surface $(U_\epsilon,h)$ given its quantum boundary length, quantum area, and $\wt{\eta}'(\tau_\epsilon)$ is the same as in the setting of an exploration of a $\gamma$-quantum cone by an independent $\SLE_{\kappa'}(\kappa'-6)$ process.  Also, as in the setting of a $\gamma$-quantum cone, we have that $(U_\epsilon,h)$ corresponds to a $\pi/2$-cone excursion $A_\epsilon$ of the time-reversal of $(L,R)$.  By the limiting construction, we know that the joint law of $(U_\epsilon,h)$ and $A_\epsilon$ is the same as in the setting of a $\gamma$-quantum cone, conditional on quantum boundary length and area.  It therefore follows that both $(U_\epsilon,h)$ and $\eta'$ restricted to the interval of time in which it is filling $U_\epsilon$ are a.s.\ determined by $A_\epsilon$ since we know this to be true in the case of a $\gamma$-quantum cone.  The claim follows since Lemma~\ref{lem::conformal_cylinder_hulls} implies that the effect of resampling $(\cyl \setminus U_\epsilon,h)$ given $(U_\epsilon,h)$ tends to $0$ as $\epsilon \to 0$.
\end{proof}

We are now going to explain how a variant of the argument given above handles the case that $\gamma \in (0,\sqrt{2}]$.

\begin{figure}[ht!]
\begin{center}
\includegraphics[scale=0.85]{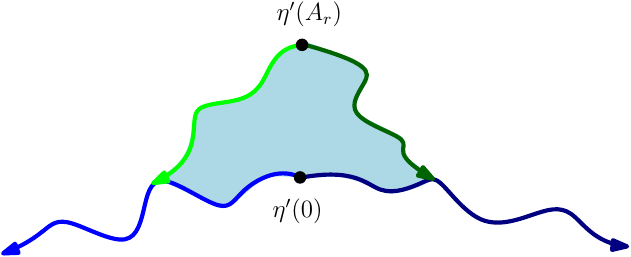}	
\end{center}
\caption{\label{fig::kappa_bigger_than_8_illustration} Illustration of the setup for the proof of Theorem~\ref{thm::sphere_equivalent_constructions} for $\gamma \in (0,\sqrt{2}]$ and $\kappa' \geq 8$.  Shown is a space-filling $\SLE_{\kappa'}$ process $\eta'$ on an $\gamma$-quantum cone parameterized by $\C$ normalized so that $\eta'(0) = 0$.  The blue (resp.\ dark blue) path represents the left (resp.\ right) side of $\eta'$ stopped at time $0$ and the green (resp.\ dark green) show the left (resp.\ right) side of $\eta'$ stopped at the time $A_r$ up until merging with the blue (resp.\ dark blue) path.  The light blue region is $\eta'([0,A_r])$.  The left/right boundary length process $(L,R)$ is normalized so that $L_0 = R_0 = 0$ and $L_{A_r}$ (resp.\ $R_{A_r}$) is equal to the length of the green (resp.\ dark green) path minus the length of blue (resp.\ dark blue) path up until they merge and $L_t,R_t$ are defined similarly for general $t$ values.  In particular, for $t \in [0,A_r]$, we have that $L_t$ (resp.\ $R_t$) is at least $-1$ times the length of the blue (resp.\ dark blue) path up until merging with the green (resp.\ dark green) path.}
\end{figure}

\begin{proof}[Proof of Theorem~\ref{thm::sphere_equivalent_constructions}, $\gamma \in (0,\sqrt{2}\text{$]$}$]
We suppose that $(\cyl,h,+\infty,-\infty)$ has the law of a $\gamma$-quantum cone and that $\eta'$ is a space-filling $\SLE_{\kappa'}$ on $\cyl$ from $-\infty$ to $-\infty$ sampled independently of $h$ and then reparameterized by quantum area.  We normalize time so that $\eta'(0) = +\infty$.  Let $A_r$ be the first time $t$ (if it exists) that the quantum boundary length of the boundary of $U_t = \eta'([0,t])$ is equal to $e^{\gamma r/2}$.  Let $F_{r,\epsilon}$ be the event that $A_r$ (which is equal to the quantum area of $U_{A_r}$) is in $[1,1+\epsilon]$.  Let $(L,R)$ denote the change in the quantum boundary length of the left/right side of $\eta'$ relative to time $0$ so that $L_0 = R_0 = 0$.  Conditional on $A_r$ and $(L_{A_r},R_{A_r})$, and the length $\ol{L}$ (resp.\ $\ol{R}$) of the left (resp.\ right) side of $\eta'((-\infty,0])$ up until it merges with the left (resp.\ right) side of $\eta'([A_r,\infty))$ (see Figure~\ref{fig::kappa_bigger_than_8_illustration} for an illustration when the quantum cone is parameterized by $\C$) we have that $(L,R)$ in $[0,A_r]$ evolves as a (correlated) two-dimensional Brownian motion starting from $(0,0)$, conditioned to be in $(\ol{L},\ol{R}) + \R_+^2$, and conditioned to take the value $(L_{A_r},R_{A_r})$ at time~$A_r$.

This process is characterized by a certain Markovian resampling property.  Namely, if we condition on part of its initial and terminal segments, then the conditional law of the remaining part is that of a (correlated) two-dimensional Brownian motion with given starting and ending points conditioned to stay in $(\ol{L},\ol{R}) + \R_+^2$.  Conditional on $F_{r,\epsilon}$, we have that $(\ol{L},\ol{R})$ converges to $0$ as $r \to -\infty$ and we have that $A_r \to 1$ as $r \to -\infty$ and $\epsilon \to 0$.  Arguments analogous to those in Section~\ref{sec::brownian_excursions} imply that the law of $(L,R)$ in $[0,A_r]$ conditional on $F_{r,\epsilon}$ converges weakly with respect to the topology of uniform convergence when we take a limit as $r \to -\infty$ and then $\epsilon \to 0$ to that of a correlated Brownian loop of length $1$.  Indeed, this follows because any subsequential limit will satisfy the Markovian resampling property that conditional on an initial and terminal segment, the conditional law of the remaining part is that of a (correlated) two-dimensional Brownian motion with the given starting and ending points conditioned to stay in $\R_+^2$.  Moreover, this resampling property uniquely characterizes the limiting law.

The same argument given in the case that $\gamma \in (\sqrt{2},2)$ implies that the conditional law of the quantum surface $(U_{A_r},h)$ given $F_{r,\epsilon}$ converges as $r \to -\infty$ and $\epsilon \to 0$ to that of the unit area quantum sphere and $\eta'$ in $[0,A_r]$ converges to an independent space-filling $\SLE_{\kappa'}$ process from $-\infty$ to $-\infty$.

Suppose that $(\cyl,h,+\infty,-\infty)$ has the law of a unit area quantum sphere, $\eta'$ is an independent $\SLE_{\kappa'}$ process from $-\infty$ to $-\infty$ reparameterized by quantum area, and $(L,R)$ is the quantum length of the left/right side of $\eta'$.  To finish the proof, it is left to show that $(L,R)$ a.s.\ determines $(\cyl,h,+\infty,-\infty)$ and $\eta'$.  Fix $\delta > 0$.  Then the conditional law of the quantum surface parameterized by $\cyl \setminus (\eta'([0,\delta]) \cup \eta'([1-\delta,1]))$ given $(L_\delta,R_\delta)$ and $(L_{1-\delta},R_{1-\delta})$ is the same as the corresponding conditional law in the case of a $\gamma$-quantum cone.  In the latter setting, we know that $(L,R)$ restricted to $[\delta,1-\delta]$ a.s.\ determines both the surface and $\eta'|_{[\delta,1-\delta]}$ (viewed as a path in the surface)  Thus arguing as in the end of the proof in the case that $\gamma \in (\sqrt{2},2)$, the result follows because the distortion estimates from Section~\ref{subsec::maps} imply that the effect of resampling the part of the surface which is parameterized by $\eta'([0,\delta])$ and $\eta'([1-\delta,1])$ a.s.\ tends to $0$ as $\delta \to 0$.
\end{proof}

\section{L\'evy excursion construction for $\gamma=\sqrt{8/3}$}
\label{sec::levy_construction}

In this section, we are going to complete the proof of Theorem~\ref{thm::pure_sphere_equivalent_constructions}.  We will then show that the conditional law of the tip of the $\SLE_6$ grown up to a given time conditional on the region it has cut off from its target point (as a path-decorated quantum surface) is uniformly distributed on the hull boundary according to the quantum length measure and identify the law of the surface component containing the target point of the $\SLE_6$.  We will end by completing the proof of Theorem~\ref{thm::pure_sphere_radon_nikodym_derivative}.

One notion which will be important for this section is the so-called \emph{quantum natural time}, constructed in \cite{dms2014mating}.  It is a quantum analog of the natural parameterization first defined in \cite{ls2011natural}.  We now review its definition and basic construction.  Its existence is implied by \cite[Theorem~1.15]{dms2014mating}, which considers a quantum wedge of weight $2-\tfrac{\gamma^2}{2}$.  This is the quantum surface that one obtains by considering a $\gamma$-quantum cone $(\C,h,0,\infty)$ decorated by an independent space-filling $\SLE_{\kappa'}$ curve $\eta'$ from $\infty$ to $\infty$ normalized so that $\eta'(0) = 0$ and then taking the quantum surface parameterized by $\eta'([0,\infty))$.  If $\kappa' \in (4,8)$, then this surface is not homeomorphic to $\h$ and is described by a Poissonian string of beads (and this is the case were interested in) while for $\kappa' \geq 8$ it is homeomorphic to~$\h$.  If we assume we are in the former setting, then we can draw on top of the surface a concatenation of independent $\SLE_{\kappa'}(\kappa'/2-4;\kappa'/2-4)$ processes, one for each bead.  This divides the surface parameterized by $\eta'([0,\infty))$ into a collection of quantum disks and it is shown in \cite[Theorem~1.15]{dms2014mating} that the path decorated surfaces can be represented as a gluing of a pair of independent forested lines.  A forested line is a forest of stable looptrees which is defined out of a $\kappa'/4$-stable L\'evy process.  In particular, there is a natural time parameterization associated with a forested line which corresponds to the time parameterization of the underlying L\'evy process.  Recall that for a stable L\'evy process, one can recover its time parameterization as a measurable function of its jumps.  Indeed, one can recover the time elapsed by counting the number of jumps between $\epsilon$ and $2\epsilon$, normalizing by an appropriate power of $\epsilon$, and then sending $\epsilon \to 0$ (we will remind the reader of this in the proof of Theorem~\ref{thm::pure_sphere_radon_nikodym_derivative} in additional detail).  In the case of $\SLE_{\kappa'}$ on top of an LQG surface, the jumps correspond to quantum disks which the $\SLE$ separates from its target point.

We begin by collecting the following which implies that we can make sense of $\lsphere$ conditioned on having quantum area in a given interval $[a,b]$ with $0 < a < b \leq +\infty$.

\begin{lemma}
\label{lem::finite_mass}
Let $A$ be the total quantum area of the quantum disks associated with a sample produced from $\lsphere$.  For each $a > 0$ we have that $\lsphere[A > a] < \infty$.
\end{lemma}
We will identify $\lsphere[A > a]$ explicitly in Section~\ref{subsec::comparison_of_infinite_sphere_measures} below.
\begin{proof}[Proof of Lemma~\ref{lem::finite_mass}]
Suppose for contradiction that there exists $a > 0$ such that $\lsphere[A > a] = \infty$.  Let $X_t$ be a $3/2$-stable L\'evy process with only upward jumps and let $I_t$ be its running infimum.  Assume that $X_0 = 1$ and let $\tau = \inf\{t \geq 0 : X_t = 0\}$.  Given $X|_{[0,\tau]}$, we sample a family $\CA_1$ of conditionally independent quantum disks indexed by the jumps of $X|_{[0,\tau]}$ where the boundary length of the disk associated with a given jump has boundary length equal to the size of the jump.  Let $A_1$ be the total quantum area of the quantum disks in $\CA_1$.  Note that the jumps of $(X-I)|_{[0,\tau]}$ are equal to those of $X|_{[0,\tau]}$.  By the Poissonian representation of the excursions that $X-I$ makes from $0$, it follows from the assumption that $\lsphere[A > a] = \infty$ that $\p[A_1 = \infty] = 1$.  It therefore suffices to show that $\p[A_1 < \infty] = 1$.  One can either see this using a direct computation or, alternatively, noting by \cite[Corollary~10.2]{dms2014mating} and \cite[Chapter~VII, Theorem~18]{bertoin96levy} that $A_1$ is equal in law to the amount of quantum area cut out by a whole-plane $\SLE_6$ from a $\sqrt{8/3}$-quantum cone stopped at the last time that the quantum length of its outer boundary hits $1$.
\end{proof}

\subsection{Proof of Theorem~\ref{thm::pure_sphere_equivalent_constructions}}
\label{subsec::pure_sphere_proof}

We know from \cite[Corollary~10.2]{dms2014mating} that the quantum boundary length of the complementary component containing $-\infty$ of a whole-plane $\SLE_6$ process $\ol{\eta}'$ from $+\infty$ to $-\infty$ drawn on top of a $\sqrt{8/3}$-quantum cone $(\cyl,h,+\infty,-\infty)$ evolves as a totally asymmetric $3/2$-stable L\'evy process with negative jumps and conditioned to be non-negative (see \cite[Chapter~VII.3]{bertoin96levy} for more on this process).  Moreover, the components viewed as quantum surfaces separated by $\ol{\eta}'$ from $-\infty$ are conditionally independent quantum disks with quantum boundary length equal to the size of the corresponding jump made by the boundary length process and the ordered sequence of quantum disks together with whether they were surrounded on the left or right side of $\ol{\eta}'$ and the marked point corresponding to the last point on the disk boundaries visited by $\ol{\eta}'$ a.s.\ determine both $\ol{\eta}'$ and the quantum cone.  This implies that we can sample from the law of a $\sqrt{8/3}$-quantum cone decorated with an independent whole-plane $\SLE_6$ using the following steps:
\begin{itemize}
\item Sample a totally asymmetric $3/2$-stable L\'evy process $\ol{X}$ with only downward jumps conditioned to be non-negative.
\item Given $\ol{X}$, sample a collection of conditionally independent quantum disks $\CD$ with boundary lengths equal to the size of the jumps made by $\ol{X}$ and orient each such disk either clockwise or counterclockwise by the result of an i.i.d.\ fair coin flip.
\item For each quantum disk in $\CD$, sample a uniformly chosen point in its boundary from its quantum boundary measure.
\end{itemize}

\begin{figure}[ht!]
\begin{center}
\includegraphics[scale=0.85]{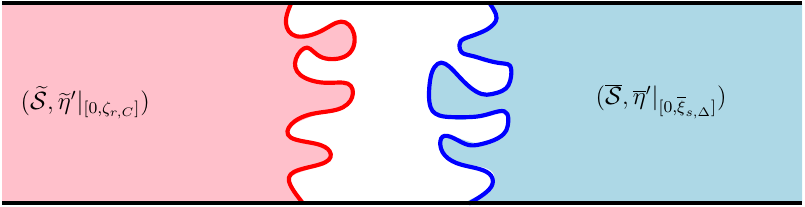}
\end{center}
\caption{\label{fig::bessel_levy_equivalence} Illustration of the setup of the last four steps of the proof of Theorem~\ref{thm::pure_sphere_equivalent_constructions}. Shown in light red is the path-decorated surface $(\wt{\CS},\wt{\eta}'|_{[0,\zeta_{r,C}]})$ which consists of the part of $\CS$ and $\wt{\eta}'$ up until the time $\zeta_{r,C}$ defined in Section~\ref{subsec::comparison_of_pinched_quantum_cones}.  Shown in light blue is the path decorated surface $(\ol{\CS},\ol{\eta}'|_{[0,\ol{\xi}_{s,\Delta}]})$ which consists of the part of $\CS$ and $\ol{\eta}'$ up until time~$\ol{\xi}_{s,\Delta}$, the first time $t$ that the quantum boundary length of the hull of $\ol{\eta}'|_{[0,t]}$ as seen from $-\infty$ falls below $e^{\gamma s/2}$ after first hitting $\Delta$.  On the event that $\wt{\CS}$ and $\ol{\CS}$ are disjoint, the three surfaces are conditionally independent ($\ol{\CS}$, $\wt{\CS}$, and the part of $\CS$ not in $\ol{\CS},\wt{\CS}$).}
\end{figure}

Fix $\delta, \Delta > 0$ and $s \in \R_-$.  We let $\ol{\xi}_{s,\Delta}$ be the first time $t$ that $\ol{X}$ falls below $e^{\gamma s/2}$ after exceeding $\Delta$ and let $\ol{Q}_{s,\Delta,\delta}$ be the event that $\ol{\xi}_{s,\Delta} < \infty$ and the area separated from $-\infty$ by $\ol{\eta}'$ is in $[1-\delta,1]$.  See Figure~\ref{fig::bessel_levy_equivalence} for an illustration of the setup.  We want to show that the surface $\ol{\CS}$ separated from $-\infty$ by $\ol{\eta}'|_{[0,\ol{\xi}_{s,\Delta}]}$ conditional on $\ol{Q}_{s,\Delta,\delta}$ converges to that of the unit area quantum sphere decorated with an independent $\SLE_6$ from $-\infty$ to $+\infty$ when we take a limit as $s \to -\infty$, $\Delta \to 0$, and then $\delta \to 0$.  This is a consequence of the following sequence of observations.  We let $F_{r,\epsilon,C}$ and $\wt{\eta}'$ be as in Section~\ref{subsec::comparison_of_pinched_quantum_cones}, where we take $\wt{\eta}'$ and $\ol{\eta}'$ to be time-reversal of each other.
\begin{enumerate}
\item The conditional probability of $\ol{Q}_{s,\Delta,\delta}$ given $F_{r,\epsilon,C}$ tends to $1$ when we take a limit first as $r \to -\infty$, then $\epsilon \to 0$, then $s \to -\infty$, then $\Delta \to 0$, and then $\delta \to 0$.
\item The conditional probability that $\ol{\eta}'|_{[0,\ol{\xi}_{s,\Delta}]}$ is disjoint from $\wt{\eta}'|_{[0,\zeta_{r,C}]}$ given $F_{r,\epsilon,C}$ and $\ol{Q}_{s,\Delta,\delta}$ tends to $1$ when we take the same sequence of limits as in the previous item.
\item Let $\wt{\CS}$ be the surface separated from $+\infty$ by $\wt{\eta}'|_{[0,\zeta_{r,C}]}$.  Then the path decorated surfaces $(\wt{\CS},\wt{\eta}'|_{[0,\zeta_{r,C}]})$ and $(\ol{\CS},\ol{\eta}'|_{[0,\ol{\xi}_{s,\Delta}]})$ are conditionally independent given their boundary lengths on $F_{r,\epsilon,C}$, $\ol{Q}_{s,\Delta,\delta}$, and the event that $\wt{\CS}$ and $\ol{\CS}$ are disjoint.  In particular, the conditional law of $(\ol{\CS},\ol{\eta}'|_{[0,\ol{\xi}_{s,\Delta}]})$ given its boundary length and conditional on $\ol{Q}_{s,\Delta,\delta}$ does not change when we further condition on $F_{r,\epsilon,C}$ and that the two surfaces are disjoint.
\item By the proof of Theorem~\ref{thm::sphere_equivalent_constructions} given in Section~\ref{sec::bessel_brownian_constructions} we have in the setting of the previous item that the surface parameterized by the complement of $\wt{\CS}$ converges to a unit area quantum sphere when we take a limit as $r \to -\infty$ and $\epsilon \to 0$.  As remarked above, the conditional probability that $\ol{\CS}$ is contained inside of the complement of $\wt{\CS}$ tends to $1$ when we take the further limits as $s \to -\infty$, $\Delta \to 0$, and then $\delta \to 0$.  In particular, when we take this sequence of limits, the limit of $\ol{\eta}'|_{[0,\ol{\xi}_{s,\Delta}]}$ yields a path on top of a unit area quantum sphere.  Recall that we took $\wt{\eta}'$ and $\ol{\eta}'$ to be time-reversals of each other.  From the construction, $\wt{\eta}'$ in the limit yields an independent whole-plane $\SLE_6$ on top of the quantum sphere.  Therefore $\ol{\eta}'$ also yields an independent whole-plane $\SLE_6$ on top of the quantum sphere by the reversibility of whole-plane $\SLE_6$.  By Proposition~\ref{prop::levy_conditioned_infinite} stated and proved below, the joint law of its boundary length process and the quantum disks it separates from $\pm \infty$ converges to the time-reversal of the type of $3/2$-stable L\'evy excursion described in the statement of Theorem~\ref{thm::pure_sphere_equivalent_constructions}.  \qed
\end{enumerate}

\begin{proposition}
\label{prop::levy_conditioned_infinite}
Let $\ol{X}$, $\ol{Q}_{s,\Delta,\delta}$, and $\ol{\xi}_{s,\Delta}$ be as earlier in this section.  For each $t \geq 0$, let $\ol{\CA}_t$ be the collection of marked, oriented quantum disks cut out by $\ol{\eta}'|_{[0,t]}$.  The joint law of the time-reversal of $\ol{X}|_{[0,\ol{\xi}_{s,\Delta}]}$ and $\ol{\CA}_{\ol{\xi}_{s,\Delta}}$ cut out by $\ol{\eta}'|_{[0,\ol{\xi}_{s,\Delta}]}$ conditional on $\ol{Q}_{s,\Delta,\delta}$ converges as $s \to -\infty$, $\Delta \to 0$, and then $\delta \to 0$ weakly to $\lsphere$ conditioned on the total quantum area of the quantum disks being equal to $1$ (as described just before the statement of Theorem~\ref{thm::pure_sphere_equivalent_constructions}).
\end{proposition}
\begin{proof}
\cite[Chapter~VII, Theorem~18]{bertoin96levy} implies that the conditional law of the time-reversal of $\ol{X}|_{[0,\ol{\xi}_{s,\Delta}]}$ and $\ol{\CA}_{\ol{\xi}_{s,\Delta}}$ given $\ol{\xi}_{s,\Delta} < \infty$ converges to the law given by $\lsphere$ conditioned on the maximum of the L\'evy excursion being at least $\Delta$ when we take a limit as $s \to -\infty$.  Therefore the conditional law of the time-reversal of $\ol{X}|_{[0,\ol{\xi}_{s,\Delta}]}$ and $\ol{\CA}_{\ol{\xi}_{s,\Delta}}$ given $\ol{Q}_{s,\Delta,\delta}$ converges as $s \to -\infty$ to the law given by $\lsphere$ conditioned on the maximum of the L\'evy excursion being at least $\Delta$ and the quantum area of the quantum disks being between $1-\delta$ and $1$.  Taking a further limit as $\Delta \to 0$ yields $\lsphere$ conditioned on the quantum area being in $[1-\delta,1]$, so the result follows by sending $\delta \to 0$.
\end{proof}

\subsection{Comparison of Bessel and L\'evy measures}
\label{subsec::comparison_of_infinite_sphere_measures}

We are now going to complete the proof of Theorem~\ref{thm::pure_sphere_radon_nikodym_derivative}.  Before we do so, we first collect the following result which gives the distribution of the quantum area associated with the Bessel construction of the unit area quantum sphere.

\begin{proposition}
\label{prop::bessel_area_decay}
There exists a constant $c_0 > 0$ such that the density of the distribution of the quantum area $A$ of a $\sqrt{8/3}$-quantum sphere sampled from $\bsphere$ with respect to Lebesgue measure on $\R_+$ is given by $c_0 A^{-3/2}$.
\end{proposition}
\begin{proof}
We first recall that the Bessel dimension used in the construction is given by $\delta = 4-8/\gamma^2 = 1$.  Consequently, the result follows from \cite[Proposition~4.18]{dms2014mating}.
\end{proof}

\begin{proof}[Proof of Theorem~\ref{thm::pure_sphere_radon_nikodym_derivative}]
We assume throughout that $\gamma=\sqrt{8/3}$.  We are going to derive the result using scaling.  Namely, we consider the transformation given by multiplying the quantum area of the surface by a constant.  This, in turn, corresponds to adding a constant to the field used to parameterize the surface.  If we add the constant $C$ to this field, then the total quantum area is multiplied by $e^{\gamma C}$.

We are now going to determine how the quantum natural time of the $\SLE_6$ path is scaled under this operation.  First, we suppose that $X$ is a $3/2$-stable process with only upward jumps.  Fix $T > 0$.  We recall that there exists a constant $c_0 > 0$ such that the following is true.  Let $\Lambda$ be a \ppp\ on $[0,T] \times \R_+$ with intensity measure $c_0 dt \otimes u^{-5/2} du$ where $dt$ (resp.\ $du$) denotes Lebesgue measure on $[0,T]$ (resp.\ $\R_+$).  Then $\Lambda$ is equal in law to the set of jump time/size pairs for $X$ in $[0,T]$, i.e., the set of pairs $(t,X_{t^-} - X_t)$ with $t \in [0,T]$ and $X_{t^{-}} - X_t \neq 0$.  (Recall that $X$ has downward jumps so that $X_{t^-} - X_t \geq 0$.)  Using this fact, we can a.s.\ determine the amount of time that such a $3/2$-stable L\'evy process has been run if we only observe its jumps in that time interval (and not the length of the interval).  Indeed, for $0 < x_1 < x_2$, we let $N(x_1,x_2)$ be the number of jumps made by $X$ in that interval with size contained in $[x_1,x_2]$.  Then the almost sure limit as $j \to \infty$ of
\begin{equation}
\label{eqn::levy_time}
\frac{N(e^{-j-1},e^{-j})}{ \tfrac{2}{3} c_0 e^{3/2 j} (e^{3/2}-1)}
\end{equation}
is equal to the length of the interval.

Using the same principle, we can a.s.\ determine the length of a $3/2$-stable L\'evy excursion if we only observe its jumps.  Suppose that we have a sample $(\CS,x,y)$ produced from $\lsphere$.  Let $T$ be the length of the L\'evy excursion used to generate the surface.  Now suppose that we add $C$ to the field used to parameterize the surface and let $T_C$ be the length of the resulting L\'evy excursion.  Then the quantum boundary lengths of each of the components cut out by the $\SLE_6$ path are scaled by the factor $e^{\gamma C/2}$.  Therefore the number of quantum disks with boundary length between $e^{-j-1}$ and $e^{-j}$ after adding $C$ to the field is given by $N(e^{-j-1-\gamma C/2},e^{-j-\gamma C/2})$.  If we divide this quantity by $ \tfrac{2}{3} c_0 e^{3/2 j}(e^{3/2}-1)$ as in~\eqref{eqn::levy_time} and then send $j \to \infty$, the almost sure limit that we obtain is $T_C = e^{3\gamma C/4} T$.

Letting $\lsphere( \giv  t)$ be the probability measure under which we have conditioned the length of the L\'evy excursion to be equal to $t$, we therefore have that
\begin{equation}
\label{eqn::lsphere_scaling}
\lsphere( A \in [a,a+\epsilon] \giv t) = \lsphere(A \in [1,1+\epsilon/a] \giv a^{-3/4} t).
\end{equation}
Thus,
\begin{align*}
     \lsphere(A \in [a,a+\epsilon])
&= c \int_0^\infty \lsphere( A \in [a,a+\epsilon] \giv t) t^{-5/3} dt \quad\text{(Section~\ref{subsubsec::stable_levy}; $c=c_{3/2}$)}\\
&= c\int_0^\infty \lsphere(A \in [1,1+\epsilon/a] \giv a^{-3/4} t) t^{-5/3} dt \quad\text{(by~\eqref{eqn::lsphere_scaling})}\\
&= c a^{-1/2} \int_0^\infty \lsphere( A \in [1,1+\epsilon/a] \giv t) t^{-5/3} dt\\
&= a^{-1/2} \lsphere(A \in [1,\epsilon/a])
\end{align*}
Dividing both sides by $\epsilon$ and sending $\epsilon \to 0$ implies that the density of $A$ with respect to Lebesgue measure on $\R_+$ is given by a constant times $A^{-3/2}$.  (The constant is explicitly given by the density of $A$ with respect to Lebesgue measure evaluated at $1$.)  Combining this with Proposition~\ref{prop::bessel_area_decay} proves~\eqref{eqn::cm_m_rn}.
\end{proof}

\subsection{Tip is uniformly random and law of the unexplored region}
\label{subsec::unexplored_law}

We are now going to show that if one performs an $\SLE_6$ exploration on a $\gamma=\sqrt{8/3}$ unit area quantum sphere $(\CS,x,y)$ from $x$ to $y$ where $x,y$ are sampled independently from the quantum measure on $\CS$ then:
\begin{itemize}
\item The location of the tip of the path is uniformly distributed according to the quantum boundary measure on its hull relative to its target point $y$ and
\item The conditional law of the complement of the hull given its boundary length is that of a quantum disk weighted by its quantum area.
\end{itemize}
These results will be important in \cite{qlebm} in which we construct a version of $\QLE(8/3,0)$ on a quantum sphere.  We will then deduce from this the analogous results in the setting of a $\sqrt{8/3}$-quantum cone.

Both of these statements hold if we explore the $\SLE_6$ up to any stopping time $\tau$ for the filtration $(\CF_t)$ which is defined as follows.  For each $t$, we let $\CF_t$ be the $\sigma$-algebra which is generated by the collection of quantum disks that $\wt{\eta}'|_{[0,t]}$ has separated from $y$, each marked by the last point on their boundary visited by $\wt{\eta}'$ and oriented by the direction in which $\wt{\eta}'$ has traced their boundary.  Here, we assume that $\wt{\eta}'$ has the quantum natural time parameterization.  (Note that the quantum boundary length of the component of $\CS \setminus \wt{\eta}'([0,t])$ is $\CF_t$-measurable for each deterministic $t$ since the boundary length process is absolutely continuous to a $3/2$-stable L\'evy process and $\CF_t$ determines the jumps made up to time $t$.  The same holds if $t$ is replaced by a stopping time.)

\begin{proposition}
\label{prop::unexplored_region}
Fix $t > 0$ and suppose that $(\CS,x,y)$ is distributed according to $\lsphere$ conditioned on the event that the amount of quantum natural time for $\wt{\eta}'$ to go from $x$ to $y$ is at least $t$.  Let $\tau$ be a stopping time for $(\CF_t)$ as defined just above with $\p[\tau \geq t] = 1$.  Then the conditional law of the component $U_y$ of $\CS \setminus \wt{\eta}'([0,\tau])$ containing $y$, viewed as a quantum surface, given its boundary length is that of a quantum disk with the given boundary length weighted by its quantum area.  Moreover, $\wt{\eta}'(\tau)$ is uniformly distributed from the quantum boundary measure on $\partial U_y$.
\end{proposition}
\begin{proof}
This follows from an argument which is very similar to that used to prove Proposition~\ref{prop::unexplored_region_qc_inf_to_0}, though the present case is actually simplified because we are working with a finite volume surface rather than an infinite volume surface.  In particular, we can resample the target point $y$ by picking another independent point from the quantum measure.  We are going to prove the result first for deterministic times $s > t$, then deduce that the result holds for stopping times which take on dyadic rational values, and then finally by continuity deduce the result for general stopping times.

Fix $s > t$ deterministic.  Let $w$ be a point on $\CS$ picked from the quantum area measure independently of everything else.  Using an argument which is analogous to that in the proof of Proposition~\ref{prop::unexplored_region_qc_inf_to_0}, it is easy to see that if $w$ lands in one of the bubbles that $\wt{\eta}'$ separates from $y$, then the conditional law of that bubble given its boundary length is that of a quantum disk weighted by its quantum area with the given boundary length.  By the symmetry of $w$ and $y$, it thus follows that if we run $\wt{\eta}'$ until the first time $\tau$ that it separates $y$ from $w$, then the conditional law of the component of $\CS \setminus \wt{\eta}'([0,\tau])$ containing $y$ given $\CF_\tau$ is that of a quantum disk weighted by its quantum area with the given boundary length.  This holds even if we condition further on which component contains $w$.  The claim thus follows for the deterministic time $s$ by conditioning on the event that $w$ is contained in a component which is separated by $\wt{\eta}'$ from $y$ in the time interval $[s,s+\epsilon]$ for $\epsilon > 0$ fixed and then taking a limit as $\epsilon \to 0$.

Now suppose that $\tau$ is a stopping time for $(\CF_t)$ with $\p[\tau \geq t] = 1$.  Fix $k \in \N$ and let $\tau_k$ be the first multiple of $2^{-k}$ which occurs after time $\tau$.  Then the conditional law is as above at the time $\tau_k$.  Finally, we note that the law of a quantum disk weighted by its quantum area with given boundary length is a continuous function of the boundary length since we can change the boundary length by adding a constant to the field.  Therefore the form of the conditional law at the time $\tau$ follows by taking a limit as $k \to \infty$.

\end{proof}

We are now going to use Proposition~\ref{prop::unexplored_region} to deduce the conditional law of the region with infinite quantum area when we explore a $\sqrt{8/3}$-quantum cone with an independent whole-plane $\SLE_6$ parameterized by quantum natural time.  We will not give an explicit description of this law, but rather describe it in terms of certain resampling properties.

\begin{proposition}
\label{prop::unexplored_quantum_cone}
Suppose that $(\cyl,h,+\infty,-\infty)$ is a $\sqrt{8/3}$-quantum cone and that $\wt{\eta}'$ is a whole-plane $\SLE_6$ process from $+\infty$ to $-\infty$ sampled independently of $h$ and then reparameterized by quantum natural time.  Let $\tau$ be an a.s.\ finite stopping time for the filtration generated by the quantum surfaces separated by $\wt{\eta}'$ from $-\infty$.  Let $\varphi$ be the unique conformal transformation from the component of $\cyl \setminus \wt{\eta}'([0,\tau])$ containing $-\infty$ to $\cyl_-$ with $\varphi(-\infty)  = -\infty$ and $\varphi'(-\infty) > 0$, let $\wt{h} = h \circ \varphi^{-1} + Q \log|(\varphi^{-1})'|$, and let $b$ be the quantum boundary length of $\partial \cyl_-$ assigned by $\nu_{\wt{h}}$.  Then the law of $\wt{h}$ possesses the following properties:
\begin{enumerate}
\item\label{it:cone_resamp_finite} For each $r > 0$, the conditional law of $\wt{h}$ in the annulus $[-r,0] \times [0,2\pi] \subseteq \cyl_-$ given its values on $\cyl_- - r$ is that of a GFF with free boundary conditions on $\partial \cyl_-$ and Dirichlet boundary conditions on $\partial \cyl_- - r$ given by those of $\wt{h}$, conditioned so that the quantum boundary length of $\partial \cyl_-$ is equal to $b$.
\item\label{it:cone_resamp_infinite} The conditional law of $\wt{h}$ given its values on $\partial \cyl_-$ is equal to that of the sum of the function $z \mapsto (\gamma-Q)\Re(z)$ and a GFF on $\cyl_-$ with boundary conditions on $\partial \cyl_-$ which agree with $\wt{h}$.
\end{enumerate}
\end{proposition}
\begin{proof}
\noindent{\it Step 1: Local behavior of a quantum sphere at its marked points is described by quantum cone.} Suppose that $(\CS,x,y) = (\cyl,h,-\infty,+\infty)$ is a doubly-marked quantum sphere sampled from $\lsphere = c_{\mathrm {LB}} \bsphere$ conditioned so that its mass is at least $1$.  We are first going to explain why the local behavior of the surface near $+\infty$ (hence also near $-\infty$ by symmetry) is described by a $\gamma$-quantum cone.  Since the amount of mass that a quantum sphere contains has a density with respect to Lebesgue measure, it follows that $(\CS,x,y)$ a.s.\ has mass which is strictly larger than $1$.  Suppose that we take the horizontal translation for the embedding into $\cyl$ so that $\cyl_-$ has quantum area equal to~$1$.  Then the conditional law of $h$ in $\cyl_+$ given its values in $\cyl_-$ is that of a GFF with the given boundary values on $\partial \cyl_+$ plus the function $(\gamma-Q) \re(z)$.  By \cite[Proposition~4.13]{dms2014mating}, the following is true.  Let $\tau_C$ be the smallest $u \geq 0$ so that the average of $h+C$ on $u + [0,2\pi i]$ is equal to $0$.  Then $h(\cdot + \tau_C)+C$ converges as $C \to \infty$ to a $\gamma$-quantum cone.

\noindent{\it Step 2: Resampling property for quantum sphere conditioned on both quantum natural time and area.} Suppose that we fix $\epsilon > 0$ and we now let $(\CS,x,y) = (\cyl,h,-\infty,+\infty)$ be a doubly-marked quantum sphere sampled from $\lsphere = c_{\mathrm {LB}} \bsphere$.  Let $\wt{\eta}'$ be an independent whole-plane $\SLE_6$ from $+\infty$ to $-\infty$ parameterized by quantum natural time.  We assume that $(\CS,x,y)$ is conditioned so that the amount of quantum natural time for $\wt{\eta}'$ to go from $+\infty$ to $-\infty$ is at least $\epsilon$  (without conditioning on the quantum area for the moment).  By Proposition~\ref{prop::unexplored_region}, we know that the quantum surface parameterized by the component of $\cyl \setminus \eta'([0,\epsilon])$ containing $-\infty$ is a quantum disk weighted by its quantum area.  Let $\varphi$ be the unique conformal transformation from this component to $\cyl_-$ which fixes and has positive derivative at $-\infty$.  By \cite[Proposition~A.9]{dms2014mating}, we know that the conditional law of $\wt{h} = h \circ \varphi^{-1} + Q\log|(\varphi^{-1})'|$ in $[-r,0] \times [0,2\pi]$ given its values in $(-\infty,-r] \times [0,2\pi]$ is that of a GFF with the given Dirichlet boundary conditions on $-r + [0,2\pi i]$ and free boundary conditions on $[0,2\pi i]$ conditioned to have quantum boundary length equal to that of the component of $\cyl \setminus \eta'([0,\epsilon])$ containing $-\infty$.

Now suppose that we condition further on the quantum area of $(\CS,x,y)$ being at least~$1$.  On the event that the quantum area assigned to $(-\infty,-r] \times [0,2\pi]$ by $\wt{h}$ is at least $1$ we have the same form for the conditional law for the field in $[-r,0] \times [0,2\pi]$ as above, even with this extra conditioning.  We also note that the probability of this event tends to $1$ as $\epsilon \to 0$.

\noindent{\it Step 3: Completion of proof of property~\ref{it:cone_resamp_finite}.}  We now want to combine the observations made in Step 1 and Step 2 in order to complete the proof of property~\ref{it:cone_resamp_finite} in the case of the $\gamma$-quantum cone.  We suppose that $(\CS,x,y) = (\cyl,h,-\infty,+\infty)$ is a doubly marked quantum sphere with quantum area conditioned to be at least $1$ and let $\wt{\eta}'$ be an independent whole-plane $\SLE_6$ from $+\infty$ to $-\infty$ which is subsequently reparameterized by quantum natural time.  Recall from the proof of Proposition~\ref{prop::bessel_area_decay} that replacing $h$ with $h+C$ has the effect of multiplying the quantum natural time by $e^{3 \gamma C/4}$.  So, if we run $\wt{\eta}'$ for $\epsilon e^{-3\gamma C/4}$ units of quantum natural time measured using $h$ then the amount of quantum natural time elapsed measured using $h+C$ is $\epsilon$.  Suppose that we condition further on the quantum natural time required by $\wt{\eta}'$ to go from $+\infty$ to $-\infty$ to be at least $\epsilon e^{-3\gamma C/4}$.  We note that the probability of this event tends to $1$ as $C \to \infty$.  Therefore in the limit as $C \to \infty$, the first observation implies that (after horizontally translating) we obtain a $\gamma$-quantum cone decorated by an independent whole-plane $\SLE_6$.  Combining this with the second observation implies that property~\ref{it:cone_resamp_finite} holds for a quantum cone.

\noindent{\it Step 4: Completion of the proof of property~\ref{it:cone_resamp_infinite}.}  Suppose that $(\cyl,h,+\infty,-\infty)$ is a $\gamma$-quantum cone and $\wt{\eta}'$ is an independent whole-plane $\SLE_6$ from $+\infty$ to $-\infty$, reparameterized according to quantum natural time.  Let $\wt{h}$ be the field which describes the unexplored region after mapping back to $\cyl_-$ as in the statement of the proposition.  We first claim that the average of $\wt{h}$ on $\partial \cyl_- - r$ multiplied by $r^{-1}$ converges to $\gamma-Q$ in probability as $r \to \infty$.  We are going to deduce this from the corresponding property of a $\gamma$-quantum cone parameterized by $\cyl$.  Since we will be applying a conformal mapping, it will be more convenient to consider the field integrated against a smooth test function and then make use of the argument used to prove Lemma~\ref{lem::behavior_of_average_at_loop_mu_F}.  In order to accomplish this, we will consider two annuli on our quantum cone which differ by a horizontal translation of $r$ along $\cyl$ and then use that they are transformed into approximate annuli by the conformal map using Lemma~\ref{lem::conformal_cylinder_hulls}.

To make the idea sketched above more precise, we begin by letting $C_1,C_2$ be the constants from Lemma~\ref{lem::conformal_cylinder_hulls}.  Let $v_0 = \inf\{ \re(z) : z \in \wt{\eta}'([0,\epsilon])\}$, $R_1 = -C_1 + [v_0,v_0 - 3 C_2] \times [0,2\pi]$, and $R_2 = -r + [v_0,v_0-3C_2] \times [0,2\pi]$.  It follows from the definition of a $\gamma$-quantum cone that the difference in the average of the field on $R_2$ and $R_1$ multiplied by $r^{-1}$ converges in probability to $\gamma-Q$ as $r \to \infty$.  Lemma~\ref{lem::conformal_cylinder_hulls} implies that $\varphi(R_i)$ for $i=1,2$ contains a non-empty rectangle~$\wt{R}_i$ in~$\cyl_-$ where the distance of~$\wt{R}_1$ from~$\partial \cyl_-$ is bounded and the distance between $\wt{R}_2$ and $\wt{R}_1$ is given by $r$ up to a constant error which does not grow with $r$.  It follows from the argument used to prove Lemma~\ref{lem::behavior_of_average_at_loop_mu_F} that the difference in the average of $\wt{h}$ on $\wt{R}_2$ and $\wt{R}_1$ multiplied by $r^{-1}$ converges in probability to $\gamma-Q$ as $r \to \infty$.

Fix $u > 0$ and consider $z \in \cyl_-$ with $\re(z) = -u$.  Let $\phi_0$ (resp.\ $\phi_r)$ be the function which is harmonic in $\cyl_-$ (resp.\ $\cyl_+ - r$) with boundary values given by those of $\wt{h}$ on $\partial \cyl_-$ (resp.\ $\partial \cyl_+ - r$).  Finally, let $\phi$ be the function which is harmonic in the annulus $[-r,0] \times [0,2\pi]\subseteq \cyl_-$ with boundary values given by those of $h$ on $\partial \cyl_-$ and $\partial \cyl_- - r$.  Then we know that
\[ \phi(z) = \left(1-\frac{u}{r}\right) \phi_1(z) + \frac{u}{r} \phi_2(z) + o(1) \quad\text{as}\quad r \to \infty.\]
We note that $\phi_1(z)$ converges as $r \to \infty$ to the harmonic extension of the boundary values of $h$ from $\partial \cyl_-$ to $\cyl_-$ and the previous claim implies that
\[ \frac{u}{r} \phi_2(z) \to (Q-\gamma)u \quad\text{as}\quad r \to \infty.\]  This completes the proof.
\end{proof}

\section{Exploring a quantum sphere with $\SLE_{\kappa'}(\kappa'-6)$}
\label{sec::extensions}

We will now establish some generalizations of our earlier results to the case of $\gamma \in (\sqrt{2},2)$.

\begin{theorem}
\label{thm::sphere_sle_kp_kp_minus_6}
Let $(\cyl,h,-\infty,+\infty)$ be a unit area quantum sphere with $\gamma \in (\sqrt{2},2)$.  Let $\wt{\eta}'$ be a whole-plane $\SLE_{\kappa'}(\kappa'-6)$ process in $\cyl$ from $-\infty$ to $+\infty$.  Let $\CU_1$ be the collection of components of $\cyl \setminus \wt{\eta}'$ which are cut off by $\wt{\eta}'$ from $+\infty$ by $\wt{\eta}'$ (viewed as a path in the universal cover of $\cyl \setminus \{+\infty\}$) and let $\CU_2$ be the remaining components of $\cyl \setminus \wt{\eta}'$ which are cut off by $\wt{\eta}'$ from $+\infty$.  (The elements of $\CU_1$ are marked by the first, equivalently last, boundary point visited by $\wt{\eta}'$ and the elements of $\CU_2$ are doubly marked by the first and last boundary point visited by $\wt{\eta}'$.)  Conditional on their quantum boundary lengths and areas, the elements of $\CU_1,\CU_2$ are conditionally independent.  The elements of the former are quantum disks and the elements of the latter are surfaces sampled from the infinite measure on quantum surfaces used to construct a weight $2-\gamma^2/2$ quantum wedge with the given conditioning.

Moreover, $(\cyl,h,-\infty,+\infty)$ and $\wt{\eta}'$ are a.s.\ determined up to horizontal translation and a global rotation of $\cyl$ about $\pm \infty$ by the ordered sequence of oriented, marked components cut out by $\wt{\eta}'$ viewed as quantum surfaces.
\end{theorem}

We note that in Theorem~\ref{thm::sphere_sle_kp_kp_minus_6} we do not describe the evolution of the boundary length of the component containing $+\infty$.  However, the proof of Theorem~\ref{thm::sphere_sle_kp_kp_minus_6} follows from the same argument used to prove Theorem~\ref{thm::pure_sphere_equivalent_constructions}.

We now state the analog of Proposition~\ref{prop::unexplored_region} in the general case of $\gamma \in (\sqrt{2},2)$.

\begin{proposition}
\label{prop::unexplored_general_sphere}
Suppose that $(\cyl,h,-\infty,+\infty)$ is a unit area quantum sphere with $\gamma \in (\sqrt{2},2)$ and that~$\wt{\eta}'$ is a whole-plane $\SLE_{\kappa'}(\kappa'-6)$ process in~$\cyl$ from $-\infty$ to $+\infty$ which is sampled independently of~$h$ and then reparameterized by quantum natural time.  Let~$\tau$ be a stopping time for the filtration generated by the bubbles cut out by $\wt{\eta}'|_{[0,t]}$ such that a.s.\ the boundary of the component $U$ of $\cyl \setminus \wt{\eta}'([0,\tau])$ containing $+\infty$ is contained in one side of $\wt{\eta}'$.  Conditional on its quantum boundary length and area, the quantum surface $(U,h)$ has the law of a quantum disk.  Moreover, $\wt{\eta}'(\tau)$ is uniformly distributed from the quantum boundary measure on $\partial U$.
\end{proposition}
\begin{proof}
This follows from the same argument used to prove Proposition~\ref{prop::unexplored_region_qc_inf_to_0}.
\end{proof}

By combining Theorem~\ref{thm::sphere_sle_kp_kp_minus_6} and Proposition~\ref{prop::unexplored_general_sphere}, we can determine the law of the components which are cut off when we perform an exploration by a radial $\SLE_{\kappa'}(\kappa'-6)$ on a disk.

\begin{theorem}
\label{thm::disk_explore2}
Suppose that $(\cyl_+,h,+\infty)$ is a quantum disk (where $+\infty$ is uniform from the area measure given the quantum surface $(\cyl_+,h)$), $x \in \partial \cyl_+$ is chosen uniformly from the quantum boundary measure on $\partial \cyl_+$, and $\wt{\eta}'$ is a radial $\SLE_{\kappa'}(\kappa'-6)$ starting from $x$ and targeted at $+\infty$ which is sampled conditionally independently of $h$ given $x$.   Let $\CU_1$ be the collection of components of $\cyl_+ \setminus \wt{\eta}'$ which are cut off by $\wt{\eta}'$ from $+\infty$ by $\wt{\eta}'$ viewed as a path in the universal cover of $\cyl_+ \setminus \{+\infty\}$ and let $\CU_2$ be the remaining components of $\cyl_+ \setminus \wt{\eta}'$ which are cut off by $\wt{\eta}'$ from $+\infty$.  (The elements of $\CU_1$ are marked by the first, equivalently last, boundary point visited by $\wt{\eta}'$ and the elements of $\CU_2$ are doubly marked by the first and last boundary point visited by $\wt{\eta}'$.)  Conditional on their quantum boundary lengths and areas, the elements of $\CU_1,\CU_2$ are conditionally independent.  The elements of the former are quantum disks and the elements of the latter are surfaces sampled from the infinite measure on quantum surfaces used to construct a weight $2-\gamma^2/2$ quantum wedge with the given conditioning.

Moreover, $(\cyl_+,h,+\infty)$ and $\wt{\eta}'$ are a.s.\ determined up to a global rotation of $\cyl_+$ about $+ \infty$ by the ordered sequence of oriented, marked components cut out by $\wt{\eta}'$ viewed as quantum surfaces.
\end{theorem}

\bibliographystyle{hmralphaabbrv}
\addcontentsline{toc}{section}{References}
\bibliography{sle_kappa_rho}

\bigskip

\filbreak
\begingroup
\small
\parindent=0pt

\bigskip
\vtop{
\hsize=5.3in
Department of Mathematics\\
Massachusetts Institute of Technology\\
Cambridge, MA, USA } \endgroup \filbreak

\end{document}